\documentclass[reqno,12pt]{amsart}
 \usepackage{amsmath}
 \usepackage{amsthm}
 \usepackage{amssymb}
 \usepackage{latexsym,longtable}
 \usepackage{url}
 \usepackage{graphicx}
 \usepackage{multicol}
 \usepackage{mathrsfs}
 \usepackage{young}
 \usepackage[vcentermath,enableskew,stdtext]{youngtab}
 \usepackage[left=1.04in,right=1.04in,top=1.05in,bottom=1.12in]{geometry}
 \usepackage[all]{xy}
    \SelectTips{cm}{10}     
    \everyxy={<2.5em,0em>:} 
 \usepackage{fancyhdr}      
 \usepackage{multicol}
 \usepackage{multirow}
 \usepackage{enumitem}
 \usepackage{stmaryrd}
 \usepackage{etex}
 \usepackage{tikz}
 \usepackage{float}
 \usepackage{array}
 \usepackage{colortbl}
 \usepackage{mathtools}
\usepackage[centertableaux]{ytableau}
 \usepackage{bm}
 \restylefloat{figure}
 \usetikzlibrary{shapes}
 \usetikzlibrary{arrows}
 \usetikzlibrary{positioning}
 \usetikzlibrary{decorations.pathmorphing}
 \usetikzlibrary{decorations.markings}
\usepackage{color}
\definecolor{gray}{rgb}{.5, .5, .5}

\usepackage{amsfonts,epsfig,color}
\makeatletter
\newcommand*{\centerfloat}{%
  \parindent \z@
  \leftskip \z@ \@plus 1fil \@minus \textwidth
  \rightskip\leftskip
  \parfillskip \z@skip}
\makeatother
\newcounter{ctr}
\theoremstyle{plain}
\newtheorem{theorem}{Theorem}[section]
\newtheorem{lemma}[theorem]{Lemma}
\newtheorem{corollary}[theorem]{Corollary}
\newtheorem{proposition}[theorem]{Proposition}
\newtheorem{conjecture}[theorem]{Conjecture}

\theoremstyle{definition}
\newtheorem{definition}[theorem]{Definition}

\newtheorem{remark}[theorem]{Remark}
\newtheorem{example}[theorem]{Example}

\newcommand{\ignore}[1]{}

\newcommand\drawcell[2]{
\draw (0+#1,0+#2)--(1+#1,0+#2)--(1+#1,1+#2)--(0+#1,1+#2)--(0+#1,0+#2);
}

\renewcommand{\a}{\ensuremath{\mathfrak{a}}}

\newcommand{\A}{{\ensuremath{\mathcal{A}}}}

\newcommand{\R}{\ensuremath{\mathscr{R}}}

\newcommand{\sgn}{\text{\rm sgn}}

\newcommand{\U}{\mathcal{U}}

\newcommand{\ZZ}{\ensuremath{\mathbb{Z}}}

\newcommand{\be}{\begin{equation}}
\newcommand{\ee}{\end{equation}}

\renewcommand{\S}{\ensuremath{\mathcal{S}}}

\newcommand{\creading}{\text{\rm colword}}
\newcommand{\sh}{\text{\rm sh}}

\newcommand{\ngb}{\ensuremath{\tiny \hspace{-.16mm}-\hspace{-.5mm}}}

\def\Tiny{\fontsize{6pt}{6pt}\selectfont}

\newcommand{\overbar}[1]{\mkern 1.1mu\overline{\mkern-1.1mu#1\mkern-1.1mu}\mkern 1.1mu}
\renewcommand{\underbar}[1]{\mkern 1.2mu\underline{\mkern-1.2mu#1\mkern-1.2mu}\mkern 1.2mu}
\newcommand{\crc}[1]{\ensuremath{\overbar{#1}}\vphantom{\underline{\overline{#1}}}}
\newcommand{\crcempty}{{\crc{\phantom{{}_\circ}}}}

\newcommand{\convert}[2]{\ensuremath{{#2}(#1)}}

\newcommand{\stand}{\ensuremath{\text{\rm st}}} 

\newcommand{\plainr}{\ensuremath{{\text{\rm brgt}}}}

\newcommand\mybox[1]{
\vcenter{
\let\\=\cr
\baselineskip=-16000pt \lineskiplimit=16000pt \lineskip=0pt
\halign{&\boxcell{##}\cr\vline#1\vline\crcr}}}

\newcommand{\boxcell}[1]{{%
\unitlength=\cellsizeCol
\begin{picture}(1,1)
\put(0,0){\makebox(1,1){$#1$}}
\put(0,0){\line(1,0){1}}
\put(0,1){\line(1,0){1}}
\end{picture}%
}}

\newcommand\pad[1]{
\vtop{
\let\\=\cr
\baselineskip=-16000pt
\lineskiplimit=16000pt
\lineskip=0pt
\halign{& \inviscell{##} \cr #1 \crcr} }
\hspace{-.73ex}}

\newcommand{\inviscell}[1]{{%
\unitlength=\cellsizeCol
\begin{picture}(1,1)
\put(0,0){\makebox(1,1){$#1$}}
\end{picture}%
}}

\newlength{\cellsize}
\newcommand\tableau[1]{
\vcenter{
\let\\=\cr
\baselineskip=-16000pt \lineskiplimit=16000pt \lineskip=0pt
\halign{&\tableaucell{##}\cr#1\crcr}}}

\newcommand{\tableaucell}[1]{{%
\def \arg{#1}\def \void{}%
\ifx \void \arg
\vbox to \cellsize{\vfil \hrule width \cellsize height 0pt}%
\else \unitlength=\cellsize
\begin{picture}(1,1)
\put(0,0){\makebox(1,1){$#1\vphantom{\crc{#1}}$}}
\put(0,0){\line(1,0){1}}
\put(0,1){\line(1,0){1}}
\put(0,0){\line(0,1){1}}
\put(1,0){\line(0,1){1}}
\end{picture}%
\fi}}

\newcommand\boldtableau[1]{
\vcenter{
\let\\=\cr
\baselineskip=-16000pt \lineskiplimit=16000pt \lineskip=0pt
\halign{&\boldtableaucell{##}\cr#1\crcr}}}

\newcommand{\boldtableaucell}[1]{{%
\def \arg{#1}\def \void{}%
\ifx \void \arg
\vbox to \cellsize{\vfil \hrule width \cellsize height 0pt}%
\else \unitlength=\cellsize
\begin{picture}(1,1)
\put(0,0){\makebox(1,1){$\mathbf{#1\vphantom{\crc{#1}}}$}}
\put(0,0){\line(1,0){1}}
\put(0,1){\line(1,0){1}}
\put(0,0){\line(0,1){1}}
\put(1,0){\line(0,1){1}}
\end{picture}%
\fi}}

\newlength{\colskip}
\newlength{\dwidth}

\newcommand{\partition}[1]{{\setlength{\cellsize}{1ex} \tiny \tableau{#1}}}

\newcommand{\lerow}{\leq_{\text{row\,}}}
\newcommand{\lecol}{\leq_{\text{col\,}}}

\newcommand{\gecol}{\geq_{\text{col\,}}}
\newcommand{\gedotcol}{\mathbin{\underbar{\gtrdot}_{\text{col\,}}}}
\newcommand{\ledotrow}{\mathbin{\underbar{\lessdot}_{\text{row\,}}}}

\newcommand{\Ikron}{\ensuremath{{{I_\text{\rm Kron}}}}} 
\newcommand{\Ikronknuth}{\ensuremath{{{I_\text{\rm Kron-K}}}}} 

\newcommand{\Iplacord}[1]{\ensuremath{{{I_\text{\rm plac}^{#1}}}}}

\newcommand{\Irkstp}[1]{\ensuremath{{{I_{\text{\rm KR}, #1}^\stand}}}}

\newcommand{\Des}{\ensuremath{{\text{\rm Des}}}}

\newcommand{\e}{\mathsf}
\newcommand{\noe}{} 

\newcommand{\sqread}{\text{\rm arwread}}
\newcommand{\sqreadLLT}{\text{\rm sqread}}

\newcommand{\mydownarrow}{{\hspace{0mm} \tiny \downharpoonleft}}
\newcommand{\myDownarrow}{{\hspace{0mm} \tiny \downarrow}}

\newcommand{\CT}{\text{\rm CT}}
\newcommand{\CYW}{\text{\rm CYW}}
\newcommand{\CYT}{\text{\rm CYT}}
\newcommand{\SYT}{\text{\rm SYT}}

\newcommand{\RCT}{\text{\rm RCT}}
\newcommand{\Tab}{\text{\rm Tab}}
\newcommand{\SSYT}{\text{\rm SSYT}}

\newcommand{\ver}{\text{\rm Vert}}

\newcommand{\nearr}{{\scalebox{.75}{$\nearrow$}}}

\newcommand{\searr}{{\scalebox{.75}{$\searrow$}}}
\newcommand{\nwarr}{{\scalebox{.75}{$\nwarrow$}}}
\newcommand{\nearrsub}{{\scalebox{.5}{$\nearrow$}}}

\newcommand{\searrsub}{{\scalebox{.5}{$\searrow$}}}

\newcommand{\Jnot}[1]{\ensuremath{{\hspace{1pt}\text{\rm:}~\raisebox{1pt}{$\e{#1}$}~\text{\rm:}\hspace{1pt}}}}
\newcommand{\Jnotb}[2]{\ensuremath{{{}_{#1}{}^{\e{#2}}}}}

\title{Kronecker coefficients and noncommutative super Schur functions}

\begin{document}

\author{Jonah Blasiak}
\address{Department of Mathematics, Drexel University, Philadelphia, PA 19104}
\email{jblasiak@gmail.com}
\thanks{J. Blasiak was supported by NSF Grant DMS-14071174.}
\author{Ricky Ini Liu}
\address{Department of Mathematics, North Carolina State University, Raleigh, NC 27695}
\email{riliu@ncsu.edu}

\begin{abstract}
The theory of noncommutative Schur functions can be used to obtain positive combinatorial formulae for the Schur expansion of various classes of symmetric functions, as shown by Fomin and Greene \cite{FG}.
We develop a theory of noncommutative super Schur functions and use it to prove a positive combinatorial rule for the Kronecker coefficients $g_{\lambda\mu\nu}$ where one of the partitions is a hook, recovering previous results of the two authors \cite{BHook, Ricky}.  This method also gives a precise connection between this rule and a heuristic for Kronecker coefficients first investigated by Lascoux \cite{Lascoux}.
\end{abstract}
\maketitle



\section{Introduction}
Let $M_\lambda$ be the irreducible representation of the symmetric group $\S_n$ corresponding to the partition $\lambda$. Given three partitions $\lambda$, $\mu$, and $\nu$ of  $n$,
the \emph{Kronecker coefficient $g_{\lambda \mu \nu}$} is the multiplicity of
$M_\nu$ in the tensor product $M_\lambda \otimes M_\mu$.  A longstanding open problem in algebraic combinatorics, called the \emph{Kronecker problem}, is to find a positive combinatorial formula for these coefficients.
See \cite{BO,BHook, BMSGCT4,  BOR,HayashiKronecker,  Ricky, Lascoux,Remmel,  RemmelHookTwoRow,RW, Rosas} for some known special cases.

Our story begins with the work of Lascoux \cite{Lascoux}, wherein he gave a formula for the Kronecker coefficients
 $g_{\lambda \mu \nu}$ when two of the partitions are hooks by considering products of permutations in certain Knuth equivalence classes. 
Though this rule no longer holds outside the hook-hook case, it seems to approximate Kronecker coefficients amazingly well for \emph{any three partitions}
and therefore gives a useful heuristic.

Several years ago, the first author \cite{BHook} gave a rule for
Kronecker coefficients when one of the partitions is a hook.
This rule was discovered using Lascoux's heuristic,
but it was left as an open problem to give a precise statement relating it to the heuristic.  
Recently, the second author \cite{Ricky} gave a simplified description and proof of this rule.



We develop a theory of noncommutative super Schur functions based on work of Fomin-Greene \cite{FG} and the first author \cite{BLamLLT}.
Using this we
\begin{itemize}
\item reprove and strengthen the rule from \cite{Ricky},
\item establish a precise connection between this rule and the Lascoux heuristic, and
\item uncover a surprising parallel between this rule and combinatorics underlying
transformed Macdonald polynomials indexed by a 3-column shape, as described in \cite{Haglund, BLamLLT}.
\end{itemize}

\section{Main results}\label{s main results}

In this section, we state our main theorem on noncommutative super Schur functions (Theorem~\ref{t J intro}),
and show how it can be used to recover the rule from \cite{Ricky} for Kronecker coefficients where one of the partitions is a hook. Proofs will be deferred to later sections.

\subsection{Colored words and the algebra $\U$}
Let $\A_\varnothing = \{1,2, \dots, N\}$ denote the alphabet of unbarred letters and $\A_\crcempty = \{\crc{1},\crc{2}, \dots, \crc{N}\}$ the alphabet of barred letters.
A \emph{colored word} is a word in the total alphabet $\A = \A_\varnothing \sqcup \A_\crcempty$.

We will consider total orders $\lessdot$ on $\A$ such that $1 \lessdot 2 \lessdot \cdots \lessdot N$ and $\crc{1} \lessdot \crc{2} \lessdot \cdots \lessdot \crc{N}$; we call such orders \emph{shuffle orders}.
Two shuffle orders we will work with frequently are
\[
\begin{tabular}{lll}
the \emph{natural order} $<$ given by &$1 < \crc{1} < 2 < \crc{2} < \cdots < N < \crc{N},$& and\\[1.4mm]
the \emph{big bar order} $\prec$ given by & $1 \prec 2 \prec \cdots \prec N \prec \crc{1} \prec \crc{2} \prec \cdots \prec \crc{N}.$&
\end{tabular}
\]

Let  $\U$ be the free associative  $\ZZ$-algebra in the noncommuting variables $u_x$,  $x \in \A$.
Equivalently,  $\U$ is the tensor algebra of  $\ZZ \A$, the  $\ZZ$-module freely spanned by the elements of  $\A$.
We identify the monomials of $\U$ with colored words and frequently write  $\e{x}$ for the variable $u_x$ and $\e{w} = \e{w_1\cdots w_t} = u_{w_1}\cdots u_{w_t}$ for a colored word/monomial.

For the natural order  $<$, it is useful to have a notation for ``going down by one.''  Accordingly, for any $\noe{x} \in \A$, define
\begin{align}\label{e myDownarrow}
\noe{x} \myDownarrow =
\begin{cases}
\noe{\crc{a-1}} &\text{ if $\noe{x = a}$,  $a \in \{2,\ldots, N\}$}, \\
\noe{a} &\text{ if $\noe{x = \crc{a}}$, $a \in \{1,\ldots,N\}$}.
\end{cases}
\end{align}


\subsection{Quotients of  $\U$}
\label{ss quotients of U}
Often we will be interested in performing computations in a quotient of $\U$. One important such quotient is the
\emph{$<$-colored plactic algebra}, denoted $\U/\Iplacord{<}$, which is the quotient of  $\U$ by the relations
\begin{alignat}{3}
&\e{xzy} = \e{zxy} \qquad &&\text{for } x,y,z \in \A,\, x< y< z,\label{e plac nat rel knuth1} \\
&\e{yxz} = \e{yzx} \qquad &&\text{for } x,y,z \in \A,\, x< y< z,\label{e plac nat rel knuth2} \\
&\e{yyx} = \e{yxy} \qquad &&\text{for } x \in \A,\, y \in \A_{\varnothing},\, x< y, \label{e plac nat rel knuth3} \\
&\e{zyy} = \e{yzy} \qquad &&\text{for } z \in \A,\, y \in \A_{\varnothing},\, y< z, \label{e plac nat rel knuth3b} \\
&\e{yyz} = \e{yzy} \qquad &&\text{for } z \in \A,\, y \in \A_{\crcempty},\, y< z, \label{e plac nat rel knuth4}\\
&\e{xyy} = \e{yxy} \qquad &&\text{for } x \in \A,\, y \in \A_{\crcempty},\, x< y; \label{e plac nat rel knuth4b}
\end{alignat}
let $\Iplacord{<}$ denote the corresponding two-sided ideal of $\U$.
The $<$-colored plactic algebra behaves much like the ordinary plactic algebra.

For applications to Kronecker coefficients, we need the following variant of the  $<$-colored plactic algebra.
Let $\U/\Ikron$ denote the quotient of $\U$ by the relations
\eqref{e plac nat rel knuth3}, \eqref{e plac nat rel knuth3b}, \eqref{e plac nat rel knuth4}, \eqref{e plac nat rel knuth4b}, and
\begin{alignat}{3}
\e{(xz-zx )y} &= \e{y(xz-zx)} \qquad &&\text{for }  x,y,z \in \A,\, x  = y \myDownarrow = z \myDownarrow\myDownarrow, \label{e kron rel rotate12} \\
\e{xz} &= \e{zx} \qquad &&\text{for }  x,z \in \A,\, x<z \myDownarrow \myDownarrow; \label{e kron rel far commute}
\end{alignat}
let $\Ikron$ denote the corresponding two-sided ideal of $\U$.
We refer to \eqref{e kron rel far commute} as the \emph{far commutation relations}.

\subsection{Main theorem} 
\label{ss main theorem}
For any two letters $x, y \in \A$, write $y \gecol x$ to mean either $y > x$, or $y$ and $x$ are equal barred letters.
The \emph{noncommutative super elementary symmetric functions} are defined by
\[
e_k(\mathbf{u})=\sum_{\substack{z_1 \gecol z_2 \gecol \cdots \gecol z_k \\ z_1,\dots,z_k \in \A}} u_{z_1} u_{z_2} \cdots u_{z_k} \ \in \U
\]
for any positive integer $k$; set $e_0(\mathbf{u})=1$ and $e_k(\mathbf{u}) = 0$ for $k<0$.

We will prove that $e_k(\mathbf u)$ and $e_l(\mathbf u)$ commute for all $k$ and $l$ in both $\U/\Ikron$ and $\U/\Iplacord{<}$ in Propositions \ref{p es commute kron} and \ref{p es commute plac}.
Similar results are proved in \cite{LS,FG,LamRibbon,BD0graph,Kirillov-Notes}.

Using these elementary functions, we can define the noncommutative super Schur functions as follows.

\begin{definition}
\label{d noncommutative Schur functions}
Let $\nu=(\nu_1,\nu_2,\dots)$ be a partition.
Let $\nu'$ be the conjugate partition, which has $t=\nu_1$ parts.
The \emph{noncommutative super Schur function}
$\mathfrak{J}_\nu(\mathbf{u})$ is given by the following
noncommutative version of the Kostka-Naegelsbach/Jacobi-Trudi formula:
\[
\mathfrak{J}_\nu(\mathbf{u}) = \sum_{\pi\in \S_{t}}
\sgn(\pi) \, e_{\nu'_1+\pi(1)-1}(\mathbf{u})
e_{\nu'_2+\pi(2)-2}(\mathbf{u}) \cdots e_{\nu'_{t}+\pi(t)-t}(\mathbf{u}) \ \in \U.
\]
\end{definition}

As explained in \cite{FG} and as we will see below, establishing monomial positivity of $\mathfrak{J}_\nu(\mathbf{u})$ in various quotients of $\U$
has important consequences for proving manifestly positive combinatorial formulae for structure coefficients.

It follows easily from \cite[Lemma 3.2]{FG} that  $\mathfrak{J}_\nu(\mathbf{u})$ is equal to a positive sum of monomials in  $\U/\Iplacord{<}$ (Theorem \ref{t J plac}).
We conjecture that this can be strengthened to monomial positivity of $\mathfrak{J}_\nu(\mathbf{u})$  in an algebra  $\U/\Ikronknuth$ (see \textsection\ref{ss switchboards})
 which has both $\U/\Iplacord{<}$ and $\U/\Ikron$
as quotients. 
Our main result is a weaker version of this,
a proof that  $\mathfrak{J}_\nu(\mathbf{u})$  is monomial positive in   $\U/ \Ikron$.
(We believe that $\U/\Ikronknuth$ is the natural algebra for the applications described below, but we work with
$\U/\Ikron$ instead since this makes statements easier to prove.) 
 We now give the precise statement of our main result.

A \emph{$<$-colored tableau} is a tableau with entries in $\A$ such that each row and column is weakly increasing with respect to the natural order $<$, while the unbarred letters in each column and the barred letters in each row are strictly increasing. Let $\CT_{\nu}^<$ denote the set of  $<$-colored tableaux of shape $\nu$.

\begin{definition}\label{d sqread}
For a  $<$-colored  tableau $T$, define the colored word $\sqread(T)$ as follows: let $D^1,D^2,\dots,D^k$ be the diagonals of $T$, starting from the southwest.
Let $\e{w^i}$ be the result of reading the unbarred entries of $D^i$ in the direction $\nwarr$, followed by the barred entries of $D^i$ in the direction $\searr$.
Set $\sqread(T)=\e{w^1}\e{w^2}\cdots \e{w^k}$.
\end{definition}

The word $\sqread(T)$ is a particular choice of an \emph{arrow respecting reading word} of  $T$, which will be defined in \textsection\ref{ss Arrow respecting reading words}.

For example, for  $\nu = (5,4,4)$,
\begin{align*}
T  &= { \tiny \tableau{1 & 1 & \crc{3} & \crc{4} & 6\\  \crc{2} & 3 & 4 & \crc{4} \\ 3 & \crc{3} & \crc{4} & 5}}
 \in \CT_\nu^<, \\[2mm]
\sqread(T) &= \e{3\, \crc{2}\, \crc{3}\, 3\, 1\, \crc{4}\, 5\, 4\, 1\, \crc{3}\, \crc{4}\, \crc{4}\, 6}.
\end{align*}

We now state our main theorem. Section~\ref{s proof of theorem} is devoted to its proof.
\begin{theorem}\label{t J intro}
In the algebra  $\U/\Ikron$, the noncommutative super Schur function  $\mathfrak{J}_{\nu}(\mathbf{u})$ is equal to the following positive sum of monomials:
\begin{align}
\mathfrak{J}_{\nu}(\mathbf{u}) =
\sum_{ T \in \CT_{\nu}^< }\sqread(T)
 \qquad \text{in \hspace{-.6mm} $\U/\Ikron$}.
\end{align}
\end{theorem}

In \cite{BLamLLT}, the first author used the theory of noncommutative Schur functions to prove a positive combinatorial formula for
the Schur
expansion of LLT polynomials indexed by 3-tuples of skew shapes and  thereby settled a conjecture of
Haglund \cite{Haglund} on transformed Macdonald polynomials indexed by 3-column shapes.
The noncommutative Schur function computation required for this work (\cite[Theorem 1.1]{BLamLLT}) is quite similar to Theorem \ref{t J intro}
(though it seems neither result can be obtained from the other; see \textsection\ref{ss comparison with results}).
It is quite surprising that LLT polynomials for 3-tuples of skew shapes and
Kronecker coefficients for one hook shape have such similar underlying combinatorics.

%
\subsection{Applications}
The machinery of noncommutative Schur functions as developed in \cite{FG,BF}
can be used to study symmetric functions in the usual
sense, i.e., formal power series in infinitely many (commuting) variables
$\mathbf{x}=(x_1,x_2,\dots)$ which are symmetric under permutations of these variables.

Let  $\e{w} = \e{w_1 \cdots w_t}$ be a colored word.
For any shuffle order $\lessdot$ on $\A$, define the \emph{$\lessdot$-descent set of $\e{w}$}, denoted $\Des_\lessdot(\e{w})$,
by
\[\Des_\lessdot(\e{w}) := \big\{i \in [t-1] \mid \e{w_i} \gtrdot \e{w_{i+1}} \text{ or ($\e{w_i}$ and  $\e{w_{i+1}}$ are equal barred letters)}\big\}.\]
The associated \emph{fundamental quasisymmetric function} (Gessel \cite{GesselPPartition}) is given by
\[Q_{\Des_\lessdot(\e{w})}(\mathbf{x}) = \sum_{\substack{1 \le i_1 \le \, \cdots \, \le i_t\\j \in \Des_\lessdot(\e{w}) \implies i_j < i_{j+1} }} x_{i_1}\cdots x_{i_t}.\]

\begin{definition}
\label{d F gamma}
For any $\gamma = \sum_\e{w} \gamma_\e{w} \e{w} \in \U$ (where the sum is over colored words $\e{w}$, and $\gamma_\e{w} \in \ZZ$), define
\begin{equation*}
F^\lessdot_\gamma(\mathbf{x}) = \sum_{\e{w}} \gamma_\e{w} \, Q_{\Des_\lessdot(\e{w})}(\mathbf{x}) \in \ZZ[[x_1, x_2, \dots]],
\end{equation*}
and for a set of colored words $W$, define
$F^\lessdot_W(\mathbf{x})=F^\lessdot_\gamma(\mathbf{x})$ where $\gamma=\sum_{\e{w} \in W} \e{w}$.
\end{definition}

Let $\langle\cdot ,\cdot  \rangle$ be the symmetric bilinear form on $\U$ for which the monomials (colored words) form an orthonormal basis.
Note that any element of $\U/I$ has a well-defined pairing with any element of $I^\perp$ for any two-sided ideal  $I$ of  $\U$.

The next result is a straightforward adaptation of the theory of noncommutative Schur functions \cite{FG,BF} to the super setting.
It also requires the fact mentioned above that the noncommutative super elementary symmetric functions commute in $\U/\Ikron$.
(A full proof is given by Theorem \ref{t basics full} and Proposition \ref{p es commute kron}.)
\begin{theorem} \label{t intro basics}
For any $\gamma \in (\Ikron)^\perp$, the function  $F^<_\gamma(\mathbf{x})$ is symmetric and
\[
F^<_\gamma(\mathbf{x})
 = \sum_{\nu} s_\nu(\mathbf x) \langle \mathfrak{J}_{\nu}(\mathbf{u}), \gamma \rangle. \]
\end{theorem}

Theorems \ref{t intro basics} and \ref{t J intro} then immediately yield the following.
\begin{corollary}\label{c Ikron perp}
For any set of colored words $W$ such that $\sum_{\e{w} \in W} \e{w} \in (\Ikron)^\perp$,
\[
\Big(\text{the coefficient of  $s_\nu(\mathbf{x})$ in  $F^<_W(\mathbf{x})$}\Big)
=
\big|\big\{T \in \CT_\nu^< \mid \sqread(T) \in W \big\}\big|.
\]
\end{corollary}

We will soon see how Corollary~\ref{c Ikron perp} can be used to prove a positive combinatorial formula for the Kronecker coefficients indexed by one hook shape and two arbitrary
shapes, recovering results of \cite{Ricky, BHook}.  It is more powerful, however, than this main application.
In \textsection\ref{ss switchboards}, we give a way of visualizing the implications of Corollary~\ref{c Ikron perp} and give examples to better indicate its full strength.

\subsection{Relation to Kronecker coefficients}
\label{ss kronecker}
To apply Corollary~\ref{c Ikron perp} to the Kronecker problem,
we define a set $W$ of colored Yamanouchi words such that the coefficient of $s_\nu(\mathbf x)$ in $F^<_W(\mathbf x)$ is given in terms of Kronecker coefficients. The proofs of the results in this subsection are deferred to \textsection\ref{s some proofs}.

\begin{definition}[Colored Yamanouchi words]
Let  $\e{w}$ be a colored word.
Define $\e{w}^\plainr$ to be the ordinary word formed from $\e{w}$ by shuffling the barred letters to the right, reversing this subword of barred letters and removing their bars.
We say that $\e{w}$ is \emph{Yamanouchi} of content $\lambda$ if $\e{w}^\plainr$ is Yamanouchi of content $\lambda$.
Define $\CYW_{\lambda, d}$ to be the set of colored Yamanouchi words of content $\lambda$ having exactly  $d$ barred
letters.\footnote{Warning: the colored Yamanouchi words defined here are not the same as those defined in \cite{BHook}.}
\end{definition}

For example, if $\e{w} = \e{2 \crc{1} \crc{2} 1 \crc{3} \crc{1} 2 1}$, then  $\e{w}^\plainr =  2 1 2 1 1 3 2 1$,  and these are Yamanouchi of content  $(4,3,1)$.
See Figure \ref{f CYW 32}.

For  $0 \le d \le n-1$, let $\mu(d)$ denote the hook partition $(n-d,1^d)$.
The following is in some sense well known (see \textsection\ref{ss proof1} for a proof).

\begin{proposition} \label{p intro sum of kron}
For any partitions  $\lambda, \nu$ of  $n$ and  $d \leq n$,
\begin{align*}
g_{\lambda\, \mu(d)\, \nu} + g_{\lambda\, \mu(d-1)\, \nu} = \Big( \text{the coefficient of $s_\nu(\mathbf{x})$ in } F^\prec_{\CYW_{\lambda,d}}(\mathbf{x})\Big).
\end{align*}
(By convention, set $g_{\lambda \, \mu(n) \, \nu}  = g_{\lambda \, \mu(-1) \, \nu} = 0$.)
\end{proposition}

We use a trick called word conversion to convert from the big bar order $\prec$ to the natural order $<$ (see \textsection\ref{ss proof2}).
\begin{proposition}\label{p intro word conversion QDes}
For any partition $\lambda$ of  $n$ and  $d \leq n$, $F^\prec_{\CYW_{\lambda,d}}(\mathbf{x}) = F^<_{\CYW_{\lambda,d}}(\mathbf{x})$.
\end{proposition}

By showing that $\sum_{\e{w} \in \CYW_{\lambda, d}} \e{w} \in (\Ikron)^\perp$ (see \textsection\ref{ss proof3}) and applying Propositions \ref{p intro sum of kron} and \ref{p intro word conversion QDes} and Corollary~\ref{c Ikron perp}, we obtain the following result, first proved in \cite{Ricky} using the conversion operation on colored tableaux (see Lemma 3.1 and Remark 3.3 of \cite{Ricky}).

\begin{corollary} \label{c intro main}
For any partitions  $\lambda, \nu$ of  $n$ and  $d \leq n$,
\begin{align*}
g_{\lambda\, \mu(d)\, \nu} + g_{\lambda\, \mu(d-1)\, \nu}
& = \Big(\text{the coefficient of  $s_\nu(\mathbf{x})$ in  $F^<_{\CYW_{\lambda,d}}(\mathbf{x})$}\Big)\\
& = \big|\big\{T \in \CT_\nu^< \mid \sqread(T) \in \CYW_{\lambda,d}\big\}\big|.
\end{align*}
\end{corollary}

This result demonstrates the importance of $\Ikron$ since it is not always the case that $\sum_{\e{w} \in \CYW_{\lambda, d}} \e{w} \in (\Iplacord{<})^\perp$. For example, $\sum_{\e{w} \in \CYW_{\lambda, d}} \e{w} \notin (\Iplacord{<})^\perp$ for $\lambda = (2,2)$, $d = 2$ since $\e{\crc{1}\crc{2}21} \in \CYW_{\lambda, d}$, but $\e{ \crc{2}\crc{1}21 }\notin \CYW_{\lambda, d}$, whereas
$\e{\crc{1}\crc{2}21 \equiv \crc{2}\crc{1}21} \ \bmod{\Iplacord{<}}.$

Corollary~\ref{c intro main} easily implies an explicit combinatorial formula for $g_{\lambda \mu(d) \nu}$. Partition $\CYW_{\lambda,d}$ into sets $\CYW_{\lambda, d}^-$ and $\CYW_{\lambda, d}^+$ consisting of the words ending in an unbarred letter or barred letter, respectively. Since
there is a bijection from $\CYW_{\lambda, d+1}^+$ to $\CYW_{\lambda, d}^-$ given by removing the bar from the last letter, and this bijection respects the set of words of the form $\sqread(T)$, this partition separates the two Kronecker coefficients appearing in Corollary~\ref{c intro main}, giving the following result. (We omit the complete proof here; see \cite[\textsection3.3]{BHook}, \cite[\textsection4]{Ricky} for more details.)

\begin{corollary}[\cite{Ricky}]
\label{c kronecker ricky}
For any partitions $\lambda, \nu$ of $n$ and $d \leq n-1$,
\[g_{\lambda\mu(d)\nu} = \big|\big\{T \in \CT_\nu^< \mid \sqread(T) \in \CYW_{\lambda,d}^-\big\}\big|.\]
\end{corollary}

\subsection{The Lascoux heuristic}

We now describe a heuristic for computing Kronecker coefficients first investigated by Lascoux in \cite{Lascoux} and relate it to Corollary~\ref{c kronecker ricky}.

For any partition $\lambda$, let $Z_\lambda^\stand$ be the \emph{(standardized) superstandard tableau} of shape $\lambda$ whose boxes are labeled in order across rows from top to bottom. Let $\Gamma_\lambda$ be the set of permutations whose insertion tableau is $Z_\lambda^\stand$. For any two partitions $\lambda$ and $\mu$, consider the multiset of permutations
\be \label{e lambda circ mu}
\Gamma_\lambda \circ \Gamma_\mu = \big\{ u \circ v \mid u \in \Gamma_\lambda, v \in \Gamma_\mu \big\},
\ee
where $\circ$ denotes ordinary composition of permutations.

For example, take
\[
\begin{array}{lll}
\lambda = (3,1), &Z_\lambda^\stand = \tiny \tableau{1&2&3\\4},&\Gamma_\lambda = \{4123,1423,1243\};\\[5mm]
\mu = (2,1,1), &Z_\mu^\stand = \tiny\tableau{1&2\\3\\4}, &\Gamma_\mu = \{4312, 4132, 1432\}.
\end{array}
\]
Then $\Gamma_\lambda \circ \Gamma_\mu$ consists of the nine products in the multiplication table:
\[
\begin{array}{c|ccc}
\circ&4312&4132&1432\\\hline
4123&3241&3421&4321\\
1423&3214&3124&1324\\
1243&3412&3142&1342
\end{array}
\]
In this case, $\Gamma_\lambda \circ \Gamma_\mu$ is a union of four Knuth equivalence classes with insertion tableaux
\[\tiny \tableau{1&4\\2\\3}\,,\quad \tiny \tableau{1&2\\3&4}\,,\quad \tiny \tableau{1&2&4\\3}\,, \quad \normalsize\text{and} \quad \tiny \tableau{1\\2\\3\\4}\,.\]
Moreover, $g_{\lambda\mu\nu} = 1$ if $\nu = (2,1,1)$, $(2,2)$, $(3,1)$, or $(1,1,1,1)$, and $g_{\lambda\mu\nu} = 0$ otherwise, so the classes present in $\Gamma_\lambda \circ \Gamma_\mu$ precisely describe the Kronecker coefficients $g_{\lambda\mu\nu}$.

Lascoux \cite{Lascoux} showed that if $\lambda$ and $\mu$ are both hooks, then this phenomenon occurs in general.
\begin{theorem}[Lascoux's Kronecker Rule \cite{Lascoux}] If $\lambda$ and $\mu$ are hook shapes, then $\Gamma_\lambda \circ \Gamma_\mu$ is a union of Knuth equivalence classes, and $g_{\lambda\mu\nu}$ is the number of these classes with insertion tableau of shape $\nu$.
\end{theorem}

Lascoux \cite{Lascoux} and Garsia-Remmel \cite[\textsection6--7]{GarsiaRemmel} both investigate the extent to which this rule generalizes to other shapes. They give examples showing that it does not extend beyond the hook-hook case. Nevertheless, computations suggest that this rule is often close to holding even when it fails, and therefore it provides a valuable heuristic (see \cite[\textsection1]{BHook}).

For example, when only $\mu$ is a hook, $\Gamma_\lambda \circ \Gamma_\mu$ may not be a union of Knuth equivalence classes, but the corresponding quasisymmetric function $F^<_{\Gamma_\lambda \circ \Gamma_\mu}(\mathbf x)$ is still symmetric and equal to $\sum_{\nu} g_{\lambda\mu\nu} s_\nu(\mathbf x)$.
(This follows from Propositions~\ref{p intro sum of kron} and \ref{p intro word conversion QDes} and \eqref{e cyw stand} below.)

We can rephrase Lascoux's rule using colored words and standardization.
For a colored word $\e{w}$, the \emph{standardization of $\e{w}$}, denoted $\e{w}^{\stand}$,
is the permutation obtained from $\e{w}$ by first relabeling, from left to right (resp. right to left), the occurrences of the smallest letter (with respect to $<$) in $\e{w}$ by  $1,\ldots,k$
if this letter is unbarred (resp. barred),
then relabeling the occurrences of the next smallest letter of $\e{w}$ by $k+1,\ldots,k+k'$,
and so on.  For example,  $(\e{2 \crc{1} \crc{2} 1 \crc{2} \crc{1} 2 1})^\stand = 548 17 362$. (For more about standardization, see \textsection \ref{ss standardization}.)

For any colored word $\e{w} \in \CYW_{\lambda,d}$, $\e{w}^\stand = (\e{w}^\plainr)^\stand \circ \e{v}^\stand$, where $\e{v}$ is the colored word obtained from $\e{w}$ by replacing all unbarred letters with $1$ and all barred letters with $\crc 1$. Since $\e{w}^\plainr$ is Yamanouchi, $(\e{w}^\plainr)^\stand \in \Gamma_\lambda$. Similarly, $\e{v}^\stand \in \Gamma_{\mu(d)}$ if $\e{w}$ ends with an unbarred letter, and $\e{v}^\stand \in \Gamma_{\mu(d-1)}$ if $\e{w}$ ends with a barred letter.

Hence the set of standardizations of words in $\CYW_{\lambda, d}$ is
\begin{align}\label{e cyw stand}
(\CYW_{\lambda, d})^\stand = (\CYW_{\lambda, d}^-)^\stand \sqcup (\CYW_{\lambda, d}^+)^\stand = (\Gamma_\lambda \circ \Gamma_{\mu(d)}) \sqcup (\Gamma_\lambda \circ \Gamma_{\mu(d-1)}).
\end{align}
Lascoux's rule then implies that when $\lambda$ is a hook, $(\CYW_{\lambda, d})^\stand$ is a union of Knuth equivalence classes and hence $\CYW_{\lambda, d}$ is a union of $<$-colored plactic equivalence classes; that is, $\sum_{\e{w} \in \CYW_{\lambda, d}} \e{w} \in (\Iplacord{<})^\perp$. These equivalence classes give an explicit partition of $\CYW_{\lambda, d}$ from which one can directly obtain the coefficients $g_{\lambda\mu(d)\nu}$.

Unfortunately, this does not typically hold when $\lambda$ is not a hook (for example, see the discussion after Corollary~\ref{c intro main}). However, in proving Corollary~\ref{c intro main} we show that $\sum_{\e{w} \in \CYW_{\lambda, d}} \e{w} \in (\Ikron)^\perp$ for all $\lambda$. Although we cannot explicitly partition $\CYW_{\lambda, d}$ into classes to find the Schur expansion of $F^<_{\CYW_{\lambda, d}}(\mathbf x)$, we still get an explicit description of its Schur expansion from Corollary~\ref{c Ikron perp} using the subset of words of the form $\sqread(T)$. It would be interesting to determine if there exists an explicit partition (say, by way of a modified insertion algorithm) that gives this Schur expansion as in the case $\lambda$ is a hook.

\subsection{Organization}
The remainder of this paper is organized as follows. Section~\ref{s switchboards etc} supports the main results above by giving examples of Corollaries \ref{c Ikron perp} and \ref{c intro main};
we also prove two strengthenings of Corollary \ref{c intro main} and investigate a connection between noncommutative super Schur functions and conversion (\textsection\ref{ss conversion mysterious}).
Section~\ref{s some proofs} fills in the missing proofs from \textsection\ref{ss kronecker} and shows how the colored plactic algebra and ordinary plactic algebra are related via standardization.
Section~\ref{s reading} develops tableau combinatorics for the proof of the main theorem, and Section~\ref{s proof of theorem} is devoted to its proof.
In Section~\ref{s generalizations}, we conjecture a strengthening of the main theorem and compare it to results about LLT polynomials from \cite{BLamLLT}.
Finally, Section~\ref{s commuting super elementary symmetric functions} lays out the basic theory of noncommutative super Schur functions, including two results about when elementary functions commute.

\section{Switchboards and conversion}
\label{s switchboards etc}
This section supports the results of the previous section with examples, further context, and strengthenings.
We introduce certain graphs called $\Ikronknuth$-switchboards and use them to illustrate
Corollaries~\ref{c Ikron perp} and \ref{c intro main} (\textsection\ref{ss switchboards}),
put the results of the previous section  in the context of similar, easier results for colored plactic algebras (\textsection\ref{ss super Schur colored plactic}), and
prove two strengthenings of Corollary \ref{c intro main}, which hint at a mysterious connection between noncommutative super Schur functions and conversion (\textsection\ref{ss conversion mysterious}).

\subsection{$\Ikronknuth$-switchboards}
\label{ss switchboards}
Here we introduce certain graphs to give an intuitive understanding of Corollaries \ref{c Ikron perp} and \ref{c intro main}
and to give examples of these corollaries.

Let $\Ikronknuth$ be the two-sided ideal of $\U$ corresponding to the relations
\eqref{e plac nat rel knuth3},
\eqref{e plac nat rel knuth3b},
\eqref{e plac nat rel knuth4},
\eqref{e plac nat rel knuth4b},
\eqref{e kron rel rotate12},
 and
\begin{alignat}{3}
&\e{xzy} = \e{zxy} \qquad\text{for } && x,y,z \in \A,\, x<y<z,\, x<z \myDownarrow \myDownarrow, \label{e kronknuth rel knuth1} \\
&\e{yxz} = \e{yzx} \qquad\text{for } && x,y,z \in \A,\, x<y<z,\, x<z \myDownarrow \myDownarrow. \label{e kronknuth rel knuth2}
\end{alignat}
We conjecture that Theorem \ref{t J intro} (and hence Corollary~\ref{c Ikron perp}) holds with  $\Ikronknuth$ in place of  $\Ikron$ (see Conjecture \ref{cj kronknuth}).
The graphs we introduce below are better suited to studying these strengthenings than  the original statements.

\begin{definition}
\label{d CKR graph}
Let $\e{w}=\e{w_1\cdots w_n}$ and $\e{w'=w_1'\cdots w_n'}$
be two colored words of the same length~$n$ in the
alphabet $\A$.
We say that $\e{w}$ and $\e{w'}$  are related by a \emph{switch} in
position~$i$
if 
\renewcommand{\labelitemii}{\textbf{-}}
\begin{itemize}
\item
$\e{w}_j=\e{w}_j'$ for any $j\notin\{i-1,i,i+1\}$;
\item
the unordered pair 
$\{\e{w_{i-1}w_iw_{i+1}},
\e{w'_{i-1}w'_iw'_{i+1}}\}$
fits one of the following patterns:
\begin{list}{\ \ \ (\alph{ctr}) \ }{\usecounter{ctr} \setlength{\itemsep}{2pt} \setlength{\topsep}{3pt}}
\item
$\{\e{xzy}, \e{zxy}\}$ or $\{\e{yxz}, \e{yzx}\}$ for $x < y < z$, or 
\item
$\{\e{yxz}, \e{xzy}\}$ or $\{\e{yzx}, \e{zxy}\}$ for $x  = y \myDownarrow = z \myDownarrow\myDownarrow$, or
\item
$\{\e{yyx}, \e{yxy}\}$  for  $y \in \A_\varnothing$, $x < y$, or
\item
$\{\e{zyy}, \e{yzy}\}$ for  $y \in \A_\varnothing$, $y< z$, or
\item
$\{\e{yyz}, \e{yzy}\}$ for  $y \in \A_\crcempty$,  $y < z$, or
\item
$\{\e{xyy}, \e{yxy}\}$ for $y \in \A_\crcempty$,  $x <y$.
\end{list}
\end{itemize}
We refer to the switches in (b)
as \emph{rotation switches} and the
other switches as \emph{Knuth switches}.
\end{definition}

The next definition is an adaptation of the switchboards of \cite{BF} (which are based on the D graphs of Assaf \cite{SamiOct13})
to the super setting and the ideal  $\Ikronknuth$.

\begin{definition}
\label{def-switchboard}
An  $\Ikronknuth$-\emph{switchboard}
is an edge-labeled graph $\Gamma$ on a vertex set of colored words of fixed length~$n$
in the alphabet $\A$
with edge labels from the set \mbox{$\{2,3,\dots,n-1\}$} such that
each edge labeled~$i$ corresponds to a switch in position~$i$,
and each vertex in~$\Gamma$ 
which has exactly
one  $<$-descent in positions $i-1$ and~$i$
belongs to exactly one $i$-edge. 
\end{definition}

Note that there can be more than one $\Ikronknuth$-switchboard on
a given vertex set since the conditions in (a) and (b)
of Definition~\ref{d CKR graph} are not mutually exclusive---see Figure \ref{f not unique}.
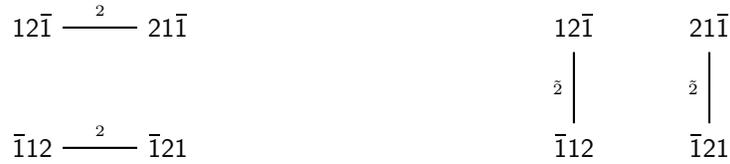
\begin{figure}
\begin{tikzpicture}[xscale = 1.8,yscale = 1.6]
\tikzstyle{vertex}=[inner sep=0pt, outer sep=4pt]
\tikzstyle{aedge} = [draw, thin, ->,black]
\tikzstyle{edge} = [draw, thick, -,black]
\tikzstyle{dashededge} = [draw, very thick, dashed, black]
\tikzstyle{LabelStyleH} = [text=black, anchor=south]
\tikzstyle{LabelStyleV} = [text=black, anchor=east]
\tikzstyle{LabelStyleH2} = [text=black, anchor=north]
\tikzstyle{doubleedge} = [draw, thick, double distance=1pt, -,black]
\tikzstyle{hiddenedge} = [draw=none, thick, double distance=1pt, -,black]

\begin{scope} [xshift = -2cm]
 \node[vertex] (v1) at (0,0){\footnotesize$\e{\crc{1}12}$};
 \node[vertex] (v2) at (1,0){\footnotesize$\e{\crc{1}21}$};
 \node[vertex] (v3) at (0,1){\footnotesize$\e{12\crc{1}}$};
 \node[vertex] (v4) at (1,1){\footnotesize$\e{21\crc{1}}$};
\end{scope}
 \draw[edge] (v1) to node[LabelStyleH]{{\Tiny 2}} (v2);
 \draw[edge] (v3) to node[LabelStyleH]{{\Tiny 2}} (v4);

\begin{scope} [xshift = 2cm]
\node[vertex] (v1) at (0,0){\footnotesize$\e{\crc{1}12}$};
 \node[vertex] (v2) at (1,0){\footnotesize$\e{\crc{1}21}$};
 \node[vertex] (v3) at (0,1){\footnotesize$\e{12\crc{1}}$};
 \node[vertex] (v4) at (1,1){\footnotesize$\e{21\crc{1}}$};
\end{scope}
 \draw[edge] (v1) to node[LabelStyleV]{{\Tiny $\tilde{2}$}} (v3);
 \draw[edge] (v2) to node[LabelStyleV]{{\Tiny $\tilde{2}$}} (v4);

 \end{tikzpicture}
\caption{\label{f not unique} On the left, an $\Ikronknuth$-switchboard whose edges are Knuth switches;
on the right,  an $\Ikronknuth$-switchboard on the same vertex set whose edges are rotation switches.
Knuth (resp. rotation) switches in position~$i$
are labeled~$i$ (resp.~$\tilde{i}$).}
\end{figure}

The $\Ikronknuth$-switchboards give a convenient and intuitive understanding of
the condition that a  $(0,1)$-vector  $\gamma \in \U$
lies in  $\Ikronknuth^\perp$.
\begin{proposition}
\label{p 01=switchboard}
For a set of colored words $W$ of the same length, the following are
equivalent:
\begin{itemize}
\item
$\sum_{\e{w} \in W} \e{w} \in \Ikronknuth^\perp$;
\item
$W$ is the vertex set of an  $\Ikronknuth$-switchboard.
\end{itemize}
\end{proposition}
We omit the proof of Proposition \ref{p 01=switchboard}, which is not difficult; a similar result is proved in {\cite[Proposition-Definition 3.2]{BD0graph}}.

\begin{proposition}\label{p CYW switchboard}
There is an $\Ikronknuth$-switchboard with vertex set
$\CYW_{\lambda, d}$ for any partition $\lambda$ of  $n$ and  $d \leq n$.
Moreover, there is a unique $\Ikronknuth$-switchboard with this vertex set in which every  $i$-edge
$\{\e{w},\e{w}'\}$ such that $\{\e{w_{i-1},w_i,w_{i+1}}\} = \{\e{x,y,z}\}$ with $x  = y \myDownarrow = z \myDownarrow\myDownarrow$
is a rotation switch.
\end{proposition}
\begin{proof}
By Proposition \ref{p CYW in kron perp},
$\sum_{\e{w} \in \CYW_{\lambda, d}} \e{w} \in \Ikron^\perp \subset \Ikronknuth^\perp$.
Hence the first statement follows from Proposition \ref{p 01=switchboard}.

Since the set $\CYW_{\lambda,d}$ is closed under shuffling barred letters past unbarred letters,
if $\e{w} \in \CYW_{\lambda,d}$ and $\{\e{w_{i-1}w_iw_{i+1}}\}$ fits one of the patterns
$\e{yxz}, \e{xzy}, \e{yzx}, \e{zxy}$ with $x  = y \myDownarrow = z \myDownarrow\myDownarrow$,
then the unique  $\e{w'}$ such that  $\e{w}$ and $\e{w'}$ are related by a rotation switch in position  $i$ belongs to  $\CYW_{\lambda, d}$.
The second statement follows.
\end{proof}


\begin{figure}
        \centerfloat
\begin{tikzpicture}[xscale = 2.54, yscale = 1.36]
\tikzstyle{vertex}=[inner sep=0pt, outer sep=4pt]
\tikzstyle{framedvertex}=[inner sep=2.4pt, outer sep=4pt, draw=gray]
\tikzstyle{aedge} = [draw, thin, ->,black]
\tikzstyle{edge} = [draw, thick, -,black]
\tikzstyle{doubleedge} = [draw, thick, double distance=1pt, -,black]
\tikzstyle{hiddenedge} = [draw=none, thick, double distance=1pt, -,black]
\tikzstyle{dashededge} = [draw, very thick, dashed, black]
\tikzstyle{LabelStyleH} = [text=black, anchor=south]
\tikzstyle{LabelStyleHn} = [text=black, anchor=north]
\tikzstyle{LabelStyleV} = [text=black, anchor=east]
\tikzstyle{LabelStyleVn} = [text=black, anchor=west]

\node[framedvertex] (v1) at (0,9){\footnotesize$\e{\crc{1}\crc{1}122} $};
\node[framedvertex] (v2) at (0,8){\footnotesize$\e{\crc{1}1\crc{1}22} $};
\node[vertex] (v3) at (0,7){\footnotesize$\e{\crc{1}12\crc{1}2} $};
\node[vertex] (v4) at (0,6){\footnotesize$\e{\crc{1}122\crc{1}} $};
\node[vertex] (v5) at (0,0){\footnotesize$\e{1\crc{1}\crc{1}22} $};
\node[vertex] (v6) at (0,1){\footnotesize$\e{1\crc{1}2\crc{1}2} $};
\node[vertex] (v7) at (0,2){\footnotesize$\e{1\crc{1}22\crc{1}} $};
\node[vertex] (v8) at (0,3){\footnotesize$\e{12\crc{1}\crc{1}2} $};
\node[vertex] (v9) at (0,4){\footnotesize$\e{12\crc{1}2\crc{1}} $};
\node[vertex] (v10) at (0,5){\footnotesize$\e{122\crc{1}\crc{1}} $};
\node[vertex] (v11) at (3,9){\footnotesize$\e{\crc{1}\crc{2}121} $};
\node[vertex] (v12) at (3,8){\footnotesize$\e{\crc{1}1\crc{2}21} $};
\node[vertex] (v13) at (3,7){\footnotesize$\e{\crc{1}12\crc{2}1} $};
\node[vertex] (v14) at (3,6){\footnotesize$\e{\crc{1}121\crc{2}} $};
\node[vertex] (v15) at (3,0){\footnotesize$\e{1\crc{1}\crc{2}21} $};
\node[vertex] (v16) at (3,1){\footnotesize$\e{1\crc{1}2\crc{2}1} $};
\node[vertex] (v17) at (3,2){\footnotesize$\e{1\crc{1}21\crc{2}} $};
\node[vertex] (v18) at (3,3){\footnotesize$\e{12\crc{1}\crc{2}1} $};
\node[vertex] (v19) at (3,4){\footnotesize$\e{12\crc{1}1\crc{2}} $};
\node[vertex] (v20) at (3,5){\footnotesize$\e{121\crc{1}\crc{2}} $};
\node[vertex] (v21) at (4,9){\footnotesize$\e{\crc{1}\crc{2}211} $};
\node[vertex] (v22) at (4,8){\footnotesize$\e{\crc{1}2\crc{2}11} $};
\node[vertex] (v23) at (4,7){\footnotesize$\e{\crc{1}21\crc{2}1} $};
\node[framedvertex] (v24) at (4,6){\footnotesize$\e{\crc{1}211\crc{2}} $};
\node[vertex] (v25) at (4,0){\footnotesize$\e{2\crc{1}\crc{2}11} $};
\node[framedvertex] (v26) at (4,1){\footnotesize$\e{2\crc{1}1\crc{2}1} $};
\node[framedvertex] (v27) at (4,2){\footnotesize$\e{2\crc{1}11\crc{2}} $};
\node[vertex] (v28) at (4,3){\footnotesize$\e{21\crc{1}\crc{2}1} $};
\node[vertex] (v29) at (4,4){\footnotesize$\e{21\crc{1}1\crc{2}} $};
\node[framedvertex] (v30) at (4,5){\footnotesize$\e{211\crc{1}\crc{2}} $};
\node[vertex] (v31) at (1,9){\footnotesize$\e{\crc{1}\crc{1}212} $};
\node[vertex] (v32) at (1,8){\footnotesize$\e{\crc{1}2\crc{1}12} $};
\node[framedvertex] (v33) at (1,7){\footnotesize$\e{\crc{1}21\crc{1}2} $};
\node[vertex] (v34) at (1,6){\footnotesize$\e{\crc{1}212\crc{1}} $};
\node[framedvertex] (v35) at (1,0){\footnotesize$\e{2\crc{1}\crc{1}12} $};
\node[framedvertex] (v36) at (1,1){\footnotesize$\e{2\crc{1}1\crc{1}2} $};
\node[vertex] (v37) at (1,2){\footnotesize$\e{2\crc{1}12\crc{1}} $};
\node[vertex] (v38) at (1,3){\footnotesize$\e{21\crc{1}\crc{1}2} $};
\node[vertex] (v39) at (1,4){\footnotesize$\e{21\crc{1}2\crc{1}} $};
\node[vertex] (v40) at (1,5){\footnotesize$\e{212\crc{1}\crc{1}} $};
\node[vertex] (v41) at (2,9){\footnotesize$\e{\crc{1}\crc{1}221} $};
\node[vertex] (v42) at (2,8){\footnotesize$\e{\crc{1}2\crc{1}21} $};
\node[vertex] (v43) at (2,7){\footnotesize$\e{\crc{1}22\crc{1}1} $};
\node[vertex] (v44) at (2,6){\footnotesize$\e{\crc{1}221\crc{1}} $};
\node[vertex] (v45) at (2,0){\footnotesize$\e{2\crc{1}\crc{1}21} $};
\node[vertex] (v46) at (2,1){\footnotesize$\e{2\crc{1}2\crc{1}1} $};
\node[framedvertex] (v47) at (2,2){\footnotesize$\e{2\crc{1}21\crc{1}} $};
\node[vertex] (v48) at (2,3){\footnotesize$\e{22\crc{1}\crc{1}1} $};
\node[vertex] (v49) at (2,4){\footnotesize$\e{22\crc{1}1\crc{1}} $};
\node[vertex] (v50) at (2,5){\footnotesize$\e{221\crc{1}\crc{1}} $};
\draw[edge, bend left=25] (v1) to node[LabelStyleVn]{\Tiny$\tilde{3} $} (v3);
\draw[edge, bend right=20] (v2) to node[LabelStyleV]{\Tiny$2 $} (v5);
\draw[edge, bend right=25] (v3) to node[LabelStyleV]{\Tiny$\tilde{2} $} (v8);
\draw[edge] (v3) to node[LabelStyleVn]{\Tiny$4 $} (v4);
\draw[edge, bend right=25] (v4) to node[LabelStyleV]{\Tiny$\tilde{2} $} (v9);
\draw[edge] (v5) to node[LabelStyleVn]{\Tiny$3 $} (v6);
\draw[edge] (v6) to node[LabelStyleVn]{\Tiny$4 $} (v7);
\draw[edge] (v8) to node[LabelStyleVn]{\Tiny$4 $} (v9);
\draw[edge] (v9) to node[LabelStyleVn]{\Tiny$3 $} (v10);
\draw[edge] (v11) to node[LabelStyleH]{\Tiny$4 $} (v21);
\draw[hiddenedge] (v11) to node[LabelStyleVn]{\Tiny$2 $} (v12);
\draw[doubleedge] (v11) to node[LabelStyleV]{\Tiny$3 $} (v12);
\draw[edge] (v13) to node[LabelStyleVn]{\Tiny$4 $} (v14);
\draw[edge, bend right=25] (v13) to node[LabelStyleV]{\Tiny$\tilde{2} $} (v18);
\draw[edge] (v14) to node[LabelStyleH]{\Tiny$3 $} (v24);
\draw[edge, bend right=25] (v14) to node[LabelStyleV]{\Tiny$\tilde{2} $} (v19);
\draw[edge, bend left=25] (v15) to node[LabelStyleV]{\Tiny$\tilde{3} $} (v18);
\draw[edge] (v16) to node[LabelStyleVn]{\Tiny$4 $} (v17);
\draw[edge, bend right=25] (v17) to node[LabelStyleVn]{\Tiny$\tilde{3} $} (v20);
\draw[edge] (v18) to node[LabelStyleVn]{\Tiny$4 $} (v19);
\draw[edge] (v20) to node[LabelStyleH]{\Tiny$2 $} (v30);
\draw[edge, bend left=19] (v21) to node[LabelStyleVn]{\Tiny$\tilde{2} $} (v25);
\draw[hiddenedge] (v22) to node[LabelStyleV]{\Tiny$4 $} (v23);
\draw[doubleedge] (v22) to node[LabelStyleVn]{\Tiny$3 $} (v23);
\draw[edge, bend left=25] (v23) to node[LabelStyleVn]{\Tiny$\tilde{2} $} (v28);
\draw[edge, bend left=25] (v24) to node[LabelStyleVn]{\Tiny$\tilde{2} $} (v29);
\draw[hiddenedge] (v25) to node[LabelStyleV]{\Tiny$3 $} (v26);
\draw[doubleedge] (v25) to node[LabelStyleVn]{\Tiny$4 $} (v26);
\draw[edge, bend left=25] (v27) to node[LabelStyleV]{\Tiny$3 $} (v29);
\draw[edge] (v28) to node[LabelStyleVn]{\Tiny$4 $} (v29);
\draw[edge] (v31) to node[LabelStyleVn]{\Tiny$2 $} (v32);
\draw[edge, bend right=25] (v31) to node[LabelStyleV]{\Tiny$\tilde{3} $} (v33);
\draw[edge] (v31) to node[LabelStyleH]{\Tiny$4 $} (v41);
\draw[edge, bend left=25] (v32) to node[LabelStyleVn]{\Tiny$\tilde{4} $} (v34);
\draw[edge, bend right=25] (v33) to node[LabelStyleV]{\Tiny$\tilde{2} $} (v38);
\draw[edge, bend right=25] (v34) to node[LabelStyleV]{\Tiny$\tilde{2} $} (v39);
\draw[edge] (v34) to node[LabelStyleH]{\Tiny$3 $} (v44);
\draw[edge, bend right=25] (v35) to node[LabelStyleVn]{\Tiny$\tilde{4} $} (v37);
\draw[edge, bend left=25] (v36) to node[LabelStyleV]{\Tiny$3 $} (v38);
\draw[edge, bend right=25] (v37) to node[LabelStyleVn]{\Tiny$\tilde{3} $} (v40);
\draw[edge] (v38) to node[LabelStyleVn]{\Tiny$4 $} (v39);
\draw[edge] (v40) to node[LabelStyleH]{\Tiny$2 $} (v50);
\draw[edge] (v41) to node[LabelStyleV]{\Tiny$2 $} (v42);
\draw[edge, bend left=25] (v42) to node[LabelStyleVn]{\Tiny$\tilde{4} $} (v44);
\draw[edge] (v42) to node[LabelStyleV]{\Tiny$3 $} (v43);
\draw[edge] (v45) to node[LabelStyleV]{\Tiny$3 $} (v46);
\draw[edge, bend right=25] (v45) to node[LabelStyleVn]{\Tiny$\tilde{4} $} (v47);
\draw[edge, bend left=25] (v46) to node[LabelStyleV]{\Tiny$2 $} (v48);
\draw[edge, bend right=28] (v47) to node[LabelStyleV]{\Tiny$2 $} (v49);
\draw[edge, bend right=30] (v47) to node[LabelStyleV]{\Tiny$\tilde{3} $} (v50);
\draw[edge] (v49) to node[LabelStyleV]{\Tiny$4 $} (v50);

\end{tikzpicture}
\vspace{-2pt}
\caption{\label{f CYW 32}
An $\Ikronknuth$-switchboard on the vertex set $\CYW_{(3,2), 2}$.
The Schur expansion of the symmetric function  $F^<_{\CYW_{(3,2), 2}}(\mathbf{x})$ can be read off from the outlined words; see Example~\ref{ex CYW 32}.}
\end{figure}
\vspace{-4mm}
\begin{example}\label{ex CYW 32}
Let $\lambda = (3,2)$,  $d=2$.
The set  $\CYW_{\lambda,d}$ is shown in Figure \ref{f CYW 32}, along with the
$\Ikronknuth$-switchboard on this vertex set that is described in Proposition \ref{p CYW switchboard}.
This $\Ikronknuth$-switchboard gives a way to ``see'' that $\sum_{\e{w} \in \CYW_{\lambda, d}} \e{w} \in \Ikronknuth^\perp$.
Since only letters $\leq \crc{2}$ appear in  $\CYW_{\lambda,d}$, this also implies $\sum_{\e{w} \in \CYW_{\lambda, d}} \e{w} \in \Ikron$.

To compute the symmetric function $F^<_{\CYW_{\lambda, d}}(\mathbf{x})$, i.e., the sum of quasisymmetric functions associated to the
 $<$-descent sets of  $\CYW_{\lambda,d}$, we apply Corollary~\ref{c Ikron perp}, which
says that the coefficient of  $s_\nu(\mathbf{x})$ in $F^<_{\CYW_{\lambda, d}}(\mathbf{x})$ is equal to the
number of  $<$-colored tableaux  $T$ of shape  $\nu$ such that  $\sqread(T) \in \CYW_{\lambda, d}$.
The next line gives the $<$-colored tableau $T$ such that $\sqread(T) \in \CYW_{\lambda,d}$,
and below each tableau $T$ is the colored word $\sqread(T)$ (these are the outlined
words in Figure \ref{f CYW 32}).
\[
\begin{array}{ccccccccccccccccccccccccccc}
{\tiny \tableau{1 & 1 & \crc{1} & \crc{2} \\ 2}}
&{\tiny \tableau{1 & \crc{1} & 2 & 2 \\ \crc{1}}}
&{\tiny \tableau{1 & 1 & \crc{2} \\ \crc{1} & 2}}
&{\tiny \tableau{1 & \crc{1} & 2 \\ \crc{1} & 2}}
&{\tiny \tableau{1 & 1 & \crc{2} \\ \crc{1} \\ 2}}
&{\tiny \tableau{1 & \crc{1} & 2 \\ \crc{1} \\ 2}}
&{\tiny \tableau{1 & 2 & 2 \\ \crc{1} \\ \crc{1}}}
&{\tiny \tableau{1 & \crc{1} \\ \crc{1} & 2 \\ 2}}
&{\tiny \tableau{1 & 1 \\ \crc{1} & \crc{2} \\ 2}}
&{\tiny \tableau{1 & 2 \\ \crc{1} \\ \crc{1} \\ 2}} \\[7mm]
\e{2 1 1 \crc{1} \crc{2}}
&\e{\crc{1} 1 \crc{1} 2 2}
&\e{\crc{1} 2 1 1 \crc{2}}
&\e{\crc{1} 2 1 \crc{1} 2}
&\e{2 \crc{1} 1 1 \crc{2}}
&\e{2 \crc{1} 1 \crc{1} 2}
&\e{\crc{1} \crc{1} 1 2 2}
&\e{2 \crc{1} 2 1 \crc{1}}
&\e{2 \crc{1} 1 \crc{2} 1}
&\e{2 \crc{1} \crc{1} 1 2}
\end{array}
\]
Hence
$F^<_{\CYW_{\lambda, d}} = 2s_{41} + 2s_{32} + 3s_{311} + 2s_{221} + s_{2111}.$
\end{example}

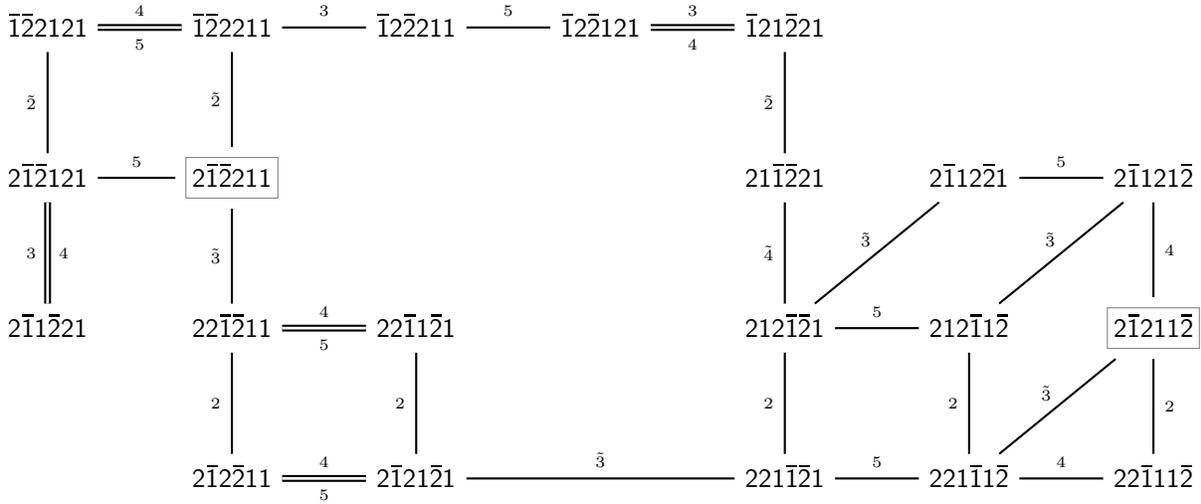
\begin{figure}
        \centerfloat
\begin{tikzpicture}[xscale = 2.45, yscale = 2]
\tikzstyle{vertex}=[inner sep=0pt, outer sep=4pt]
\tikzstyle{framedvertex}=[inner sep=2.4pt, outer sep=4pt, draw=gray]
\tikzstyle{aedge} = [draw, thin, ->,black]
\tikzstyle{edge} = [draw, thick, -,black]
\tikzstyle{doubleedge} = [draw, thick, double distance=1pt, -,black]
\tikzstyle{hiddenedge} = [draw=none, thick, double distance=1pt, -,black]
\tikzstyle{dashededge} = [draw, very thick, dashed, black]
\tikzstyle{LabelStyleH} = [text=black, anchor=south]
\tikzstyle{LabelStyleHn} = [text=black, anchor=north]
\tikzstyle{LabelStyleNE} = [text=black, anchor=south west, inner sep=2pt]
\tikzstyle{LabelStyleV} = [text=black, anchor=east]
\tikzstyle{LabelStyleVn} = [text=black, anchor=west]
\tikzstyle{LabelStyleNW} = [text=black, anchor=south east, inner sep=2pt]

\node[vertex] (v1) at (0,3){\footnotesize$\e{\crc{1}\crc{2}2121} $};
\node[vertex] (v2) at (3,3){\footnotesize$\e{\crc{1}2\crc{2}121} $};
\node[vertex] (v3) at (4,3){\footnotesize$\e{\crc{1}21\crc{2}21} $};
\node[vertex] (v4) at (0,2){\footnotesize$\e{2\crc{1}\crc{2}121} $};
\node[vertex] (v5) at (0,1){\footnotesize$\e{2\crc{1}1\crc{2}21} $};
\node[vertex] (v6) at (5,2){\footnotesize$\e{2\crc{1}12\crc{2}1} $};
\node[vertex] (v7) at (6,2){\footnotesize$\e{2\crc{1}121\crc{2}} $};
\node[vertex] (v8) at (4,2){\footnotesize$\e{21\crc{1}\crc{2}21} $};
\node[vertex] (v9) at (4,1){\footnotesize$\e{212\crc{1}\crc{2}1} $};
\node[vertex] (v10) at (5,1){\footnotesize$\e{212\crc{1}1\crc{2}} $};
\node[vertex] (v11) at (1,3){\footnotesize$\e{\crc{1}\crc{2}2211} $};
\node[vertex] (v12) at (2,3){\footnotesize$\e{\crc{1}2\crc{2}211} $};
\node[framedvertex] (v13) at (1,2){\footnotesize$\e{2\crc{1}\crc{2}211} $};
\node[vertex] (v14) at (1,0){\footnotesize$\e{2\crc{1}2\crc{2}11} $};
\node[vertex] (v15) at (2,0){\footnotesize$\e{2\crc{1}21\crc{2}1} $};
\node[framedvertex] (v16) at (6,1){\footnotesize$\e{2\crc{1}211\crc{2}} $};
\node[vertex] (v17) at (1,1){\footnotesize$\e{22\crc{1}\crc{2}11} $};
\node[vertex] (v18) at (2,1){\footnotesize$\e{22\crc{1}1\crc{2}1} $};
\node[vertex] (v19) at (6,0){\footnotesize$\e{22\crc{1}11\crc{2}} $};
\node[vertex] (v20) at (4,0){\footnotesize$\e{221\crc{1}\crc{2}1} $};
\node[vertex] (v21) at (5,0){\footnotesize$\e{221\crc{1}1\crc{2}} $};
\draw[doubleedge] (v1) to node[LabelStyleHn]{\Tiny$5 $} (v11);
\draw[edge] (v1) to node[LabelStyleV]{\Tiny$\tilde{2} $} (v4);
\draw[hiddenedge] (v1) to node[LabelStyleH]{\Tiny$4 $} (v11);
\draw[hiddenedge] (v2) to node[LabelStyleH]{\Tiny$3 $} (v3);
\draw[edge] (v2) to node[LabelStyleH]{\Tiny$5 $} (v12);
\draw[doubleedge] (v2) to node[LabelStyleHn]{\Tiny$4 $} (v3);
\draw[edge] (v3) to node[LabelStyleV]{\Tiny$\tilde{2} $} (v8);
\draw[edge] (v4) to node[LabelStyleH]{\Tiny$5 $} (v13);
\draw[hiddenedge] (v4) to node[LabelStyleV]{\Tiny$3 $} (v5);
\draw[doubleedge] (v4) to node[LabelStyleVn]{\Tiny$4 $} (v5);
\draw[edge] (v6) to node[LabelStyleNW]{\Tiny$\tilde{3} $} (v9);
\draw[edge] (v6) to node[LabelStyleH]{\Tiny$5 $} (v7);
\draw[edge] (v7) to node[LabelStyleNW]{\Tiny$\tilde{3} $} (v10);
\draw[edge] (v7) to node[LabelStyleVn]{\Tiny$4 $} (v16);
\draw[edge] (v8) to node[LabelStyleV]{\Tiny$\tilde{4} $} (v9);
\draw[edge] (v9) to node[LabelStyleV]{\Tiny$2 $} (v20);
\draw[edge] (v9) to node[LabelStyleH]{\Tiny$5 $} (v10);
\draw[edge] (v10) to node[LabelStyleV]{\Tiny$2 $} (v21);
\draw[edge] (v11) to node[LabelStyleH]{\Tiny$3 $} (v12);
\draw[edge] (v11) to node[LabelStyleV]{\Tiny$\tilde{2} $} (v13);
\draw[edge] (v13) to node[LabelStyleV]{\Tiny$\tilde{3} $} (v17);
\draw[edge] (v14) to node[LabelStyleV]{\Tiny$2 $} (v17);
\draw[hiddenedge] (v14) to node[LabelStyleHn]{\Tiny$5 $} (v15);
\draw[doubleedge] (v14) to node[LabelStyleH]{\Tiny$4 $} (v15);
\draw[edge] (v15) to node[LabelStyleH]{\Tiny$\tilde{3} $} (v20);
\draw[edge] (v15) to node[LabelStyleV]{\Tiny$2 $} (v18);
\draw[edge] (v16) to node[LabelStyleVn]{\Tiny$2 $} (v19);
\draw[edge] (v16) to node[LabelStyleNW]{\Tiny$\tilde{3} $} (v21);
\draw[hiddenedge] (v17) to node[LabelStyleH]{\Tiny$4 $} (v18);
\draw[doubleedge] (v17) to node[LabelStyleHn]{\Tiny$5 $} (v18);
\draw[edge] (v19) to node[LabelStyleH]{\Tiny$4 $} (v21);
\draw[edge] (v20) to node[LabelStyleH]{\Tiny$5 $} (v21);

\end{tikzpicture}
\caption{\label{f some CYW 33 2}
The $\Ikronknuth$-switchboard  $\Gamma_1$ from Example~\ref{ex CYW 33}.
As explained in the example, its symmetric function  $F^<_{\ver(\Gamma_1)} = s_{321}+s_{222}$
can be computed from the outlined words.}
\end{figure}

\begin{definition}\label{d insertion}
For a colored word $\e{w}$ and shuffle order $\lessdot$ on  $\A$, \emph{the  $\lessdot$-insertion tableau of $\e{w}$}, denoted $P^\lessdot(\e{w})$, is defined using the usual Schensted insertion algorithm
using the order $\lessdot$ except that when a barred letter  $\crc{a}$ is inserted into a row, it bumps the smallest letter  $\underbar{\gtrdot} \, \crc{a}$
(an unbarred letter  $a$ bumps the smallest letter  $\gtrdot \, a$ as usual).
The resulting tableau is indeed a  $\lessdot$-colored tableau (see \cite[\textsection3]{BereleRemmel} for details).
%
\end{definition}

\begin{remark}
It is easy to show that a colored word  $\e{w}$ is of the form $\sqread(T)$ for  $T \in \CT^<_\nu$ if and only  if $\e{w} = \sqread(P^<(\e{w}))$.
This gives a good way in practice to compute the subset of words of the form  $\sqread(T)$ of a given set of colored words.
 \end{remark}

\begin{example}\label{ex CYW 33}
Let $\lambda = (3,3)$,  $d=2$.  The $\Ikronknuth$-switchboard
with vertex set $\CYW_{\lambda,d}$ that is described in Proposition \ref{p CYW switchboard}
has four components $\Gamma_1, \Gamma_2, \Gamma_3, \Gamma_4$.
The component $\Gamma_1$ is depicted in Figure \ref{f some CYW 33 2}.
It follows from Proposition \ref{p 01=switchboard} and the fact that $\CYW_{\lambda, d}$ only involves letters  $\le \crc{2}$,
that $\sum_{\e{w} \in \ver(\Gamma_i)} \e{w} \in \Ikron^\perp \subseteq \Ikronknuth^\perp$.
Hence by Corollary~\ref{c Ikron perp},
the symmetric function $F^<_{\ver(\Gamma_i)}(\mathbf{x})$ is computed by finding all  $<$-colored tableaux  $T$ such that
$\sqread(T) \in \ver(\Gamma_i)$.

In the table below, we give the symmetric functions  $F^<_{\ver(\Gamma_i)}(\mathbf{x})$, the sets
\[\{\text{$<$-colored tableau } T \mid \sqread(T) \in \ver(\Gamma_i)\},\]
and the words  $\sqread(T)$ for each tableau $T$ in these sets.
\[
\begin{tabular}{cc@{~~}c|@{\quad}c@{\quad}|@{\quad}cc@{~~}c|@{\quad}cc}
\multicolumn{2}{c}{$\Gamma_1$} & & $\Gamma_2$ & \multicolumn{2}{c}{$\Gamma_3$} & & \multicolumn{2}{c}{$\Gamma_4$}\\[2mm]\hline
\multicolumn{2}{c}{$s_{321}+s_{222}$}
&&$s_{42}$
&\multicolumn{2}{c}{$s_{42}+s_{411}$}
&&\multicolumn{2}{c}{$s_{321}+s_{3111}$}\\[2mm]
${\tiny \tableau{1 & 1 & \crc{2} \\ \crc{1} & 2\\ 2}}$
&${\tiny \tableau{1 & 1 \\ \crc{1} & 2\\ 2 & \crc{2}}}$
&&${\tiny \tableau{1 & 1 & \crc{1} & \crc{2}\\ 2 & 2}}$
&${\tiny \tableau{1 & \crc{1} & 2 & 2\\ \crc{1} & 2}}$
&${\tiny \tableau{1 & \crc{1} & 2 & 2\\ \crc{1}\\ 2}}$
&&${\tiny \tableau{1 & \crc{1} & 2\\ \crc{1} &  2\\ 2}}$
&${\tiny \tableau{1 & 2 & 2\\ \crc{1}\\ \crc{1}\\ 2}}$\\[8mm]
$\e{2 \crc{1} 2 1 1 \crc{2}}$
&$\e{2 \crc{1} \crc{2} 2 1 1}$
&&$\e{2 2 1 1 \crc{1} \crc{2}}$
&$\e{\crc{1} 2 1 \crc{1} 2 2}$
&$\e{2 \crc{1} 1 \crc{1} 2 2}$
&&$\e{2 \crc{1} 2 1 \crc{1} 2}$
&$\e{2 \crc{1} \crc{1} 1 2 2}$
\end{tabular}
\]

This yields an expression for  $F^<_{\CYW_{\lambda,d}}$ as a sum of four Schur positive expressions:
$F^<_{\CYW_{\lambda,d}} = \sum_i F^<_{\ver(\Gamma_i)}$.
This shows that Corollary~\ref{c Ikron perp} is stronger than Corollary~\ref{c intro main}.
\end{example}

\subsection{Noncommutative super Schur functions in colored plactic algebras}
\label{ss super Schur colored plactic}
Here we compare Theorem \ref{t J intro} and Corollaries \ref{c Ikron perp} and \ref{c intro main}
to analogous, but easier results for colored plactic algebras for any shuffle order (defined below).

We first extend several definitions involving the natural order (from \textsection\ref{ss quotients of U}--\ref{ss main theorem}) to an arbitrary shuffle order on  $\A$.
Throughout this subsection, fix a shuffle order  $\lessdot$ on  $\A$.

The \emph{$\lessdot$-colored plactic algebra}, denoted $\U/\Iplacord{\lessdot}$, is the quotient of $\U$ by the relations
\begin{alignat}{3}
&\e{xzy} = \e{zxy} \qquad &&\text{for } x,y,z \in \A,\, x\lessdot y\lessdot z,\label{e plac rel knuth1} \\
&\e{yxz} = \e{yzx} \qquad &&\text{for } x,y,z \in \A,\, x\lessdot y\lessdot z,\label{e plac rel knuth2} \\
&\e{yyx} = \e{yxy} \qquad &&\text{for } x \in \A,\, y \in \A_{\varnothing},\, x\lessdot y, \label{e plac rel knuth3} \\
&\e{zyy} = \e{yzy} \qquad &&\text{for } z \in \A,\, y \in \A_{\varnothing},\, y\lessdot z, \label{e plac rel knuth3b} \\
&\e{yyz} = \e{yzy} \qquad &&\text{for } z \in \A,\, y \in \A_{\crcempty},\, y\lessdot z, \label{e plac rel knuth4}\\
&\e{xyy} = \e{yxy} \qquad &&\text{for } x \in \A,\, y \in \A_{\crcempty},\, x\lessdot y \label{e plac rel knuth4b};
\end{alignat}
let $\Iplacord{\lessdot}$ denote the corresponding two-sided ideal of $\U$.
The relations \eqref{e plac rel knuth1}--\eqref{e plac rel knuth4b} are the same as \eqref{e plac nat rel knuth1}--\eqref{e plac nat rel knuth4b} with  $\lessdot$ in place of the natural order $<$.

\begin{definition}[\emph{Noncommutative super Schur functions for shuffle orders}]
\label{d noncommutative Schur order}
For any two letters $x, y \in \A$, write $\noe{y} \gedotcol \noe{x}$ to mean either $\noe{y \gtrdot x}$, or $\noe y$ and $\noe x$ are equal barred letters.
Define the following generalization of the  $e_k(\mathbf{u})$:
\[
e^\lessdot_k(\mathbf{u})=\sum_{\substack{z_1 \gedotcol z_2 \gedotcol \cdots \gedotcol z_k \\ z_1,\dots,z_k \in \A}} u_{z_1} u_{z_2} \cdots u_{z_k} \ \in \U
\]
for any positive integer $k$; set $e^\lessdot_0(\mathbf{u})=1$ and $e^\lessdot_k(\mathbf{u}) = 0$ for $k<0$.
Now define the noncommutative super Schur functions  $\mathfrak{J}^\lessdot_\nu(\mathbf{u})$ exactly as in Definition~\ref{d noncommutative Schur functions}, except
with  $e^\lessdot_k(\mathbf{u})$ in place of  $e_k(\mathbf{u})$.
%
\end{definition}

A \emph{$\lessdot$-colored tableau} is a tableau with entries in $\A$ such that each row and column is weakly increasing with respect to the order $\lessdot$, while the unbarred letters in each column and the barred letters in each row are strictly increasing.
Let $\CT_{\nu}^\lessdot$ denote the set of  $\lessdot$-colored tableaux of shape $\nu$.

For a  $\lessdot$-colored tableaux  $T$, the \emph{column reading word} of $T$, denoted $\creading(T)$, is the colored word obtained by concatenating the columns of  $T$ (reading each column bottom to top), starting with the leftmost column.
For example,
\[\creading\Bigg( \, { \tiny \tableau{1 & 1 & \crc{3} & \crc{4} & 6\\  \crc{2} & 3 & 4 & \crc{4} \\ 3 & \crc{3} & \crc{4} & 5}} \, \Bigg)
= \e{3\, \crc{2}\, 1\, \crc{3}\, 3\, 1\, \crc{4}\, 4\, \crc{3}\, 5\, \crc{4}\, \crc{4}\, 6}.\]
Also, for any  $\lessdot$-colored tableaux  $T$, we define the colored word $\sqread(T)$ exactly as in Definition~\ref{d sqread}.

The next theorem is an analog of Theorem \ref{t J intro} for the  $\lessdot$-colored plactic algebra. 

\begin{theorem}\label{t J plac}
For any partition $\nu$,
\begin{align*}
\mathfrak{J}^\lessdot_\nu(\mathbf{u}) =
\sum_{ T \in \CT_{\nu}^{\lessdot}} \creading(T) \qquad \text{in \, $\U/\Iplacord{\lessdot}$}.
\end{align*}
\end{theorem}
This theorem follows from a similar result for the ordinary plactic algebra and a standardization argument.
So as not to interrupt the discussion, we postpone the proof to \textsection\ref{ss standardization}.

\begin{remark} \label{r insertion words equivalent}
The word $\creading(T)$ in Theorem \ref{t J plac} can be replaced by any colored word that has  $\lessdot$-insertion tableau $T$   (see Definition~\ref{d insertion}).
In particular, it can be replaced by $\sqread(T)$. Here we are using that
$\lessdot$-colored plactic equivalence classes are the same as
sets of colored words with a fixed $\lessdot$-insertion tableau;
this can be shown by a standardization argument similar to that in Proposition \ref{p plac stand}.
\end{remark}

Similar to Theorem \ref{t intro basics}, the next result is by Proposition \ref{p es commute plac} and Theorem \ref{t basics full}
(it is also a straightforward adaptation of the setup of \cite{FG}).
Recall from Definition~\ref{d F gamma} the notation $F^\lessdot_\gamma(\mathbf{x}) = \sum_{\e{w}} \gamma_\e{w} Q_{\Des_{\lessdot}(\e{w})}(\mathbf{x})$ for
 $\gamma = \sum_{\e{w}}\gamma_\e{w}\e{w}$, and $F^\lessdot_W(\mathbf{x}) = \sum_{\e{w} \in W} Q_{\Des_{\lessdot}(\e{w})}(\mathbf{x})$ for a set of colored words  $W$.
\begin{theorem} \label{t basics 2}
For any $\gamma \in (\Iplacord{\lessdot})^\perp$, the function  $F^\lessdot_\gamma(\mathbf{x})$ is symmetric and
\[
F^\lessdot_\gamma(\mathbf{x})
 = \sum_{\nu} s_\nu(\mathbf x) \langle \mathfrak{J}^\lessdot_{\nu}(\mathbf{u}), \gamma \rangle. \]
\end{theorem}

Theorems \ref{t basics 2} and \ref{t J plac} and Remark \ref{r insertion words equivalent}
then have the following consequence, which is an analog of Corollary~\ref{c Ikron perp} for the ideal $\Iplacord{\lessdot}$.
\begin{corollary}\label{c Iplac perp}
For any set of colored words $W$ such that $\sum_{\e{w} \in W} \e{w} \in (\Iplacord{\lessdot})^\perp$,
\begin{align*}
\Big(\text{the coefficient of  $s_\nu(\mathbf{x})$ in  $F^\lessdot_W(\mathbf{x})$}\Big)
=
\big|\big\{T \in \CT^\lessdot_\nu \mid \sqread(T) \in W \big\}\big|.
\end{align*}
\end{corollary}
Corollary~\ref{c Iplac perp} can also be proven directly from the fact \cite{GesselPPartition} that
$s_\nu(\mathbf{x}) = F^<_{C}(\mathbf{x})$ for any ordinary Knuth equivalence class  $C$ with insertion tableau of shape  $\nu$,
and a standardization argument;
however, we have presented it this way to parallel Corollary~\ref{c Ikron perp}  and to give a sense of the comparative strengths of Theorems \ref{t J intro} and \ref{t J plac}.

We also have the (partial) analog of Corollary~\ref{c intro main} for the  $\prec$-colored plactic algebra.  
\begin{corollary} \label{c main plac}
For any partitions $\lambda, \nu$ of $n$ and $d \leq n$,
\begin{align*}
\Big(\text{the coefficient of  $s_\nu(\mathbf{x})$ in  $F^\prec_{\CYW_{\lambda, d}}(\mathbf{x})$}\Big)
=
\big|\big\{T \in \CT^\prec_\nu \mid \sqread(T) \in \CYW_{\lambda,d} \big\}\big|.
\end{align*}
\end{corollary}
\begin{proof}
This follows from Corollary~\ref{c Iplac perp} with  $W = \CYW_{\lambda, d}$, provided we verify $\sum_{\e{w} \in W}\e{w} \in  (\Iplacord{\prec})^\perp$.
We must check that if colored words $\e{v}$ and  $\e{w}$ differ by a single application of one of the relations
\eqref{e plac rel knuth1}--\eqref{e plac rel knuth4b},
then neither or both belong to  $\CYW_{\lambda,d}$.
This holds in the case when $\e{v}$ and $\e{w}$ differ by a  relation involving three barred letters or three unbarred letters because
the set of Yamanouchi words of content  $\lambda$ is an ordinary Knuth equivalence class.
On the other hand, if $\e{v}$ and $\e{w}$ differ by a  relation involving three letters not all barred nor all unbarred,
then the smallest letter (for $\prec$) must be unbarred and the largest must be barred, hence swapping these two letters
preserves the property of being in $\CYW_{\lambda, d}$.
\end{proof}

Note that the order  $\prec$ is important here since it is not always true that $\sum_{\e{w} \in \CYW_{\lambda, d}} \e{w} \in (\Iplacord{\lessdot})^\perp$ for shuffle orders $\lessdot$ other than $\prec$ (see the discussion after Corollary~\ref{c intro main}).


\subsection{Conversion}\label{ss conversion mysterious}
Conversion is an operation on colored tableaux that
gives a bijection between the sets $CT^<_\nu$ and  $CT^\prec_\nu$
(see Definition~\ref{d conversion} below).
We first state a
strengthening of Corollary~\ref{c intro main}
 and then relate this  to conversion.
We believe that the results here are the first glimpses of a
deep connection between conversion and the $\mathfrak{J}^\lessdot_\nu(\mathbf{u})$ for different orders  $\lessdot$.

A set of colored words  $W$ is \emph{shuffle closed}
 if for every $\e{w} \in W$, any word obtained from  $\e{w}$ by swapping a barred letter with an adjacent unbarred letter also lies in  $W$.

\begin{corollary} \label{c mysterious}
For any shuffle closed set of colored words $W$ such that \mbox{$\sum_{\e{w} \in W} \e{w} \in (\Ikron)^\perp$},
there holds
\begin{align}
& \big|\big\{T \in \CT_\nu^< \mid \sqread(T) \in W\big\}\big| \label{e c myst1}\\
& = \text{the coefficient of  $s_\nu(\mathbf{x})$ in  $F^<_W(\mathbf{x})$} \\ 
& = \text{the coefficient of  $s_\nu(\mathbf{x})$ in  $F^\prec_W(\mathbf{x})$} \\ 
& = \big|\big\{T \in \CT_\nu^\prec \mid \sqread(T) \in W\big\}\big|.  \label{e c myst4}
\end{align}
\end{corollary}

\begin{remark}\label{r shuffle closed}
Here is a useful way to rephrase the condition on $W$ in Corollary~\ref{c mysterious}.

Let $\U/J$ be the quotient of  $\U$ by the relations \eqref{e plac nat rel knuth3}--\eqref{e plac nat rel knuth4b}, \eqref{e kron rel far commute}, and
\begin{alignat}{3}
&\e{xz = zx}, \qquad && x \in \A_\varnothing,z \in \A_\crcempty.  \label{e rel shuffle closed 0}
\end{alignat}
(We do not require  $x < z$ in \eqref{e rel shuffle closed 0}.)
Note that $\U/J$ is a monoid algebra and  $J$ is homogeneous,  hence
the space  $J^\perp$ has  $\ZZ$-basis given by  $\sum_{\e{w} \in C} \e{w}$, as  $C$ ranges over equivalence classes of colored words  modulo $J$.
It is easy to show that
\begin{align*}
{\small \big(\text{$W$ is shuffle closed and $\textstyle \sum_{\e{w} \in W} \e{w} \in (\Ikron)^\perp$}\big) \iff \text{$W$ is a union of equivalence classes mod $J$.}}
\end{align*}
\end{remark}

\begin{proof}[Proof of Corollary~\ref{c mysterious}]
First note that by Remark \ref{r shuffle closed} and the fact $\Iplacord{\prec} \subseteq J$,  it follows that \mbox{$\sum_{\e{w} \in W} \e{w} \in (\Iplacord{\prec})^\perp$} as well.
Hence the third equality in the statement is by Corollary~\ref{c Iplac perp}.  The first equality is by Corollary~\ref{c Ikron perp} and the second is by
Proposition~\ref{p word conversion QDes}.
\end{proof}


In light of Corollary~\ref{c mysterious}, it is natural to conjecture that the sets  of tableaux appearing in \eqref{e c myst1} and \eqref{e c myst4}
can be obtained from each other by conversion.  Theorem \ref{cj mysterious} shows that this is indeed the case;
this is a generalization of \cite[Lemma 3.1]{Ricky} (which is Corollary~\ref{c intro main}) and its proof uses conversion and is similar to that of \cite[Lemma 3.1]{Ricky}.
We do not know how to prove this using noncommutative super Schur function machinery since this machinery does not easily lend itself
to bijective results.
We suspect that there is something deeper lurking here that would connect the $\mathfrak{J}^\lessdot_\nu(\mathbf{u})$ for different orders  $\lessdot$ (considered in different algebras depending on  $\lessdot$) via something like conversion.

\begin{definition}[Conversion]
\label{d conversion}
Suppose $\lessdot$ and $\lessdot'$ are shuffle orders on $\A$ that are identical except for the order of $\noe{b}$ and $\noe{\crc{a}}$, say $\noe{b} \lessdot \noe{\crc{a}}$ but $\noe{\crc{a}} \lessdot' \noe{b}$. Then there is a natural bijection between  $\lessdot$-colored tableaux and  $\lessdot'$-colored tableaux via a process called \emph{conversion}, introduced by Haiman \cite{Hmixed} (see also \cite{BSSswitching}, \cite[\textsection2.3]{Ricky}, \cite[\textsection2.6]{BHook}).

Let $T_\lessdot$ be a  $\lessdot$-colored tableau, so that the boxes containing $\noe{b}$ or $\noe{\crc{a}}$ form a ribbon. For each component of the ribbon, there is a unique way to refill it with the same number of $\noe{b}$'s and $\noe{\crc{a}}$'s in a way that is compatible with $\noe{\crc{a}}\lessdot' \noe{b}$. (If the northeast corner contains $\noe{b}$, then shift each $\noe{b}$ to the bottom of its column; if the northeast corner contains $\noe{\crc{a}}$, shift each $\noe{b}$ once to the right within its row. See Figure~\ref{fig-ribbon}.)
Replacing each component in this manner yields a  $\lessdot'$-colored tableau denoted $\convert{\lessdot \to \lessdot'}{T_\lessdot}$.

\newcommand{\yts}{\ytableaushort}
\ytableausetup{boxsize=1.3em}
\begin{figure}
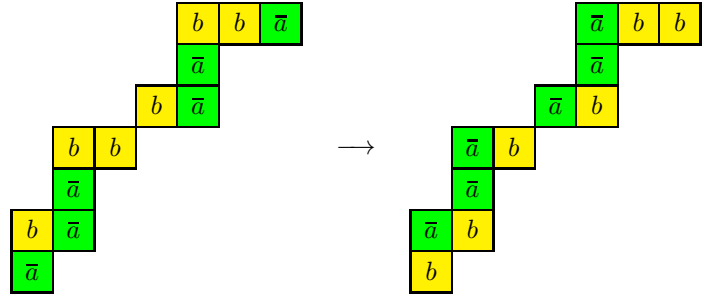

\[{\footnotesize
	\yts{\none\none\none\none {*(yellow)\noe{b}}{*(yellow)\noe{b}}{*(green)\noe{\crc{a}}},
		\none\none\none\none {*(green)\noe{\crc{a}}},
		\none\none\none {*(yellow)\noe{b}} {*(green)\noe{\crc{a}}},
		\none {*(yellow)\noe{b}}{*(yellow)\noe{b}},
		\none {*(green)\noe{\crc{a}}},
		{*(yellow)\noe{b}}{*(green)\noe{\crc{a}}},
		{*(green)\noe{\crc{a}}}}\quad  \xrightarrow{\hspace{3mm}} \quad
	\yts{\none\none\none\none {*(green)\noe{\crc{a}}}{*(yellow)\noe{b}}{*(yellow)\noe{b}},
		\none\none\none\none {*(green)\noe{\crc{a}}},
		\none\none\none {*(green)\noe{\crc{a}}}{*(yellow)\noe{b}},
		\none {*(green)\noe{\crc{a}}}{*(yellow)\noe{b}},
		\none {*(green)\noe{\crc{a}}},
		{*(green)\noe{\crc{a}}}{*(yellow)\noe{b}},
		{*(yellow)\noe{b}}}
}	
\]
	\caption{\label{fig-ribbon}Conversion from $\noe{b} \lessdot \noe{\protect\crc{a}}$ to $\noe{\protect\crc{a}} \lessdot' \noe{b}$. The ribbon above has two connected components with six boxes each.}
\end{figure}

Given any two shuffle orders $\lessdot$ and $\lessdot'$ on $\A$, one can be obtained from the other by repeatedly swapping the order of a consecutive unbarred letter and barred letter. Therefore, one can iterate the process above to convert any  $\lessdot$-colored tableau $T_\lessdot$ to a  $\lessdot'$-colored tableau denoted $\convert{\lessdot \to \lessdot'}{T_\lessdot}$.
The resulting tableau is well defined, i.e., it does not depend on the sequence of swaps
used to go from $\lessdot$ to $\lessdot'$ (see \cite{BSSswitching}).
\end{definition}

\begin{theorem}
\label{cj mysterious}
Suppose that a set of colored words $W$ is shuffle closed and
satisfies $\sum_{\e{w} \in W} \e{w} \in (\Ikron)^\perp$.
Let $\CT_\nu^\prec(W)$ and $\CT_\nu^<(W)$ denote the
sets of colored tableaux \newline
$\{T \in \CT_\nu^\prec \mid \sqread(T) \in W\}$
and $\{T \in \CT_\nu^< \mid \sqread(T) \in W\}$,
respectively.
Then
\begin{align}
\{\convert{\prec \to <}{T}\mid T \in \CT_\nu^\prec(W)\} = \CT_\nu^<(W).
\end{align}
\end{theorem}
\begin{proof}
Throughout the proof, we write $f\equiv g$ to mean that $f$ and $g$ are congruent  modulo $\Ikron$ (for $f,g \in \U$).

Fix some $T \in \CT_\nu^\prec(W)$, and for any shuffle order $\lessdot$ on $\A$, let $\e{w}^\lessdot = \sqread(\convert{\prec \to \lessdot}{T})$.
Let $\e{w_\varnothing^\lessdot}$ (resp. $\e{w_\crcempty^\lessdot}$) be the subsequence of $\e{w}^\lessdot$ consisting of
the unbarred (resp. barred) letters of $\e{w}^\lessdot$.
We need to show that  $\e{w}^<$ lies in  $W$.
Since $W$ is shuffle closed, we can instead show
$\e{w_\varnothing^<} \e{w_\crcempty^<} \in W$, and we know that
$\e{w_\varnothing^\prec} \e{w_\crcempty^\prec} \in W$.
Hence the result follows if we show
$\e{w_\varnothing^\prec} \equiv \e{w_\varnothing^<}$ and $\e{w_\crcempty^\prec} \equiv \e{w_\crcempty^<}$.  We prove the first congruence below; the other is similar.

We can convert from $\prec$ to $<$ by swapping the order of $\noe{\crc{1}}$ with $\noe{N}, \noe{N-1}, \dots, \noe{2}$,
then swapping $\noe{\crc{2}}$ with $\noe{N}, \noe{N-1}, \dots, \noe{3}$, and so forth. Each step performs a switch between consecutive letters of the form
\[\cdots \lessdot \noe{b-1} \lessdot \noe{b} \lessdot \noe{\crc a} \lessdot \noe{b+1} \lessdot \cdots \quad \to \quad \cdots \lessdot' \noe{b-1} \lessdot'  \noe{\crc a} \lessdot' \noe{b} \lessdot' \noe{b+1} \lessdot' \cdots.\]
Denote by $\e{d}^i$ the decreasing word consisting of the unbarred letters in the  $i$-th diagonal
of  $\convert{\prec \to \lessdot}{T}$,
along with $\e{b}$ (if it does not already appear);
denote by $\e{\hat{d}}^i$ the word obtained from $\e{d}^i$ by removing $\e{b}$. Thus $\e w_\varnothing^\lessdot$ and
$\e w_\varnothing^{\lessdot'}$ are products over all $i$ of either $\e{d}^i$ or $\e{\hat{d}}^i$.
An easy calculation shows that since $\e d^i$ contains $\e b$ and is decreasing, $\e b$ and $\e d^i$ commute in the (ordinary) plactic algebra
(defined in \textsection\ref{ss standardization});
hence they commute in $\U/\Ikron$ as well.

In the conversion of   $\convert{\prec \to \lessdot}{T}$ to
$\convert{\prec \to \lessdot'}{T}$, first consider the occurrences of $\e b$ that get shifted one box to the right.
We can handle such $\e b$ one row at a time from top to bottom, and each step looks as follows:
\newcommand{\yts}{\ytableaushort}
\ytableausetup{boxsize=1.3em}
\[{\tiny
	\yts{{*(yellow)\noe{b}}{*(yellow)\noe{b}}{*(yellow)\noe{b}}{*(yellow)\noe{b}} {*(green)\noe{\crc{a}}}} \quad \xrightarrow{\hspace{3mm}} \quad
	\yts{{*(green)\noe{\crc{a}}}{*(yellow)\noe{b}}{*(yellow)\noe{b}}{*(yellow)\noe{b}}{*(yellow)\noe{b}}}
}	
\]
We need to show that $\e d^i \e d^{i+1} \cdots \e d^{j-1} \e{\hat{d}}^j \equiv \e{\hat d}^i \e d^{i+1} \cdots \e d^{j-1} \e{d}^j$, where the boxes shown above lie in diagonals $i$ through $j$. Since the $\e b$ appearing in $\e d^i$ is the leftmost $\e b$ in its row, $\e d^i$ cannot contain $\e{b-1}$, so the
far commutation relations \eqref{e kron rel far commute} imply $\e d^i \equiv \e{\hat d}^i \e b$. Similarly, $\e d^j$ cannot contain $\e{b+1}$, so $\e d^j \equiv \e b \e{\hat d}^j$. Thus
\[\e d^i \e d^{i+1} \cdots \e d^{j-1} \e{\hat{d}}^j \equiv \e{\hat d}^i \e b \e d^{i+1} \cdots \e d^{j-1} \e{\hat d}^j  \equiv \e{\hat d}^i \e d^{i+1} \cdots \e d^{j-1}\e b \e{\hat d}^j \equiv  \e{\hat d}^i \e d^{i+1} \cdots \e d^{j-1} \e{d}^j.\]

Now consider, in the conversion of   $\convert{\prec \to \lessdot}{T}$ to
$\convert{\prec \to \lessdot'}{T}$, the occurrences of $\e b$ that get shifted down.
If we handle such $\e{b}$ one at a time from top to bottom, then each step looks as follows:
\ytableausetup{boxsize=1.3em}
\[{\tiny
	\yts{\none\none\none\none {*(yellow)\noe{b}}{*(yellow)\noe{b}}{*(yellow)\noe{b}}{*(yellow)\noe{b}}{*(yellow)\noe{b}},
		\none\none\none\none {*(green)\noe{\crc{a}}},
		\none\none\none\none {*(green){{\raisebox{-0.9mm}{$\vdots$}}}},
		\none\none\none\none {*(green)\noe{\crc{a}}},
		{*(yellow)\noe{b}}{*(yellow)\noe{b}}{*(yellow)\noe{b}}{*(yellow)\noe{b}} {*(green)\noe{\crc{a}}}} \quad \xrightarrow{\hspace{3mm}} \quad
	\yts{\none\none\none\none {*(green)\noe{\crc{a}}}{*(yellow)\noe{b}}{*(yellow)\noe{b}}{*(yellow)\noe{b}}{*(yellow)\noe{b}},
		\none\none\none\none {*(green)\noe{\crc{a}}},
		\none\none\none\none {*(green){{\raisebox{-0.9mm}{$\vdots$}}}},
		\none\none\none\none {*(green)\noe{\crc{a}}},
		{*(yellow)\noe{b}}{*(yellow)\noe{b}}{*(yellow)\noe{b}}{*(yellow)\noe{b}}{*(yellow)\noe{b}}}
}	
\]
Let $i \le j < k \le l$ be the indices such that the boxes shown above lie in diagonals $i$ through  $l$, and  $j$ and  $k$ are the diagonals of the
boxes whose entries change upon conversion.
We need to show that
\[\e d^i\cdots \e d^{j-1} \e{\hat d}^j\cdots \e{\hat d}^{k-1} \e d^k \cdots \e d^\ell \equiv
\e d^i\cdots \e{d}^j \e{\hat d}^{j+1}\cdots \e{\hat d}^k \e d^{k+1}\cdots \e d^\ell.\]
As in the previous case, $\e d^i \equiv \e{\hat d}^i \e b$, $\e d^j \equiv \e b \e{\hat d}^j$, $\e d^k \equiv \e{\hat d}^k \e b$, and $\e d^\ell \equiv \e b \e{\hat d}^\ell$. Moreover, the diagonals $j+1, \dots, k-1$ cannot contain $\e{b-1}$ or $\e{b+1}$, so $\e b$ commutes with $\e{\hat d}^{j+1}, \dots, \e{\hat d}^{k-1}$ in  $\U/\Ikron$.
Thus, assuming $i < j < k < l$ (the other cases are similar and easier), we have
\begin{align*}
\e d^i\cdots \e d^{j-1} \e{\hat d}^j\cdots \e{\hat d}^{k-1} \e d^k \cdots \e d^\ell &\equiv
\e {\hat d}^i\e b \e d^{i+1}\cdots \e d^{j-1} \e{\hat d}^j\cdots \e{\hat d}^{k-1} \e d^k \cdots \e d^{\ell-1} \e b\e{\hat d}^\ell\\
&\equiv \e{\hat d}^i\e d^{i+1}\cdots \e d^{j-1} \e b\e{\hat d}^j \cdots \e{\hat d}^{k-1} \e b\e d^k \cdots \e d^{\ell-1} \e{\hat d}^\ell\\
&\equiv \e{\hat d}^i\e d^{i+1}\cdots \e d^{j-1} \e{d}^j \e{\hat d^{j+1}}\cdots \e{\hat d}^{k-1} \e b\e{\hat d}^k \e b \e d^{k+1} \cdots \e d^{\ell-1} \e{\hat d}^\ell\\
&\equiv \e{\hat d}^i\e b\e d^{i+1}\cdots\e{d}^j \e{\hat d^{j+1}}\cdots \e{\hat d}^k \e d^{k+1} \cdots \e d^{\ell-1} \e b\e{\hat d}^\ell\\
&\equiv \e d^i\cdots \e{d}^j \e{\hat d^{j+1}}\cdots \e{\hat d}^k  \e d^{k+1} \cdots\e{d}^\ell,
\end{align*}
as desired.
\end{proof}

\begin{example}
As mentioned above, Theorem \ref{cj mysterious} is a generalization of Corollary~\ref{c intro main};
this is clear since the sets  $\CYW_{\lambda, d}$ are shuffle closed and satisfy  $\sum_{\e{w} \in \CYW_{\lambda, d}} \e{w} \in \Ikron^\perp$ (Proposition \ref{p CYW in kron perp}).
To see that Theorem \ref{cj mysterious} is indeed an interesting generalization,
consider the set  $\CYW_{(3,2), 2}$ depicted in Figure \ref{f CYW 32}.
The subset  $W$ of colored words in the first column is shuffle closed and satisfies $\sum_{\e{w} \in W} \e{w} \in \Ikron^\perp$;
the same goes for the colored words in the second and third columns, as well as the colored
words in the fourth and fifth columns.
None of these subsets is of the form  $\CYW_{\lambda,d}$.
\end{example}



\section{Some proofs} 
\label{s some proofs}
We now fill in the proofs of Proposition \ref{p intro sum of kron}, Proposition \ref{p intro word conversion QDes}, Corollary~\ref{c intro main}, and Theorem \ref{t J plac} in
\textsection\ref{ss proof1}, \textsection\ref{ss proof2}, \textsection\ref{ss proof3}, and \textsection\ref{ss standardization}, respectively.
The only genuinely new content here is the word conversion trick used to prove Proposition \ref{p intro word conversion QDes}.

\subsection{Proof of Proposition \ref{p intro sum of kron}}
\label{ss proof1}

Let  $T$ be a $\prec$-colored tableau and let $T_\varnothing$ (resp.  $T_\crcempty$) denote the subtableau of unbarred (resp. barred) letters of $T$;
let  $T_\crcempty^*$ denote the transpose of  $T_\crcempty$ with all bars removed.
Define  $T^\plainr$ to be the SSYT obtained by placing  $T_\crcempty^*$ above and to the right of  $T_\varnothing$ and then performing jeu de taquin slides to obtain a straight-shape tableau.
For example,
\[\text{if }\, T=
{\tiny\tableau{
1 & 1 & 2 & \crc{2}\\
2& \crc{1} \\
3& \crc{2}\\
\crc{1}
}},\text{ then }\quad
T_\varnothing = {\tiny\tableau{
1 & 1 & 2\\
2 \\
3
}},~
 T_\crcempty^* = {\tiny\tableau{
 &  &  & 1\\
 & 1 & 2\\
\\
2}}
,  \text{ and }\,
T^\plainr = {\tiny\tableau{
1 & 1 & 1 & 1 & 2\\
2 &2 & 2\\
3}}.
\]

Recall that a colored word is Yamanouchi if  $\e{w}^\plainr$ is Yamanouchi,
where $\e{w}^\plainr$ is the ordinary word formed from $\e{w}$ by shuffling the barred letters to the right, reversing this subword of barred letters and removing their bars.

\begin{proposition}\label{p CYT}
For a  $\prec$-colored tableau  $T$, the following are equivalent:
\begin{itemize}
\item $\sqread(T)$ is Yamanouchi,
\item  $T^\plainr$ is superstandard.
\end{itemize}
If these conditions hold, we say that  $T$ is a \emph{$\prec$-colored Yamanouchi tableau}.
\end{proposition}
\begin{proof}
The ordinary word  $\sqread(T)^\plainr$ consists of a reading word of  $T_\varnothing$ followed by a reading word of  $T_\crcempty^*$.
Hence the insertion tableau of $\sqread(T)^\plainr$ is  $T^\plainr$.  Since an ordinary word is Yamanouchi if and only if it has superstandard insertion tableau, the
result follows.
\end{proof}

Let  $\CYT_{\lambda,d}^\prec(\nu)$ denote the set of  $\prec$-colored Yamanouchi tableau of shape  $\nu$
having  $d$ barred letters and such that  $T^\plainr$ has content  $\lambda$.

\begin{proof}[Proof of Proposition \ref{p intro sum of kron}]
By an easy modification of \cite[Proposition 3.1]{BHook} to the conventions of this paper, 
we obtain the first equality below:
\[g_{\lambda\, \mu(d)\, \nu} + g_{\lambda\, \mu(d-1)\, \nu} = \big|\CYT_{\lambda,d}^\prec(\nu)\big| = \big|\big\{T \in \CT^\prec_\nu \mid \sqread(T) \in \CYW_{\lambda,d} \big\}\big|. \]
The second equality is by Proposition \ref{p CYT}.
Combined with Corollary~\ref{c main plac}, this proves
\[ g_{\lambda\, \mu(d)\, \nu} + g_{\lambda\, \mu(d-1)\, \nu} = \Big(\text{the coefficient of  $s_\nu(\mathbf{x})$ in  $F^\prec_{\CYW_{\lambda, d}}(\mathbf{x})$}\Big), \]
as desired.
\end{proof}


\subsection{Word conversion}
\label{ss proof2}
Here we prove Proposition \ref{p intro word conversion QDes} using an operator called word conversion.

\begin{definition}[Word conversion]
\label{d word conversion}
Let $\lessdot$ and $\lessdot'$ be two shuffle orders on $\A$ that are identical except for the order of $\crc{a}$ and $b$, say $b \lessdot \crc{a}$ and $\crc{a} \lessdot' b$ for some  $a, b \in [N]$.
For a colored word $\e{w = w_1\cdots w_t}$ consisting of $\e{\crc{a}}$'s and $\e{b}$'s, define
$\convert{\lessdot \to \lessdot'}{\e{w}}$ to be the result of cyclically rotating $\e{w}$ once to the right, i.e.
\[
\convert{\lessdot \to \lessdot'}{\e{w}} := \e{w_tw_1w_2\cdots w_{t-1}}.
\]
In general, for any colored word $\e{w}$, $\convert{\lessdot \to \lessdot'}{\e{w}}$ fixes the subword consisting of letters that are not $\e{\crc{a}}$ or $\e{b}$ and rotates each subword consisting of $\e{\crc{a}}$'s and $\e{b}$'s once to the right.
\end{definition}

\begin{example}
Let  $N=3$ and let $\prec$ and  $\prec'$  be the following two shuffle orders on $\A$
\[\arraycolsep=1.7pt\def\arraystretch{1.3}
\begin{array}{llllllllllll}
1 &\prec &2 &\prec &3 &\prec &\crc{1} &\prec &\crc{2} &\prec &\crc{3} \\
1 &\prec' &2 &\prec' &\crc{1} &\prec' &3  &\prec' &\crc{2} &\prec' &\crc{3}.
\end{array}\]
Below are two examples of colored words and the result obtained by applying the word conversion operator.
\[\begin{array}{ll}
\e{w} & \e{\crc{1}\crc{1}33 \crc{1}\crc{1}\crc{1}3\crc{1}} \\
\convert{\prec \to \prec'}{\e{w}} & \e{\crc{1}\crc{1}\crc{1}33 \crc{1}\crc{1}\crc{1}3}\\[2mm]
\e{v} & \e{33\crc{1}\crc{1}3 \crc{1}32\crc{2}1\crc{1}3} \\
\convert{\prec \to \prec'}{\e{v}} & \e{333\crc{1}\crc{1}3 \crc{1} 2\crc{2}13\crc{1}}
\end{array}\]
\end{example}

Word conversion is defined essentially to make the following property true.  Its proof is straightforward.
\begin{lemma}
\label{l word conversion}
Word conversion respects descent sets in the following sense:
\begin{align}
  \Des_{\lessdot'}( \convert{\lessdot \to \lessdot'}{\e{w}}) = \Des_{\lessdot}(\e{w}),
\end{align}
for shuffle orders $\lessdot$ and $\lessdot'$ on $\A$ as in Definition~\ref{d word conversion}.
\end{lemma}

Recall that a set of colored words  $W$ is shuffle closed
 if for every $\e{w} \in W$, any word obtained from  $\e{w}$ by swapping a barred letter with an adjacent unbarred letter also lies in  $W$.
%
Proposition \ref{p intro word conversion QDes} is a consequence of the following more general result.
\begin{proposition} \label{p word conversion QDes}
For any  shuffle closed set of colored words  $W$,  $F^\prec_W(\mathbf{x}) = F^<_W(\mathbf{x})$.
\end{proposition}
\begin{proof}
Let $\lessdot$ and $\lessdot'$ be two shuffle orders on $\A$ that differ by swapping a single covering
relation (as in Definition~\ref{d word conversion}). 
The word conversion operator $\e{w} \mapsto \convert{\lessdot \to \lessdot'}{\e{w}}$ is an involution on colored words.
In fact, it is an involution on any shuffle closed set.  Together with Lemma \ref{l word conversion}, this yields
\begin{align*}
F^\prec_W(\mathbf{x}) = \sum_{\e{w} \in W} Q_{\Des_{\lessdot'}(\e{w})}(\mathbf{x}) =
\sum_{\e{w} \in W} Q_{\Des_\lessdot(\e{w})}(\mathbf{x})  = F^<_W(\mathbf{x}).
\end{align*}
Applying this repeatedly, converting from the big bar order  $\prec$ to the natural order  $<$ by swapping one covering relation at a time, gives the result.
\end{proof}

\subsection{Proof of Corollary~\ref{c intro main}}
\label{ss proof3}

Aside from the proof of the main theorem (Theorem \ref{t J intro}),
the following fact is all that remains to prove Corollary~\ref{c intro main}.
\begin{proposition}\label{p CYW in kron perp}
There holds $\sum_{\e{w} \in \CYW_{\lambda, d}} \e{w} \in \Ikron^\perp$.
\end{proposition}
\begin{proof}
%
Recall Remark \ref{r shuffle closed} and let $J$ be as in the remark.
It suffices to show that  $\CYW_{\lambda,d}$ is a union of equivalence classes modulo  $J$.
This amounts to showing that if  $\e{w}$ and  $\e{w'}$ differ by a single application of one of the relations \eqref{e plac nat rel knuth3}--\eqref{e plac nat rel knuth4b}, \eqref{e kron rel far commute}, \eqref{e rel shuffle closed 0},
then  $\e{w} \in \CYW_{\lambda,d} \iff \e{w'} \in \CYW_{\lambda,d}$.
Since  $\CYW_{\lambda,d}$ is shuffle closed,
the necessary result holds for  $\e{w}$ and  $\e{w'}$ that differ by \eqref{e rel shuffle closed 0} or by \eqref{e plac nat rel knuth3}--\eqref{e plac nat rel knuth4b},
\eqref{e kron rel far commute}
in the case that exactly one of the two letters involved is barred.
Since Knuth transformations preserve whether an ordinary word is Yamanouchi,
the necessary result holds if $\e{w}$ and  $\e{w'}$ differ by \eqref{e plac nat rel knuth3}--\eqref{e plac nat rel knuth4b}
and the two letters are both barred or both unbarred.
Finally, an ordinary  word's being Yamanouchi is unchanged by swapping two adjacent letters that differ by more than 1, which handles
the case that $\e{w}$ and  $\e{w'}$ differ by \eqref{e kron rel far commute} and the two letters involved are both barred or both unbarred.
\end{proof}

\subsection{Standardization}
\label{ss standardization}
We fill in the details of the standardization argument needed to prove Theorem \ref{t J plac}.
This material will also be used in \textsection\ref{ss comparison with results}.

In our discussion of standardization, we will work with the alphabet of unbarred letters and the ordinary plactic algebra.
Let  $\U_\varnothing$ denote the subalgebra of  $\U$ generated by  $u_a$ for  $a \in \A_\varnothing$.
Let  $\U_\varnothing/\Iplacord{\varnothing}$ be the \emph{plactic algebra}, the quotient of  $\U_\varnothing$ by the \emph{plactic relations}
\begin{alignat}{3}
&\e{acb} = \e{cab} \qquad &&\text{for } a,b,c \in \A_\varnothing,\, a < b  < c,\label{e plac1} \\
&\e{bac} = \e{bca} \qquad &&\text{for } a,b,c \in \A_\varnothing,\, a  < b < c,\label{e plac2} \\
&\e{bba} = \e{bab} \qquad &&\text{for } a, b \in \A_{\varnothing},\, a < b, \label{e plac3} \\
&\e{cbb} = \e{bcb} \qquad &&\text{for } b, c \in \A_{\varnothing},\, b < c. \label{e plac4}
\end{alignat}
Equivalence classes of ordinary words modulo  $\Iplacord{\varnothing}$ are known as
\emph{Knuth equivalence classes}.
The  equivalence classes of colored words modulo
 $\Iplacord{\lessdot}$ we call \emph{$\lessdot$-colored plactic equivalence classes}.

For a colored word  $\e{w}$, the \emph{colored content} of $\e{w}$
is the pair of weak compositions $\beta = (\beta^\varnothing, \beta^\crcempty)$
such that  $\beta^\varnothing_a$ is the number of  $\e{a}$'s in $\e{w}$ and
$\beta^\crcempty_a$ is the number of  $\e{\crc{a}}$'s in $\e{w}$ for all \mbox{$a \in [N]$}.
For example, if $N=3$,
then the colored content of $\e{w = 2  \crc{1}  \crc{2}  1  \crc{3}  \crc{1}  2  1}$
is  $((2,2,0), (2,1,1))$.
For a colored word  $\e{w}$ of colored content $\beta = (\beta^\varnothing, \beta^\crcempty)$,
we set  $|\beta| = \sum_{a = 1}^N (\beta^\varnothing_a + \beta^\crcempty_a)$.
For a  $\lessdot$-colored tableau  $T$, the \emph{colored content} of $T$ is the colored content
of  $\creading(T)$. 

The relations of  $\U/\Iplacord{\lessdot}$ are colored content-preserving, so there is a  $\ZZ$-module decomposition
$\U / \Iplacord{\lessdot} \cong \bigoplus_{\beta} (\U/ \Iplacord{\lessdot})_{\beta}$,
where  $(\U/ \Iplacord{\lessdot})_{\beta}$ denotes the  $\ZZ$-span of the colored words of colored content  $\beta$ in the algebra  $\U/ \Iplacord{\lessdot}$.

\begin{definition}[Standardization]\label{d stand}
For a colored word $\e{w}$, the \emph{$\lessdot$-standardization of  $\e{w}$}, denoted $\e{w}^{\stand^\lessdot}$,
is the permutation obtained from $\e{w}$ by first relabeling, from left to right (resp. right to left), the occurrences of the smallest (for  $\lessdot$) letter in $\e{w}$ by  $1,\ldots,k$
if this letter is unbarred (resp. barred),
then relabeling the occurrences of the next smallest letter of $\e{w}$ by $k+1,\ldots,k+k'$,
and so on.

The standardization of a  $\lessdot$-colored tableau $T$, denoted $T^{\stand^\lessdot}$, is defined as for colored words, except that barred letters are relabeled from top to bottom and unbarred letters from left to right.
\end{definition}

For a colored content  $\beta$, define the set of permutations
\begin{align}\label{e def S beta}
\S(\beta) := \big\{\e{w}^{\stand^\lessdot} \mid \text{$\e{w}$ has colored content  $\beta$}\big\}.
\end{align}
There is a bijection
\begin{align}\label{e stand bij}
\big\{\text{colored words of colored content  $\beta$}\big\} \xrightarrow{\cong} \S(\beta), \quad \e{w} \mapsto \e{w}^{\stand^\lessdot}.
\end{align}
Also, it is well known that
\begin{align}\label{e KE class S beta}
\text{the set  $\S(\beta)$ is a union of Knuth equivalence classes}.
\end{align}

The colored plactic relations 
and the plactic relations 
 are compatible with standardization in the following sense:
\begin{proposition}\label{p plac stand}
For every colored content  $\beta$, the standardization map defines a $\ZZ$-module isomorphism
\begin{align}\label{e stand map}
(\U/\Iplacord{\lessdot})_\beta \xrightarrow{\cong} (\U_\varnothing /\Iplacord{\varnothing})[\beta], \quad \e{w} \mapsto \e{w}^{\stand^\lessdot},
\end{align}
where $(\U_\varnothing /\Iplacord{\varnothing})[\beta]$ denotes the  $\ZZ$-submodule of
$\U_\varnothing /\Iplacord{\varnothing}$ spanned by the words  $\S(\beta)$.
\end{proposition}
\begin{proof}
Fix a colored content $\beta$.
It is straightforward to check that if two colored words differ by a single application of one of the relations
\eqref{e plac rel knuth1}--\eqref{e plac rel knuth4b},
then their  $\lessdot$-standardizations differ by a single application of one of the relations \eqref{e plac1}, \eqref{e plac2};
also, if two permutations differ by a single application of one of the relations \eqref{e plac1}, \eqref{e plac2},
then either they both do not belong to  $\S(\beta)$, or they both belong to  $\S(\beta)$ and the inverse of the bijection
\eqref{e stand bij} takes these words to two colored words that differ by a single application of one of the relations \eqref{e plac rel knuth1}--\eqref{e plac rel knuth4b}.
Hence the  $\lessdot$-standardization of a $\lessdot$-colored plactic equivalence class is an (ordinary) Knuth equivalence class,
and under this correspondence,
the $\lessdot$-colored plactic equivalence classes of content  $\beta$ are in bijection with the Knuth equivalence classes
that partition $\S(\beta)$.
The result follows.
%
\end{proof}

%


Let  $\SSYT_\nu$ denote the set of semistandard Young tableaux (SSYT) of shape  $\nu$ (entries in the alphabet  $\A_\varnothing$).
We need the following fact relating standardization to colored tableaux.
\begin{proposition}\label{p stand tab}
The standardization map  $\e{w} \mapsto \e{w}^{\stand^\lessdot}$ defines a bijection
\begin{align*}
&\big\{\creading(T) \mid T \in \CT^\lessdot_\nu, \text{ $T$ has colored content $\beta$}\big\} \\
\xrightarrow{\cong} & \big\{\creading(T) \mid T \in \SSYT_\nu, \, \creading(T) \in \S(\beta)\big\}.
\end{align*}
\end{proposition}
\begin{proof}
It is clear that $\creading(T)^{\stand^\lessdot} = \creading(T^{\stand^\lessdot})$ and
 $T^{\stand^\lessdot} \in \SSYT_\nu$
 for any  $T \in \CT^\lessdot_\nu$, hence the bijection \eqref{e stand bij} restricts to a map between the given sets.
The inverse of \eqref{e stand bij} is computed on some  $\creading(T) \in \S(\beta)$  
by relabeling some of its decreasing subsequences with barred letters and some of its increasing subsequences with unbarred letters
to obtain the colored word  $\e{w}$ ($\creading(T) = \e{w}^{\stand^\lessdot}$);
 since the columns (resp. rows) of  $T$ correspond to decreasing (resp. increasing) subsequences of  $\creading(T)$,
it follows that $\e{w}$ is the column reading word of a $\lessdot$-colored tableau.
The result follows.
\end{proof}

%
%
%

We need to recall  the following definitions from \cite{FG}.
\begin{definition}\label{d normal noncom J}
The \emph{noncommutative elementary symmetric functions} are given by
\[
e^\varnothing_k=\sum_{\substack{a_1 \ge a_2 \ge \cdots \ge a_k \\ a_1,\dots,a_k \in \A_\varnothing}} u_{a_1} u_{a_2} \cdots u_{a_k} \ \in \U_\varnothing
\]
for any positive integer $k$; set $e^\varnothing_0=1$ and $e^\varnothing_k = 0$ for $k<0$.
Now define the \emph{noncommutative Schur functions}  $\mathfrak{J}^\varnothing_\nu$ exactly as in Definition~\ref{d noncommutative Schur functions}, except
with  $e^\varnothing_k$ in place of  $e_k(\mathbf{u})$.
\end{definition}

Our proof Theorem \ref{t J plac} is based on the following consequence of \cite[Lemma 3.2]{FG}.
\begin{theorem}[\cite{FG}]\label{t J FG}
For any partition $\nu$,
\begin{align*}
\mathfrak{J}^\varnothing_{\nu} =
\sum_{ T \in \SSYT_\nu} \creading(T) \qquad \text{in \, $\U_\varnothing/\Iplacord{\varnothing}$}.
\end{align*}
\end{theorem}

\begin{proof}[Proof of Theorem \ref{t J plac}]
Consider $\mathfrak{J}^\varnothing_{\nu}$ as an element of $\U_\varnothing$, written as the
signed sum of words obtained by
expanding the sum of products of  $e^\varnothing_k$'s in its definition.
Let $\mathfrak{J}^\varnothing_{\nu}[\beta]$ denote the result of restricting this signed sum to the permutations
$\S(\beta)$ (defined in \eqref{e def S beta}).
It follows from \eqref{e KE class S beta}
that
the image of $\mathfrak{J}^\varnothing_{\nu}[\beta]$ in $\U_\varnothing/\Iplacord{\varnothing}$
depends only on the image of
$\mathfrak{J}^\varnothing_{\nu}$ in  $\U_\varnothing/\Iplacord{\varnothing}$.
Moreover, fact \eqref{e KE class S beta} and Theorem \ref{t J FG} yield
\begin{align}\label{e brak beta}
\mathfrak{J}^\varnothing_{\nu}[\beta] =
\sum_{\substack{T \in \SSYT_\nu \\ \text{$\creading(T) \in \S(\beta)$}}} \creading(T) \qquad \text{in \, $\U_\varnothing/\Iplacord{\varnothing}$}.
\end{align}

By the decomposition $\U / \Iplacord{\lessdot} \cong \bigoplus_{\beta} (\U/ \Iplacord{\lessdot})_{\beta}$, we can write the noncommutative super
Schur function  $\mathfrak{J}^\lessdot_{\nu}(\mathbf{u})$ uniquely as a sum
$\mathfrak{J}^\lessdot_{\nu}(\mathbf{u}) =
\sum_\beta (\mathfrak{J}^\lessdot_{\nu}(\mathbf{u}))_\beta,$ for
$(\mathfrak{J}^\lessdot_{\nu}(\mathbf{u}))_\beta \in (\U/ \Iplacord{\lessdot})_{\beta}.$
Now note that a concatenation of colored words of the form  $\e{z_1\cdots z_t}$,  $z_1 \gedotcol z_2 \gedotcol \cdots \gedotcol z_t$, standardizes to a word that is a
concatenation of strictly decreasing (ordinary) words.
Using this, one checks that
the isomorphism \eqref{e stand map} takes $(\mathfrak{J}^\lessdot_{\nu}(\mathbf{u}))_\beta$
to  $\mathfrak{J}^\varnothing_{\nu}[\beta]$.
Applying the inverse of this isomorphism to both sides of \eqref{e brak beta}
yields (by Proposition \ref{p stand tab})
\begin{align*}
(\mathfrak{J}^\lessdot_{\nu}(\mathbf{u}))_\beta =
\sum_{\substack{T \in \CT_{\nu}^{\lessdot} \\ \text{$T$ has colored content  $\beta$}}} \creading(T) \qquad \text{in \, $\U/\Iplacord{\lessdot}$}.
\end{align*}
Summing over  $\beta$, the theorem follows.
\end{proof}

\section{Reading words for $\mathfrak{J}_\nu(\mathbf{u})$ in $\U/\Ikron$}
\label{s reading}
Here we introduce new kinds of tableaux and reading words that arose naturally in our discovery of a monomial positive expression for
$\mathfrak{J}_\nu(\mathbf{u})$ in $\U/\Ikron$.
The main new feature of these objects is that they involve posets obtained from posets of diagrams by adding a small number of covering relations.
This section has much in common with \cite[\textsection3]{BLamLLT}, but there are also important differences.


\subsection{Diagrams and tableaux}\label{ss diagram}

A \emph{diagram} or \emph{shape} is a finite subset of $\ZZ_{\ge 1}\times\ZZ_{\ge 1}$. A diagram is drawn as a set of boxes in the plane with the English (matrix-style) convention so that row (resp. column) labels start with 1 and increase from north to south (resp. west to east).

A partition $\lambda$ of $n$ is a weakly decreasing sequence $ (\lambda_1, \ldots, \lambda_l)$ of nonnegative integers that sum to $n$.
The \emph{shape} of  $\lambda$ is the subset
$\{(r,c) \mid r \in [l], \ c \in [\lambda_r]\}$
of $\ZZ_{\geq 1} \times \ZZ_{\geq 1}$.
Write $\mu \subseteq \lambda$ if the shape of $\mu$ is contained in the shape of $\lambda$.
If $\mu \subseteq \lambda$, then $\lambda/\mu$ denotes the \emph{skew shape} obtained by removing the boxes of $\mu$ from the shape of $\lambda$.
The \emph{conjugate partition} $\lambda'$ of $\lambda$ is the partition whose shape is the transpose of the shape of $\lambda$.

We will make use of the following partial orders on $\ZZ_{\ge 1}\times\ZZ_{\ge 1}$:
\begin{align*}
\text{$(r,c)\le_{\searrsub}(r',c')$  whenever $r\le r'$ and $c\le c'$,}\\
\text{$(r,c)\le_{\nearrsub}(r',c')$  whenever $r\ge r'$ and $c\le c'$.}
\end{align*}
 It will occasionally be useful to think of diagrams  as posets for the order  $<_\searrsub$ or $<_\nearrsub$.

Let  $\theta$ be a diagram.
A \emph{tableau of shape $\theta$} is the diagram $\theta$ together with a colored letter (an element of $\A$) in each of its boxes.
The \emph{size} of a tableau $T$, denoted  $|T|$, is the number of boxes of $T$, and $\sh(T)$ denotes the shape of $T$.
For a tableau $T$ and a set of boxes $S$ such that $S\subseteq\sh(T)$,
$T_S$ denotes the subtableau of $T$ obtained by restricting $T$ to the diagram $S$.
If $\beta$ is a box of $T$, then $\noe{T_\beta}$ denotes the entry of $T$ in $\beta$.
When it is clear, we will  occasionally identify a tableau entry with the box containing it.

If $T$ is a tableau, $\noe{x} \in \A$ a colored letter, and  $\beta = (r,c)$ is a box not belonging to $\sh(T)$, then $T \sqcup {\tiny \tableau{\noe{x}}}_{\, r,c}$ denotes the result of adding the box $\beta$ to $T$  and filling it with $\noe{x}$.


\subsection{Restricted shapes and restricted colored tableaux}

\begin{definition}
A \emph{restricted shape} is a lower order ideal of a partition diagram for the order $<_\nearrsub$.
We will  typically specify a restricted shape as follows:
for any weak composition  $\alpha= (\alpha_1,\ldots, \alpha_l)$, let  $\alpha'$ denote the diagram  $\{(r,c) \mid c \in [l], \ r \in [\alpha_c] \}$.
Now let $\lambda = (\lambda_1,\dots,\lambda_l)$ be a partition and $\alpha = (\alpha_1,\dots,\alpha_l)$ a weak  composition such that $0 \le \alpha_1 \le \cdots \le \alpha_{j'}$, $\alpha_1 < \lambda_1,\ \alpha_2 < \lambda_2,  \ldots, \alpha_{j'} < \lambda_{j'}$, and
 $\alpha_{j'+1} = \lambda_{j'+1}, \ldots, \alpha_{l} = \lambda_{l}$  for some  $j' \in \{0,1,\dots,l\}$.
Then the set difference of $\lambda'$ by $\alpha'$, denoted $\lambda' \setminus \alpha'$, is a restricted shape and any restricted shape can be written in this way.
\end{definition}

Note that, just as for skew shapes, different pairs $\lambda, \alpha$ may define the same restricted shape $\lambda' \setminus \alpha'$.
An example of a restricted shape is
\[ {\footnotesize (655444221)' \setminus (012223221)'} \ =\  \partition{~\\~&~\\~&~&~&~&~\\~&~&~&~&~&~\\~&~&~\\~}. \qquad\]

\begin{definition}\label{d RCT}
A \emph{restricted tableau} is a tableau whose shape is a restricted shape.
A \emph{restricted colored tableau} (RCT) is a restricted tableau such that
 each row and column is weakly increasing with respect to the natural order $<$, while the unbarred letters in each column and the barred letters in each row are strictly increasing.
\end{definition}
Hence a restricted colored tableau of partition shape  is the same as a  $<$-colored tableau (defined in \textsection\ref{ss main theorem}).

For example,
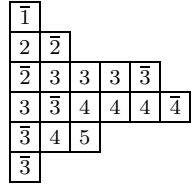
\begin{align*}
\tiny\tableau{
\crc{1}\\
2& \crc{2}\\
\crc{2}&3&3 & 3&\crc{3}\\
3 & \crc{3} & 4 & 4& 4 &\crc{4}\\
\crc{3} & 4 & 5\\
\crc{3} }
\end{align*}
is an RCT of shape $(655444)' \setminus (012223)'$.

Recall that $\noe{x} \lecol \noe{y}$ means that $\noe{x < y}$ or $\noe x$ and $\noe y$ are equal barred letters.
We write $\noe{x} \lerow \noe{y}$ to mean that $\noe{x < y}$ or $\noe{x}$ and $\noe{y}$ are equal unbarred letters.
Hence the condition for a restricted tableau to be a restricted colored tableau is exactly that if the entry $\noe{x}$ lies immediately west of $\noe{y}$, then $\noe{x \lerow y}$, and if the entry $\noe{x}$ lies immediately north of $\noe{y}$, then $\noe{x \lecol y}$.
Be aware that $\lecol$ and $\lerow$ are not partial orders.

The following basic fact will be used frequently in the proof of Theorem \ref{t J intro}.
\begin{proposition}
\label{e diagonal RCT}
If  $\noe{x} \lecol \noe{y} \lerow \noe{z}$, then  $\noe{x} < \noe{z}$ and
($\noe{x}=\noe{a}$ is unbarred  $\implies$  $\noe{z} \ge \noe{a+1}$) and ($\noe{z} = \noe{\crc{a+1}}$ is barred  $\implies$  $\noe{x} \le \noe{\crc{a}}$).
In particular, if  $\noe{d_1}, \noe{d_2}, \dots, \noe{d_t}$ is the diagonal of an RCT read in the  $\searr$ direction, then  $\noe{d_i} < \noe{d_{i+1}}$ and ($\noe{d_i} = \noe{a}$ is unbarred  $\implies$ $\noe{d_{i+1}} \ge \noe{a+1}$) and ($\noe{d_{i+1}} = \noe{\crc{a+1}}$ is barred  $\implies$  $\noe{d_i} \le \noe{\crc{a}}$).
\end{proposition}

\subsection{Arrow respecting reading words}
\label{ss Arrow respecting reading words}


\begin{definition}\label{d arrows}
An \emph{arrow subtableau} $S$ of an RCT $R$ is a subtableau of $R$ such that $\sh(S)$ is the intersection of  $\sh(R)$ with a rectangular shape,
and  $S$ is of the form
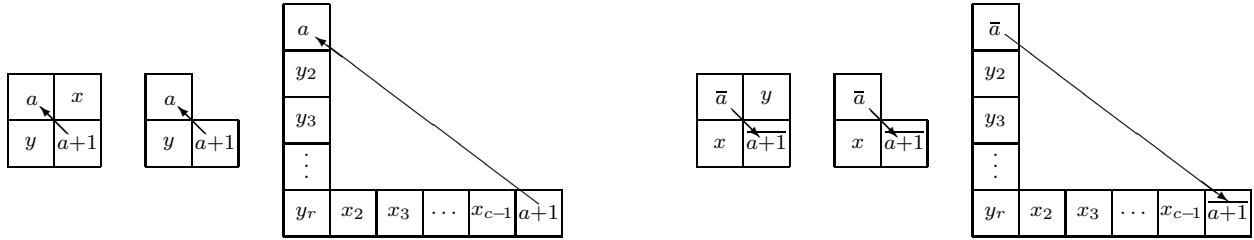
\setlength{\cellsize}{3.4ex}
\[\tiny
\tableau{{\color{black}\put(.8,-1.6){\vector(-1,1){.6}}}\put(-.1,-0.96){$\noe{a}$}&\noe{x}\\\noe{y}&\noe{a\text{+}1}}\qquad
\tableau{{\color{black}\put(.8,-1.6){\vector(-1,1){.6}}}\put(-.1,-0.96){$\noe{a}$}\\\noe{y}&\noe{a\text{+}1}}\qquad
\tableau{{\color{black}\put(5,-7.6){\vector(-4,3){4.8}}}\put(-.2,-3.96){$\noe{a}$}\\ \noe{y_2} \\ \noe{y_3} \\ {\vdots\atop} \\ \noe{y_r} & \noe{x_2} & \noe{x_3} & \ldots &\noe{x_{c\ngb 1}} & \noe{a\text{+}1}}\qquad\qquad\qquad
\tableau{{\color{black}\put(.25,-1.1){\vector(1,-1){.6}}}\put(-.1,-1){$\noe{\crc{a}}$}&\noe{y}\\\noe{x}&\noe{\crc{a\text{+}1}}}\qquad
\tableau{{\color{black}\put(.25,-1.1){\vector(1,-1){.6}}}\put(-.1,-1){$\noe{\crc{a}}$}\\\noe{x}&\noe{\crc{a\text{+}1}}} \qquad
\tableau{{\color{black}\put(.2,-3.9){\vector(4,-3){4.8}}}\put(-.15,-3.92){$\noe{\crc{a}}$}\\ \noe{y_2} \\ \noe{y_3} \\ {\vdots\atop} \\ \noe{y_r} & \noe{x_2} & \noe{x_3} & \ldots &\noe{x_{c\ngb 1}} & \noe{\crc{a\text{+}1}}}
\]
where $\noe{x},\noe{y}$, $\noe{x_2}, \dots, \noe{x_{c-1}}$, $\noe{y_2}, \dots, \noe{y_r} \in \A$ and $a \in [N-1]$; also, $r >2$, $c \ge 2$ in the third tableau and $r \ge 2$, $c > 2$ in the last tableau.
The first three of these are called \emph{$\nwarr$ arrow subtableaux} and the last three are \emph{$\searr$ arrow subtableaux}. Also, \emph{an arrow} of  $R$ is a directed edge between the two boxes of an arrow subtableau,  as indicated in the picture; we think of the arrows of  $R$ as the edges of a directed graph with vertex set the boxes of $R$.
\setlength{\cellsize}{2.2ex}
\end{definition}


\begin{definition}\label{d square reading word}
A \emph{reading word} of a tableau $R$ is a colored word $\e{w}$ consisting of the entries of $R$ such that for any two boxes $\beta$ and $\beta'$ of $R$ such that $\beta <_{\nearrsub} \beta'$, $\e{R_\beta}$ appears to the left of $\e{R_{\beta'}}$ in $\e{w}$.

An \emph{arrow respecting reading word} $\e{w}$ of an RCT $R$ is a reading word of $R$ such that for each arrow of $R$, the tail of the arrow appears to the left of the head of the arrow in $\e{w}$.

There is no reason to prefer one arrow respecting reading word over another (see Theorem \ref{t arrow respecting connected}), but it is useful to have notation for one such word.
The definition of the colored word $\sqread(T)$ (Definition~\ref{d sqread}) extends verbatim to the case   $T$ is an RCT.
Though  $\sqread(T)$ is not an arrow respecting reading word of  $T$ in general, it is in the case
 $T$ has partition shape and in all cases encountered in the proof of Theorem \ref{t J intro} (see Remark~\ref{r main theorem}~(b)).
\end{definition}

\begin{remark}
\label{r definition arrow is good}
Suppose  $S$ is the intersection of an RCT $R$ with a rectangular subtableau such that
(1) the northwesternmost box of $S$ is $\noe{a}$ ($a \in [N-1]$), (2) the southeasternmost box of $S$ is $\noe{a+1}$, and (3) the rectangular subtableau has more than two rows.
Then $S$ has the form of the third tableau in Definition~\ref{d arrows} because if not, there would be a box $\noe{\beta}$ immediately north of the box containing $\noe{x_2}$ and $\noe{\crc{a}} = \noe{y_{r-1}} \lerow R_{\beta} \lecol \noe{x_2} = \noe{a+1}$ yields a contradiction.
Note that this also shows that if an RCT  $R$ contains  an $S$ satisfying (1)--(3), then  $R$ cannot be completed to a  $<$-colored tableau.
(However, such tableaux still need to be considered in intermediate stages in the proofs of Theorems \ref{t arrow respecting connected} and \ref{t J intro}.)

Similarly, suppose $S$ is the intersection of an RCT $R$ with a rectangular subtableau such that
($1'$) the northwesternmost box of $S$ is $\noe{\crc{a}}$, ($2'$) the southeasternmost box of $S$ is $\noe{\crc{a+1}}$, and ($3'$) the rectangular subtableau has more than two columns.
Then $S$ has the form of the last tableau in Definition~\ref{d arrows}.
Just as above, it is also the case that if an RCT contains  an $S$ satisfying ($1'$)--($3'$), then it cannot be completed to a  $<$-colored tableau.

Also note that a given box is the tail of at most one arrow, but a single box can be the head of more than one arrow.
\end{remark}

\begin{example}\label{ex arrow respecting}
Here is an RCT drawn with its arrows:
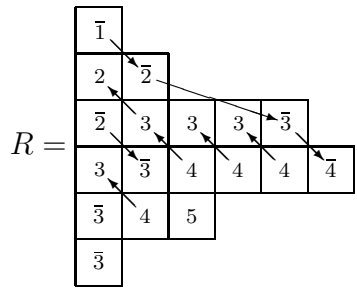
\setlength{\cellsize}{3.4ex}
\[R = \tiny\tableau{
{\color{black}\put(.25,-1 ){\vector(1,-1){0.6}}}\put(-.1,-0.99){$\crc{1}$}\\
{\color{black}\put(.8,-1.6){\vector(-1,1){0.6}}}\put(-.1,-0.94){$2$}&
{\color{black}\put(.25,-1 ){\vector(3,-1){2.6}}}\put(-.1,-1.06){{$\crc{2}$}}\\
{\color{black}\put(.25,-1 ){\vector(1,-1){0.6}}}\put(-.1,-0.94){$\crc{2}$}&{\color{black}\put(.8,-1.6 ){\vector(-1,1){0.6}}}\put(-.1,-0.94){$3$}&{\color{black}\put(.8,-1.6 ){\vector(-1,1){0.6}}}\put(-.1,-0.94){$3$} & {\color{black}\put(.8,-1.6 ){\vector(-1,1){0.6}}}\put(-.1,-0.94){$3$}&{\color{black}\put(.25,-1 ){\vector(1,-1){0.6}}}\put(-.1,-0.94){$\crc{3}$}\\
{\color{black}\put(.8,-1.6 ){\vector(-1,1){0.6}}}\put(-.1,-0.94){$3$} & \crc{3} & 4 & 4& 4 &\crc{4}\\
\crc{3}&4&5\\
\crc{3}}\]
\setlength{\cellsize}{2.2ex}

Of the following three reading words of $R$, the first two are arrow respecting, but the last is not.
\begin{align*}
\sqread(R) =~&\e{\crc{3}\, \crc{3}\, 4\, 3\, 5\, \crc{2}\, \crc{3}\, 4\, 3\, 2\, 4\, 3\, \crc{1}\, \crc{2}\, 4\, 3\, \crc{3}\, \crc{4}} \quad\text{\footnotesize (arrow respecting)}\\
& \e{\crc{3}\, \crc{3}\, 4\, 3\, 5\, \crc{2}\, \crc{3}\, 4\, 3\, 4\, 3\, 2\, \crc{1}\, \crc{2}\, 4\, 3\, \crc{3}\, \crc{4}}\quad\text{\footnotesize (arrow respecting)} \\
& \e{\crc{3}\, \crc{3}\, 4\, 3\, 5\, \crc{2}\, \crc{3}\, 3\, 2\, \crc{1}\, \crc{2}\, 4\, 4\, 3\, 4\, 3\, \crc{3}\, \crc{4}}\quad\text{\footnotesize (not arrow respecting)}
\end{align*}
\end{example}

A \emph{$\nearr$-maximal box} of a diagram is a box that is maximal for the order  $<_\nearrsub$.
A \emph{nontail removable box} of an RCT $R$ is a $\nearr$-maximal box of $R$ that is not the tail of an arrow of $R$.
The last letter of an arrow respecting reading word of $R$ must lie in a nontail removable box of $R$.
For example, the $\nearr$-maximal boxes of the $R$ in Example~\ref{ex arrow respecting} are $(1,1), (2,2), (3,5), (4,6),$ which contain the entries $\crc{1}, \crc{2}, \crc{3}, \crc{4}$, respectively.
The RCT $R$ has only the single nontail removable box $(4,6)$.


\subsection{Arrow respecting reading words in $\U/\Ikron$}
\label{ss Combinatorics of restricted colored tableaux}

We assemble some basic results about arrow respecting reading words and the images of these words in $\U/\Ikron$.
These are needed for the proof of Theorem \ref{t J intro}.

\begin{lemma}
\label{l exist no tail corner}
Every RCT has a nontail removable box.
\end{lemma}

\begin{proof}
Consider the $\nearr$-maximal boxes of an RCT $R$. The northernmost such box containing an unbarred letter (if it exists) is a nontail removable box, and the southernmost such box containing a barred letter (if it exists) is a nontail removable box.
Hence $R$ contains a nontail removable box.
\end{proof}

\begin{corollary}\label{c any end}
If $\beta$ is a nontail removable box of an RCT $R$, then there is an arrow respecting reading word of $R$ that ends in $\e{R_\beta}$.
\end{corollary}
\begin{proof}
This follows by induction on $|R|$ using Lemma \ref{l exist no tail corner}.
\end{proof}

The next result gives a natural way to associate an element of \ $\U/\Ikron$ to any RCT.
\begin{theorem}\label{t arrow respecting connected}
Any two arrow respecting reading words of an RCT are equal in \hspace{-.4mm} $\U/\Ikron$.
\end{theorem}
\begin{proof}
Throughout the proof, we write $f\equiv g$ to mean that $f$ and $g$ are congruent  modulo $\Ikron$ (for $f,g \in \U$).

Let  $R$ be an RCT and let $\R_R$ denote the graph with vertex set the arrow respecting reading words of $R$ and an edge between any two such words that differ by a single application of one of the relations \eqref{e kron rel far commute}, \eqref{e plac nat rel knuth3}, \eqref{e plac nat rel knuth4}.
We prove that  $\R_R$ is connected by induction on $|R|$.
Let $\beta_1, \ldots, \beta_t$ denote the nontail removable boxes of  $R$, labeled so that $\beta_1 <_\searrsub \beta_2 <_\searrsub \cdots <_\searrsub \beta_t$
($t \geq 1$ by Lemma \ref{l exist no tail corner}).
By induction, each $\R_{R-\beta_i}$ is connected.  Hence  the induced subgraph of $\R_R$  with  vertex set consisting of those words that end by reading the box $\beta_i$, call it $\R_{R-\beta_i} \e{R_{\beta_i}}$, is connected. Let $H$ be the graph (with $t$ vertices) obtained from  $\R_R$ by contracting the subgraphs  $\R_{R-\beta_i}\e{R_{\beta_i}}$.
We must show that  $H$ is connected.

To prove that  $H$ is connected, we will show that there is an edge between vertices  $i$ and  $i+1$ of  $H$  for each  $i$.
It follows from Proposition \ref{e diagonal RCT} that $R_{\beta_i} <  R_{\beta_{i+1}} \myDownarrow $.
We first consider the case $R_{\beta_i} < R_{\beta_{i+1}} \myDownarrow \myDownarrow$.
Let  $\e{w}$ be an arrow respecting reading word of  $R-\beta_i-\beta_{i+1}$.
Then  $\e{w R_{\beta_i} R_{\beta_{i+1}} \equiv \e{w} R_{\beta_{i+1}} R_{\beta_{i}}}$ by the far commutation relations \eqref{e kron rel far commute}, hence there is an edge between vertices  $i$ and  $i+1$ of  $H$.

Next consider the case $R_{\beta_i}= a, \, R_{\beta_{i+1}}= a+1$ for some $a \in [N-1]$.
By Proposition \ref{e diagonal RCT},  $\beta_i$ and  $\beta_{i+1}$ are consecutive  $\nearr$-maximal boxes of  $R$ i.e. there is no  $\nearr$-maximal box  $\beta$ of  $R$ such  that $\beta_i <_\searrsub \beta <_\searrsub  \beta_{i+1}$.
Let  $S$ be the subtableau of  $R$ consisting of the boxes  $\ge_\searrsub \beta_i$ and  $\le_\searrsub \beta_{i+1}$.
Since there is no arrow from $\beta_{i+1}$ to $\beta_i$ and  $\beta_i$ and  $\beta_{i+1}$ are  $\nearr$-maximal boxes of  $R$, by Definition~\ref{d arrows} the only possibility is that  $S$ has shape
 ${\tiny \tableau{~& & &\\ ~&\cdots&~&~}}$
($S$ has 2 rows and more than 2 columns).
Let $R' = R-\beta_i-\beta_{i+1}$ and let $\beta$ denote the box immediately west of $\beta_{i+1}$.
By the strictness conventions of RCT, we must have $\noe{R_{\beta}} = \noe{a+1}$.
We next claim that $\beta$ is not an arrow tail in  $R'$ (there could be an arrow from  $\beta$ to  $\beta_i$ in  $R$).
By the forms of the first three arrow subtableaux in Definition~\ref{d arrows} and since
$\beta_i \notin R'$, if there is an arrow from $\beta$ to some  $\beta'$ in  $R'$,
then  $\beta'$ must be strictly north of  $\beta_i$; but this implies  $R_{\beta'} \lecol R_{\beta_i} = a$, so there cannot be an arrow from  $\beta$ to  $\beta'$ in  $R'$.
Corollary~\ref{c any end} then implies that there is an arrow respecting reading word of $R'$ that ends in $\e{R_{\beta}}$; let $\e{w'R_\beta}$ be one such word.
Hence the congruence
\[\e{w ' R_{\beta} R_{\beta_i} R_{\beta_{i+1}} = w' (a+1) a (a+1) \equiv w' (a+1) (a+1) a = w' R_{\beta} R_{\beta_{i+1}} R_{\beta_{i}}} \]
obtained by applying the relation \eqref{e plac nat rel knuth3} exhibits an edge between vertices  $i$ and  $i+1$ of  $H$.
The case $R_{\beta_i}= \crc{a}, \, R_{\beta_{i+1}}= \crc{a+1}$ is handled in a similar way.
\end{proof}

\section{Proof of Theorem \ref{t J intro}}
\label{s proof of theorem}
Here we prove Theorem \ref{t J intro} by an inductive computation that involves flagged versions of the
noncommutative super Schur functions.  Lemma \ref{l lambda1 eq lambda2 commute} is the key computation in the algebra  $\U/\Ikron$
that makes the proof possible and is the main cause for the words  $\sqread(T)$ appearing in the statement of the theorem.
This section has much in common with \cite[\textsection4.3]{BLamLLT}, but there are also important differences.

Here is some notation we will use throughout this section:
for  $f,g \in \U$, we write $f\equiv g$ to mean that $f$ and $g$ are congruent modulo $\Ikron$.
We will work with the poset  $\{\crc{0}\} \sqcup \A$ which extends the natural order  $<$ on  $\A$ and where  $\crc{0} < x$ for all  $x \in \A$.
The notation  $x \myDownarrow$ was introduced in \eqref{e myDownarrow}; we recall this and introduce the similar notation  $x \mydownarrow$.  For any  $x \in \A$, define
\begin{align*}
\noe{x} \myDownarrow =
\begin{cases}
\noe{\crc{a-1}} &\text{ if $\noe{x = a}$, \, $a \in [N]$}, \\
\noe{a} &\text{ if $\noe{x = \crc{a}}$, \, $a \in [N]$}.
\end{cases}
\qquad
\noe{x} \mydownarrow =
\begin{cases}
\noe{\crc{a-1}} &\text{ if $\noe{x = a}$, \, $a \in [N]$}, \\
\noe{\crc{a}}   &\text{ if $\noe{x = \crc{a}}$, \, $a \in [N]$}.
\end{cases}
\end{align*}

\subsection{Noncommutative column-flagged super Schur functions}\label{ss noncommutative flagged schur}


Here we introduce a flagged generalization of the
noncommutative super Schur functions. This will be important for our inductive computation of
$\mathfrak{J}_{\nu}(\mathbf{u})$ in the proof of Theorem \ref{t J intro}.

For any subset $S \subseteq \A$ and positive integer $k$, there is a natural generalization of the functions  $e_k(\mathbf{u})$
given by
\[
e_k(S)=\sum_{\substack{z_1 \gecol z_2 \gecol \cdots \gecol z_k \\ z_1,\dots,z_k \in S}} u_{z_1} u_{z_2} \cdots u_{z_k} \, \in \U;
\]
set $e_0(S)=1$ and $e_{k}(S) = 0$ for $k<0$.
By the proof of  Proposition \ref{p es commute kron}, $e_k(S) e_l(S) = e_l(S) e_k(S)$ in  $\U/\Ikronknuth$ (and therefore in $\U/\Ikron$) for all $k$ and $l$.

For $x \in \{\crc{0}\} \sqcup \A$, define $\A_{\le x} = \{y \in \A \mid y \le x\}$ (thus  $\A_{\le \crc{0}} = \{\}$).
Given a weak composition
$\alpha=(\alpha_1,\dots,\alpha_l)$ and elements  $n_1,n_2,\dots, n_{l} \in \{\crc{0}\} \sqcup \A$, define the \emph{noncommutative column-flagged super Schur function} by
\begin{align}
J_{\alpha}(n_1,n_2,\dots, n_l)
:=\sum_{\pi\in \S_{l}}
\sgn(\pi) \, e_{\alpha_1+\pi(1)-1}(\A_{\le n_1}) e_{\alpha_2+\pi(2)-2}(\A_{\le n_2})\cdots e_{\alpha_{l}+\pi(l)-l}(\A_{\le n_l}) \, \in \U. \label{e flag schur}
\end{align}
We also use the  shorthand
\[J_\alpha^\mathbf{n} = J_\alpha(n_1,\ldots,n_l),\]
where $\mathbf{n}$ is the  $l$-tuple $(n_1,\dots,n_l)$.
These functions are related to the  $\mathfrak{J}_\nu(\mathbf{u})$ of Definition~\ref{d noncommutative Schur functions} by
$\mathfrak{J}_\nu(\mathbf{u}) = J_{\nu'}^{\crc{N}\,\crc{N}\,\cdots\,\crc{N}}$.


For colored words $\e{w^1}, \ldots, \e{w^{l-1}} \in \U$, we will also make use of the \emph{augmented noncommutative column-flagged super Schur functions}, given by
\begin{align*}
&J_{\alpha}(n_1\Jnot{w^1}n_2\Jnot{w^2}\cdots\Jnot{w^{l-1}}n_l) \\
&:=\sum_{\pi\in \S_{l}}
\sgn(\pi) \, e_{\alpha_1+\pi(1)-1}(\A_{\le n_1})\e{w^1} e_{\alpha_2+\pi(2)-2}(\A_{\le n_2})\e{w^2}\cdots  \e{w^{l-1}}e_{\alpha_{l}+\pi(l)-l}(\A_{\le n_l}) \, \in \U.
\end{align*}
If $\e{w^1}, \e{w^2}, \dots, \e{w^{j-1}}, \e{w^{j+1}}, \dots, \e{w^{l-1}}$ are empty and  $\e{w = w^j}$, we also use the shorthand
\[ J_\alpha^\mathbf{n}(\Jnotb{j}{w}) = J_{\alpha}(n_1\Jnot{w^1}n_2\Jnot{w^2}\cdots\Jnot{w^{l-1}}n_l).
\]

Note that  because the noncommutative super elementary symmetric functions commute in  $\U/\Ikron$,
\begin{align}\label{e elem sym Jswap}
J_\alpha^\mathbf{n} \equiv -J_{\alpha_1, \ldots, \alpha_{j-1},\alpha_{j+1}-1,\alpha_j+1,\ldots,\alpha_l}^\mathbf{n} \quad \text{whenever  $n_j = n_{j+1}$}.
\end{align}
In particular,
\begin{align}\label{e elem sym J0}
J_\alpha^\mathbf{n} \equiv 0 \quad \quad \quad \ \, \text{whenever  $\alpha_j = \alpha_{j+1}-1$ and  $n_j=n_{j+1}$.}
\end{align}
More generally, \eqref{e elem sym Jswap} and \eqref{e elem sym J0} hold for the augmented case provided $\e{w^j}$ is empty.

We will make frequent use of the following fact (keeping in mind $\A_{\le \crc{0}} = \{\}$):
\begin{align}\label{e ek induction}
e_k(\A_{\le x}) =&~\e{x} e_{k-1}(\A_{\le x \mydownarrow})+e_k(\A_{\le x\myDownarrow}) \qquad ~ \text{for  $x \in  \A$ and any integer $k$}.
\end{align}
Note that
\begin{align*}
e_k(\A_{\le \crc{0}}) =&
\begin{cases}
  1 & \text{ if } k = 0, \\
  0 & \text{otherwise.}
\end{cases} \notag
\end{align*}
We will often apply the identity \eqref{e ek induction} to  $J_\alpha^\mathbf{n}$ and its variants by expanding
$e_{\alpha_j+\pi(j)-j}(\A_{\le n_j})$ in \eqref{e flag schur} using \eqref{e ek induction} (so that \eqref{e ek induction} is applied once to each of the  $l!$ terms in the sum).
We refer to this as a \emph{$j$-expansion of  $J_\alpha^\mathbf{n}$} or simply a  \emph{$j$-expansion}.

\subsection{Proof of Theorem \ref{t J intro}}
After three preliminary results,
we state and prove a more technical version of Theorem~\ref{t J intro}, which involves computing $J_\nu^\mathbf{n}$ inductively by peeling off diagonals from the shape $\nu'$.

For  $\noe{x,y} \in \A$, let $[\e{x},\e{y}]$ denote the commutator $\e{xy - yx} \in \U$.

\begin{lemma}
\label{l decreasing}
Let $\noe{y,z, x_1, x_2,\dots, x_t} \in \A$ be letters satisfying $y \gecol x_1 \gecol x_2 \cdots \gecol x_t$ and
$\noe{y \leq z}$.  Suppose that  neither ($\noe{y = z} \myDownarrow$) nor ($\noe{y} = \noe{z}$,  $\noe{y,z \in \A_\crcempty}$) is true.
Then
\[\e{yx_1x_2\cdots x_t z = yzx_1x_2\cdots x_t} \quad \, \text{ in \hspace{-.4mm} $\U/\Ikron$}.\]
\end{lemma}
\begin{proof}
First suppose that $\noe{z}$ is unbarred. If no $\noe{x_i}$ is equal to $\noe{z-1}$, then the relations \eqref{e kron rel far commute},
\eqref{e plac nat rel knuth3}, and \eqref{e plac nat rel knuth4} imply
\begin{align*}
\e{yx_1x_2\cdots x_tz \equiv yx_1x_2\cdots x_{t-1}zx_t \equiv yx_1x_2\cdots x_{t-2}zx_{t-1}x_t \equiv \cdots \equiv yzx_1x_2\cdots x_t}.
\end{align*}
If $\noe{x_i = z-1}$, then the argument just given shows
\[\e{yx_1x_2\cdots x_tz \equiv yx_1\cdots x_i z x_{i+1}\cdots x_t}.\]
Since $\noe{z \geq y \geq x_1 \geq x_{i-1} > x_i}$, we must have $\noe{x_1 = x_2 = \cdots = x_{i-1}  = \crc{z-1}}$ and $\noe{y} \in \{\noe{\crc{z-1}, z}\}$. The case $\noe{y = \crc{z-1} = z\myDownarrow}$ is excluded, so $\noe{y=z}$. Now we compute
\begin{align*}
&\e{yx_1\cdots x_i z x_{i+1}\cdots x_t} \\
&=~~\e{yx_1\cdots x_{i-1}[x_i, z] x_{i+1}\cdots x_t} + \e{yx_1\cdots x_{i-1}z x_i x_{i+1}\cdots x_t} \\
&\equiv~~\e{y[x_i,z]x_1\cdots x_{i-1}x_{i+1}\cdots x_t} + \e{yx_1\cdots x_{i-1}z x_i x_{i+1}\cdots x_t} \\
&\equiv~~\e{yx_1\cdots x_{i-1}z x_i x_{i+1}\cdots x_t} \\
&\equiv~~\e{ yzx_1x_2\cdots x_{i-1}x_ix_{i+1}\cdots x_t },
\end{align*}
where the first congruence is by $i-1$ applications of the relation \eqref{e kron rel rotate12}, the second is by the relation \eqref{e plac nat rel knuth3} (which yields $\e{y[x_i,z]} = \e{z[z-1,z]} \equiv 0$), and the third is by the relations \eqref{e plac nat rel knuth3} and \eqref{e plac nat rel knuth4}.

If $\noe{z = \crc{a+1}}$ is a barred letter, then by the assumptions of the lemma, $\noe{y \leq \crc{a}}$, hence the relations \eqref{e kron rel far commute} and \eqref{e plac nat rel knuth4} imply
\[\e{yx_1x_2\cdots x_tz \equiv yx_1x_2\cdots x_{t-1}zx_t \equiv yx_1x_2\cdots x_{t-2}zx_{t-1}x_t \equiv \cdots \equiv yzx_1x_2\cdots x_t},\]
as desired.
\end{proof}

\begin{example}
In the case $\e{y = z = 4}$, $\e{x_1x_2\cdots x_t = \crc{3}\crc{3}3}$, Lemma \ref{l decreasing} and its proof yield
\[\e{4\crc{3}\crc{3}34 = 4\crc{3}\crc{3}[3,4] + 4\crc{3}\crc{3}43 \equiv 4[3,4]\crc{3}\crc{3} + 4\crc{3}\crc{3}43 \equiv 4\crc{3}\crc{3}43 \equiv 4\crc{3}4\crc{3}3 \equiv 44\crc{3}\crc{3}3},\]
where the congruences are by \eqref{e kron rel rotate12}, \eqref{e plac nat rel knuth3}, \eqref{e plac nat rel knuth4}, and \eqref{e plac nat rel knuth3}, respectively.
\end{example}

\begin{remark}
The assumptions on  $y$ and  $z$ in Lemma \ref{l decreasing} are necessary:
the lemma does not extend to the case $\noe{y = a, z = \crc{a}}$ since $\e{2\crc{1}\crc{2} \not\equiv 2\crc{2}\crc{1}}$,
nor to the case $\noe{y = \crc{a}, z = a+1}$ since $\e{\crc{2}23 \not\equiv \crc{2}32}$,
nor to the case $\noe{y = \crc{a}, z = \crc{a}}$ since $\e{\crc{2}2\crc{2} \not\equiv \crc{2}\crc{2}2}$.
\end{remark}

The following lemma  is key to the proof of Theorem \ref{t J intro}.
\begin{lemma}\label{l lambda1 eq lambda2 commute}
Let  $m \in [N-1]$ and set $x = m+1$ (an unbarred letter). Then
\[
J_{(a,a)}(\crc{m} \Jnot{\e{x}} \crc{m}) = \e{x} J_{(a,a)}(\crc{m}, \crc{m}) \quad \, \text{ in \hspace{-.4mm} $\U/\Ikron$}.
\]
More generally,
if  $\alpha$ is a weak composition satisfying  $\alpha_j = \alpha_{j+1}$ and  $\mathbf{n} = (n_1,\ldots, n_l)$ with  $n_j = n_{j+1} = \crc{m}$, then
\[
J_\alpha^\mathbf{n}(\Jnotb{j}{x}) = J_\alpha^\mathbf{n}(\Jnotb{j-1}{x}) \quad \, \text{ in \hspace{-.4mm} $\U/\Ikron$}.
\]
\end{lemma}
\begin{proof}
Since the proofs of both statements are essentially the same, we prove only the first to avoid extra notation. We compute
\begin{align*}
0& \equiv J_{(a, a+1)}(m+1, m+1)\\
&=\e{x}J_{(a-1, a+1)}(\crc{m},m+1)+
J_{(a, a+1)}(\crc{m},m+1)\\
&=\e{x}J_{(a-1, a)}(\crc{m}  \Jnot{x}  \crc{m})+
\e{x}J_{(a-1, a+1)}(\crc{m},\crc{m})+
J_{(a, a)}(\crc{m}   \Jnot{x}  \crc{m})+
J_{(a, a+1)}(\crc{m}, \crc{m})\\
& \equiv
\e{x}J_{(a-1, a)}(\crc{m}  \Jnot{x}  \crc{m}) +
\e{x}J_{(a-1, a+1)}(\crc{m},\crc{m})+
J_{(a, a)}(\crc{m}   \Jnot{x}  \crc{m}) \\
& \equiv
\e{x}\e{x}J_{(a-1, a)}(\crc{m}, \crc{m}) +
\e{x}J_{(a-1, a+1)}(\crc{m},\crc{m})+
J_{(a, a)}(\crc{m}   \Jnot{x}  \crc{m}) \\
& \equiv
\e{x}J_{(a-1, a+1)}(\crc{m},\crc{m})+
J_{(a, a)}(\crc{m}   \Jnot{x}  \crc{m}) \\
&\equiv
-\e{x}J_{(a, a)}(\crc{m},\crc{m})+
J_{(a, a)}(\crc{m}   \Jnot{x}  \crc{m}).
\end{align*}
The first, second, and fourth congruences are by \eqref{e elem sym J0}.  The equalities are a 1-expansion and a 2-expansion (see \textsection\ref{ss noncommutative flagged schur} for notation). Lemma \ref{l decreasing} implies that  $\e{x} e_k(\A_{\leq \crc{m}}) \e{x} \equiv \e{x}\e{x}e_k(\A_{\leq \crc{m}})$ for any  $k$, hence the third congruence.
The last congruence is by \eqref{e elem sym Jswap}.
\end{proof}

\begin{corollary}\label{c aaa+1 commute lam}
Maintain the notation of Lemma~\ref{l lambda1 eq lambda2 commute} and assume in addition that
\begin{align}\label{e word condition}
\text{$\e{w}$ is a colored  word such that $\e{w} = \e{x w'}$ and the letters of the word $\e{w'}$ are $\ge m+2$.}
\end{align}
Then
\[
J_{\alpha}(n_1,\dots, n_j  \Jnot{w}   n_{j+1},\dots) = J_{\alpha}(n_1,\dots,  n_{j-1}   \Jnot{x}   n_j  \Jnot{w'}   n_{j+1},\dots) \quad \, \text{ in \hspace{-.4mm} $\U/\Ikron$}.
\]
\end{corollary}
\begin{proof}
We compute (modulo  $\Ikron$)
\begin{align*}
& J_{\alpha}(n_1,\dots, n_j  \Jnot{w}   n_{j+1},\dots) \\
&\equiv J_{\alpha}(n_1,\dots,  n_{j-1}, n_j   \Jnot{x}   n_{j+1}   \Jnot{w'}   \dots) \\
&\equiv J_{\alpha}(n_1,\dots,  n_{j-1}   \Jnot{x}   n_j, n_{j+1}   \Jnot{w'}   \dots) \\
&\equiv J_{\alpha}(n_1,\dots,  n_{j-1}   \Jnot{x}   n_j   \Jnot{w'}   n_{j+1},\dots).
\end{align*}
The first and third congruences are by the far commutation relations \eqref{e kron rel far commute} and the second is by Lemma~\ref{l lambda1 eq lambda2 commute} (the proof of the lemma still works with the word  $\e{w'}$ present).
\end{proof}

For a diagram  $\theta$ contained in columns  $1,\ldots,l$ and a tuple  $\mathbf{n} = (n_1,n_2,\dots, n_l)$ with $n_1,\ldots,n_l \in \{\crc{0}\} \sqcup \A$,
let  $\Tab_\theta^\mathbf{n}$ denote the set of tableaux of shape  $\theta$ such that the entries in column  $c$ lie in $\A_{\le n_c}$.

For a weak composition or partition $\alpha=(\alpha_1,\dots,\alpha_l)$, define $\alpha_{l+1} = 0$.
\begin{theorem}\label{t flag Schur k3 technical}
Let  $\nu' \setminus \alpha'$ be a restricted shape and let  $l$ be the number of parts of  $\nu$.
Set $j = \min(\{i \mid \alpha_i > 0, \ \alpha_i \ge \alpha_{i+1}\} \cup \{l+1\} )$ and $j' = \max(\{i \mid \alpha_i < \nu_i\} \cup \{0\})$  (see Figure~\ref{f main proof} and the discussion following Remark \ref{r main theorem}).
Suppose
\begin{list}{\emph{(\roman{ctr})}}{\usecounter{ctr} \setlength{\itemsep}{2pt} \setlength{\topsep}{3pt}}
\item  $\alpha$ is of the form
\begin{align*}
&(0,\dots,0,1,2,\dots,a-1,a,a,a+1,a+2,\dots,\alpha_{j'}, \nu_{j'+1},\nu_{j'+2},\dots), \quad \text{ or } \\
&(0,\dots,0,1,2,\dots,\alpha_{j'}-1,\alpha_{j'}, \nu_{j'+1},\nu_{j'+2},\dots),
\end{align*}
where $\alpha_{j'} \ge \alpha_{j'+1}-1 = \nu_{j'+1}-1$,
the initial run of 0's may be empty, and
on the top line (resp. bottom line) $j$ is  $< j'$ and is the position of the first  $a$ (resp.  $j$ is  $j'$ or  $j'+1$);
\item $R$ is an RCT of shape  $\nu' \setminus \alpha'$ with entries $r_1 = R_{\alpha_1+1,1}, \ldots, r_{j'} = R_{\alpha_{j'}+1,j'}$ along its northern border;
\item $\e{vw}$ is an arrow respecting reading word of $R$ such that $\e{w}$ is a subsequence of $\e{r_{j+1} \cdots r_{j'}}$ which contains $\e{r_{j+1}}$ if $\e{r_{j+1}}$ is barred;
\item $\mathbf{n} =(n_1,\ldots, n_l)$, where $n_i \in \{\crc{0}\}\sqcup \A$ and  $\crc{0} \leq n_1 \leq n_2 \leq \cdots \leq n_l$;
\item $n_c = r_c \mydownarrow$ for  $c \in [j'] \setminus \{j\}$, and $n_j \le r_j \mydownarrow$;
\item it is not the case that  $w_1= r_{j+1} = a+1$ and  $n_j = a$ (for some  $a \in [N-1]$).
\end{list}
Then
\begin{equation}\label{e new statement}
\e{v}J_\alpha^{\mathbf{n}}(\Jnotb{j}{w})
 = \displaystyle \sum_{\substack{T \in \RCT_{\nu'},
\ T_{\nu' \setminus \alpha'} = R \\ T_{\alpha'} \in \Tab_{\alpha'}^{\mathbf{n}}} } \sqread(T)   \quad \, \text{ in \hspace{-.4mm} $\U/\Ikron$}.
\end{equation}
\end{theorem}

To parse this statement, it is  instructive to first understand the case when $j\ge j'$
(which implies $\e{w}$ is empty) and $\alpha_2 > 0$.  In this case the theorem becomes
\[
\begin{array}{c}
\parbox{14.6cm}{
Suppose $\nu' \setminus \alpha'$ is a restricted shape with $\alpha_1 = \alpha_2-1 = \cdots = \alpha_{j}-j+1$ and $\alpha_j \geq \alpha_{j+1} \geq \cdots \geq \alpha_{l} > \alpha_{l+1} = 0$.
Let  $R$ be an  RCT of shape  $\nu' \setminus \alpha'$ with entries $r_1, \ldots, r_{j'}$ along its northern border.
Suppose $\mathbf{n} = (n_1,\ldots, n_l)$ satisfies $\crc{0} \leq n_1 \leq n_2 \leq \cdots \leq n_l$,
$n_c = r_c\mydownarrow$ for  $c \in [j-1]$, and  $n_j \le r_j \mydownarrow$.
Then
}
\\ \\[-2mm]
\sqread(R) J_\alpha^{\mathbf{n}}
 = \displaystyle \sum_{\substack{T \in \RCT_{\nu'},
\ T_{\nu' \setminus \alpha'} = R \\ T_{\alpha'} \in \Tab_{\alpha'}^{\mathbf{n}}} } \sqread(T)  \quad \, \text{ in \hspace{-.4mm} $\U/\Ikron$}.
\end{array}
\]

\begin{remark}
\label{r main theorem}
\
\begin{list}{\rm{(\alph{ctr})}}{\usecounter{ctr} \setlength{\itemsep}{2pt} \setlength{\topsep}{3pt}}
\item Theorem~\ref{t J intro} is the special case of Theorem~\ref{t flag Schur k3 technical} when $\alpha = \nu$ and  $\mathbf{n} = (\crc{N}, \crc{N},\dots,\crc{N})$
(the  $\nu$ and  $\nu'$ of Theorem \ref{t J intro} must be interchanged to match the notation here).
\item
For any RCT  $R$ as in (ii), all arrow subtableaux of  $R$ have two rows and two columns, 
assuming  $n_{j-1} < n_j$ (this is an innocent assumption since $n_{j-1} = n_j$ implies both sides of \eqref{e new statement} are 0,
by the proof of Theorem \ref{t flag Schur k3 technical} below).  To see this, it suffices to show that there is no arrow between the boxes containing $r_{j-1}$ and  $r_{j+1}$.
Such an arrow would imply $r_{j-1} = \crc{a}, \, r_{j+1} = \crc{a+1}$,
which in turn would imply by (v)
that $n_{j-1} = \crc{a}, \, n_j \le  r_j \mydownarrow = (a+1) \mydownarrow = \crc{a}$, contradicting  $n_{j-1} < n_j$.
\item A reading word  $\e{vw}$ as in (iii) always exists---for instance take  $\e{vw} = \sqread(R)$ with  $\e{w}$ equal to the subsequence of barred letters of
$\e{r_{j+1} \cdots r_{j'}}$.
\item It follows from the theorem that  $\e{v}J_\alpha^{\mathbf{n}}(\Jnotb{j}{w}) \equiv 0$ if  $R$ cannot be completed to an RCT of shape $\nu'$, however we purposely do not make this assumption so that it does not have to be verified at the inductive step.
\end{list}
\end{remark}

The assumptions on $j$, $j'$, and  $\alpha$ look complicated, but their purpose is simply to peel off the entries of a tableau of shape  $\nu'$ one diagonal at a time, starting with the southwesternmost diagonal, and reading each diagonal in the $\nwarr$ direction (see Figure~\ref{f main proof}).
The proof goes by induction, peeling off one letter at a time from $J_\alpha^\mathbf{n}$, in the order just specified, and incorporating them into  $R$.  The index  $j$ indicates the column of the next letter to be removed.

The reader is encouraged to follow along the proof below with Example~\ref{ex inductive tree}.

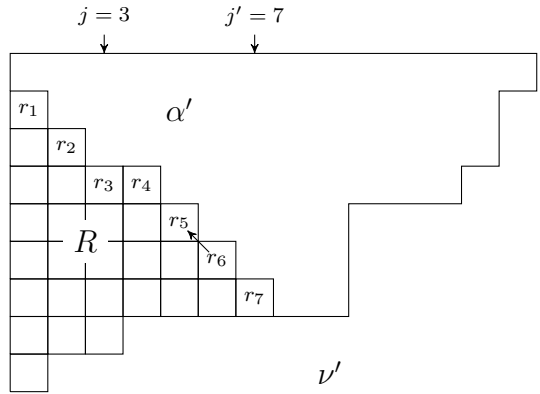
\begin{figure}
\centerfloat
\begin{tikzpicture}[xscale = 0.5,yscale = 0.5,>=stealth']
\tikzstyle{vertex}=[inner sep=3pt, outer sep=0pt]
\tikzstyle{aedge} = [draw, thin, ->,black]
\tikzstyle{edge} = [draw, thick, -,black]
\tikzstyle{dashededge} = [draw, very thick, dashed, black]
\tikzstyle{LabelStyleH} = [text=black, anchor=south, near start]
\tikzstyle{LabelStyleV} = [text=black, anchor=east, near start]

\drawcell{0}{-2}
\drawcell{0}{-1}
\drawcell{1}{-1}
\drawcell{2}{-1}
\drawcell{0}{0}
\drawcell{1}{0}
\drawcell{2}{0}
\drawcell{3}{0}
\drawcell{4}{0}
\drawcell{5}{0}
\drawcell{6}{0}
\drawcell{0}{1}
\drawcell{1}{1}
\drawcell{2}{1}
\drawcell{3}{1}
\drawcell{4}{1}
\drawcell{5}{1}
\drawcell{0}{2}
\drawcell{1}{2}
\drawcell{2}{2}
\drawcell{3}{2}
\drawcell{4}{2}
\drawcell{0}{3}
\drawcell{1}{3}
\drawcell{2}{3}
\drawcell{3}{3}
\drawcell{0}{4}
\drawcell{1}{4}
\drawcell{0}{5}
\draw (7,-0)--(9,0)--(9,3)--(12,3)--(12,4)--(13,4)--(13,6)--(14,6)--(14,7)--(0,7)--(0,6);
\draw[->] (2.5,7.5) node at (2.5,8) {\tiny$j=3$}--(2.5,7);
\draw[->] (6.5,7.5) node at (6.5,8) {\tiny$j'=7$}--(6.5,7);
\node[vertex] (a1) at (5.5,1.5) {};
\node[vertex] (a2) at (4.5,2.5) {};
\path[draw] (a1) edge[->,draw] (a2);
\node at (4.5,5.5) {$\alpha'$};
\node at (0.5,5.5) {\tiny$\noe{r_1}$};
\node at (1.5,4.5) {\tiny$\noe{r_2}$};
\node at (2.5,3.5) {\tiny$\noe{r_3}$};
\node at (3.5,3.5) {\tiny$\noe{r_4}$};
\node at (4.5,2.5) {\tiny$\noe{r_5}$};
\node at (5.5,1.5) {\tiny$\noe{r_6}$};
\node at (6.5,0.5) {\tiny$\noe{r_7}$};
\node at (8.5,-1.5) {$\nu'$};
\node[fill=white] at (2,2) {$R$};


\end{tikzpicture}
\vspace{.3in}
\caption{\label{f main proof} The setup of the proof of Theorem \ref{t flag Schur k3 technical} for
{\footnotesize$\nu = (9,8,8,7,7,7,7,7,7,4,4,4,3,1)$, $\alpha = (1,2,3,3,4,5,6,7,7,4,4,4,3,1)$}.
A possibility for the arrows of $R$ is shown. A possibility for $\e{w}$ is $\e{r_4r_5r_7}$.}
 \end{figure}

\begin{proof}[Proof of Theorem \ref{t flag Schur k3 technical}]
The proof is by induction on $|\alpha|$ and the $n_i$.
Throughout the proof,  $a$ denotes an element of  $[N-1]$, and  $a,a+1$ will often be regarded as unbarred letters and
 $\crc{a}, \crc{a+1}$ as the corresponding barred letters.

If  $|\alpha|=0$,  $J_\alpha^\mathbf{n}$ is a noncommutative version of the determinant of an upper unitriangular matrix, hence  $J_{\alpha}^{\mathbf{n}}(\Jnotb{j}{w}) = \e{w}$. The theorem then
states that  $\e{vw} \equiv \sqread(R)$, which holds by Theorem \ref{t arrow respecting connected}.

Next consider the case $n_1= \crc{0}$ and  $\alpha_1>0$.  This implies that  $J_\alpha^\mathbf{n}$ is a noncommutative version of
the determinant of a matrix whose first row is all 0's, hence $\e{v}J_{\alpha}^{\mathbf{n}}(\Jnotb{j}{w})=0$.
The right side of \eqref{e new statement} is also 0 because $\Tab_{\alpha'}^\mathbf{n}$ is empty for $n_1= \crc{0}$ and  $\alpha_1>0$.

If $n_i=n_{i+1}$ for any $i \ne j$ such that  $\alpha_i = \alpha_{i+1}-1$, then $\e{v}J_{\alpha}^{\mathbf{n}}(\Jnotb{j}{w})=0$ by \eqref{e elem sym J0}. To see that the right side of \eqref{e new statement} is also 0 in this case, observe that if  $T $ is an RCT from this sum, then
\[n_i < T_{\alpha_{i+1},i+1} \leq n_{i+1} = n_i,\]
hence the sum is empty.
Here, the inequality $n_i < T_{\alpha_{i+1},i+1}$ follows from Proposition \ref{e diagonal RCT} with  $x = n_i, \, y=r_i, \, z = T_{\alpha_{i+1}, i+1}$ (note that this argument works in the case  $\alpha_i = 0$).

We proceed to the main body of the proof.
By what has been said so far, we may assume  $|\alpha| > 0$, $\alpha_j > 0$, and $\crc{0} \le n_{j-1}<n_j$ if  $j > 1$ and  $\crc{0} < n_j$ if  $j =1$.
Since  $n_j \lecol r_j \lerow r_{j+1}$, if  $r_{j+1}$ is barred, then $n_j < r_{j+1} \myDownarrow$ by Proposition \ref{e diagonal RCT}.
If  $r_{j+1}$ is unbarred then either  ($r_{j+1} = a+1$,  $n_j = \crc{a}$) or ($n_j < r_{j+1}\myDownarrow$).
Hence exactly one of the following occurs: ($n_j < r_{j+1} \myDownarrow$) or ($r_{j+1}= a+1$, $n_j = \crc{a}$).
It follows that  exactly one of the following occurs: ($n_j < r_{j+1} \myDownarrow$ or $\e{w}$ is empty or  $w_1 \ne r_{j+1}$) or
($w_1 = r_{j+1}= a+1$, $n_j = \crc{a}$), which we consider as
 two separate cases in remainder of the proof.
Set $\mathbf{n}_\mydownarrow = (n_1, n_2, \cdots, n_{j-1}, n_j\mydownarrow,n_{j+1},\cdots,n_l)$,
$\mathbf{n}_\myDownarrow = (n_1, n_2, \cdots, n_{j-1}, n_j\myDownarrow,n_{j+1},\cdots,n_l)$,
 and  $\alpha_- = (\alpha_1,\dots, \alpha_{j-1}, \alpha_j-1, \alpha_{j+1},\dots)$.

\textbf{Case  $n_j < r_{j+1} \myDownarrow$ or $\e{w}$ is empty or  $w_1 \ne r_{j+1}$.}

There are three subcases depending on whether $j=1$,  ($j>1$ and $\alpha_j = 1$), or  ($j > 1$ and  $\alpha_j > 1$).  We argue only the last, as the first two are similar and easier.
A $j$-expansion yields
\begin{align}
\e{v}J_{\alpha}^{\mathbf{n}}(\Jnotb{j}{w})=~&
\e{v}J_{\alpha_-}(n_1,\dots, n_{j-1}   \Jnot{n_j}   n_j\mydownarrow \Jnot{w} n_{j+1},\dots)
+\e{v}J_{\alpha}^{\mathbf{n}_\myDownarrow}(\Jnotb{j}{w})\label{e case no square-1}\\
\equiv~&
\e{v}J_{\alpha_-}(n_1,\dots, n_{j-1} \Jnot{n_jw} n_j\mydownarrow,n_{j+1},\dots)
+\e{v}J_{\alpha}^{\mathbf{n}_\myDownarrow}(\Jnotb{j}{w}) \label{e case no square0} \\
\equiv~&
 \sum_{\substack{ T \in \RCT_{\nu'}, \ T_{\nu' \setminus \alpha'_-} = R\, \sqcup \,{\tiny\tableau{\noe{n_j}}}_{\alpha_j, j}\\ T_{\alpha'_-} \in \Tab_{\alpha'_-}^{\mathbf{n}_\mydownarrow}}} \sqread(T)
+  \sum_{\substack{T \in \RCT_{\nu'}, \ T_{\nu' \setminus \alpha'} = R \\ T_\alpha' \in \Tab_{\alpha'}^{\mathbf{n}_\myDownarrow}}} \sqread(T). \label{e case no square}
\end{align}
The first congruence is by Lemma \ref{l decreasing};
the conditions of the lemma are satisfied because
$n_j < r_{j+1} \myDownarrow$ if  $n_j$ is unbarred (by Proposition \ref{e diagonal RCT}),
and $w_1 \ge r_{j+1}$;
if  $n_j = \crc{a}$ is barred, then  $w_1 > a+1$ because we are not in the
case ($w_1 = r_{j+1}= a+1$, $n_j = \crc{a}$).
The second term of \eqref{e case no square0} is congruent (modulo $\Ikron$) to  the
second sum of \eqref{e case no square} by induction.
The conditions of the theorem are satisfied here:
(i)--(iii) are clear,  (iv) holds since  $n_{j-1} < n_j$,
(v) holds since $(n_\myDownarrow)_{j} = n_{j} \myDownarrow < n_j \le r_j\mydownarrow$,
and (vi) holds since we are not in the case $w_1 =r_{j+1}= a+1$, $n_j = \crc{a}$.

We next claim that the first term of \eqref{e case no square0} satisfies conditions (i)--(vi) of the theorem (with   $\alpha_-$ in place of  $\alpha$, $\mathbf{n}_\mydownarrow$ in place of  $\mathbf{n}$, $R \sqcup {\tiny\tableau{\noe{n_j}}}_{\alpha_j, j}$ in place of  $R$,  $\e{n_jw}$ in place of  
$\e{w}$,  $j-1$ in place of   $j$), hence the first term of \eqref{e case no square0} is congruent to the first sum of \eqref{e case no square} by induction.
Conditions (i) and (v) are clear, (iv) holds since  $n_{j-1} < n_j$, and (vi) follows from the fact that  $n_{j-1} = r_{j-1} \mydownarrow$ is barred.
Since $n_{j-1} < n_j$ and $n_{j-1} = r_{j-1}\mydownarrow$, it follows that $r_{j-1} \lerow n_j$, hence $R \sqcup {\tiny\tableau{\noe{n_j}}}_{\alpha_j, j}$ is an RCT; this verifies condition (ii).

Finally we check (iii), which requires showing that $\e{n_j w}$ is the end of an arrow respecting reading word of  $R \sqcup {\tiny\tableau{\noe{n_j}}}_{\alpha_j, j}$ satisfying an extra
condition.
There are two ways this can fail: either (I) $R \sqcup {\tiny\tableau{\noe{n_j}}}_{\alpha_j, j}$ has a $\searr$ arrow from  $n_j$ to  $r_{j+1}$ and  $w_1\ne r_{j+1}$, or (II) $R \sqcup {\tiny\tableau{\noe{n_j}}}_{\alpha_j, j}$ has a $\nwarr$ arrow from  $r_{j+1}$ to  $n_j$ and  $r_{j+1} = w_1$;
(I) implies  $r_{j+1}$ is barred by the definition of a  $\searr$ arrow, so by assumption (iii) of the theorem $w_1 = r_{j+1}$,  thus (I) cannot occur;
(II) implies  $n_j = a$, $r_{j+1} = a+1$ by the definition of a $\nwarr$ arrow, which combined with  $r_{j+1} = w_1$ contradicts assumption (vi) of the theorem, thus (II) cannot occur.
(By Remark \ref{r main theorem} (b),
we can assume $R \sqcup {\tiny\tableau{\noe{n_j}}}_{\alpha_j, j}$ does not have an arrow between $r_{j-2}$ and  $n_j$; in any case,
this is not an issue because such an arrow must go from  $r_{j-2}$ to  $n_j$ and
$r_{j-2}$ is a letter in  $\e{v}$.)

Now \eqref{e case no square} is equal to the right side of \eqref{e new statement}
because \eqref{e case no square} is simply the result of partitioning the set  $\{T \in \RCT_{\nu'} \mid T_{\nu' \setminus \alpha'} = R, \ T_{\alpha'} \in \Tab_{\alpha'}^{\mathbf{n}}\}$ into two, depending on whether or not  $T_{\alpha_j,j}$ is equal to or less than  ${n_j}$.  This completes the case  $n_j < r_{j+1} \myDownarrow$ or $\e{w}$ is empty or  $w_1 \ne r_{j+1}$.

\textbf{Case $w_1 =r_{j+1}= a+1$, $n_j = \crc{a}$.}

Corollary~\ref{c aaa+1 commute lam} yields the first congruence below;
the hypotheses of the corollary are satisfied since  $n_{j+1} = r_{j+1}\mydownarrow$ implies $n_{j+1}=\crc{a} = n_j$, and Proposition \ref{e diagonal RCT} implies
 $r_{j+2} \ge a+2$.
\begin{align}
\e{v}J_{\alpha}^{\mathbf{n}}(\Jnotb{j}{w}) \notag
&\equiv \e{v}J_{\alpha}(n_1,\dots,  n_{j-1}   \Jnot{w_1}   n_j \Jnot{w_2 \cdots w_t} n_{j+1},\dots) \\
&\equiv \e{vw_1} J_{\alpha}^{\mathbf{n}}(\Jnotb{j}{w_2 \cdots w_t}). \label{e case square}
\end{align}
Here, $t$ denotes the length of $\e{w}$.
The last congruence is by the far commutation relations if  $j > 1$ since $n_{j-1} < a$ by the facts that $n_{j-1} = r_{j-1} \mydownarrow$
is a barred letter and $n_{j-1} < n_j = \crc{a}$ (if $j=1$ there is nothing to prove).

Finally, observe that the case  ($n_j < r_{j+1} \myDownarrow$ or $\e{w}$ is empty or  $w_1 \ne r_{j+1}$) applies to  \eqref{e case square} (with  $\e{vw_1}$ in place of  $\e{v}$,  $\e{w_2} \cdots \e{w_t}$ in place of $\e{w}$).
Since the right side of \eqref{e new statement} depends only on $R$ and not directly on  $\e{v}$ and  $\e{w}$, \eqref{e case square} is congruent to this right side,  and this gives the desired statement in the present case  $w_1 =r_{j+1}= a+1$, $n_j = \crc{a}$.
\end{proof}

\begin{example}\label{ex inductive tree}
\setlength{\cellsize}{2.04ex}
Let  $\nu=(4,4,4,4)$.
Let $R={\tiny\tableau{2 \\ \crc{2} & 3 \\ 3 & \crc{3} & \crc{4}}}$; then $\sqread(R)=\e{3\crc{2}\crc{3}32\crc{4}}$.
We illustrate several steps of the inductive computation of $\e{3\crc{2}\crc{3}32\crc{4}} J_{(1,2,3,4)}^{\crc{1} \crc{2} \crc{4} 5}$ from the proof of Theorem~\ref{t flag Schur k3 technical}. After each step in which we add an entry to $R$, we record the new values of $R$, $j$, and $j'$.
We first apply the proof of the theorem to $\e{3\crc{2}\crc{3}32\crc{4}} J_{(1,2,3,4)}^{\crc{1} \crc{2} \crc{4} 5}$  ($\e{v}=\e{3\crc{2}\crc{3}32\crc{4}}$,  $\e{w}$ empty,  $R$ as above, $j = 4$,  $j'=3$) and expand as in \eqref{e case no square-1}:
\begin{align}
& \e{3\crc{2}\crc{3}32\crc{4}} J_{(1,2,3,4)}(\crc{1},\crc{2},\crc{4},5) \notag \\
=~& \e{3\crc{2}\crc{3}32\crc{4}} \Big( J_{(1,2,3,3)}(\crc{1},\crc{2},\crc{4}   \Jnot{5}   \crc{4}) + J_{(1,2,3,4)}(\crc{1},\crc{2},\crc{4},\crc{4}) \, \Big) \label{e ex inductive comp}
\end{align}
Next, apply the theorem to the first term  of \eqref{e ex inductive comp}
$\big( \e{v}= \e{3\crc{2}\crc{3}32\crc{4}}$,  $\e{w}=\e{5}$, $R={\tiny\tableau{2 \\ \crc{2} & 3 \\ 3 & \crc{3} & \crc{4} & 5}}$, $j=3$, $j'=4 \big)$.
This is handled by the case  $w_1 =r_{j+1}= a+1$, $n_j = \crc{a}$, hence the computation proceeds by applying \eqref{e case square} and then a  $3$-expansion:
\begin{align*}
& \e{3\crc{2}\crc{3}32\crc{4}} J_{(1,2,3,3)}(\crc{1},\crc{2}, \crc{4}  \Jnot{5}   \crc{4}) \\
\equiv ~& \e{3\crc{2}\crc{3}32\crc{4}5} J_{(1,2,3,3)}(\crc{1}, \crc{2}, \crc{4}, \crc{4})\\
=~& \e{3\crc{2}\crc{3}32\crc{4}5} \Big( J_{(1,2,2,3)}(\crc{1}, \crc{2}   \Jnot{\crc{4}}   \crc{4}, \crc{4}) + J_{(1,2,3,3)}(\crc{1}, \crc{2}, 4, \crc{4}) \, \Big).
\end{align*}
The first term is  $\equiv$ 0 by \eqref{e elem sym J0}. We expand the second term as in \eqref{e case no square-1} ($j=3$):
\begin{align*}
& \e{3\crc{2}\crc{3}32\crc{4}5} J_{(1,2,3,3)}(\crc{1}, \crc{2}, 4, \crc{4}) \\
=~& \e{3\crc{2}\crc{3}32\crc{4}5} \Big( J_{(1,2,2,3)}(\crc{1},\crc{2}  \Jnot{4}  \crc{3},\crc{4}) + J_{(1,2,3,3)}(\crc{1},\crc{2},\crc{3},\crc{4}) \, \Big).
\end{align*}

Next, we proceed with the inductive computation of the first term above $ \big( R={\tiny\tableau{2 \\ \crc{2} & 3 & 4 \\ 3 & \crc{3} & \crc{4} & 5}}$, $j=2$, $j'=4 \big)$
\begin{align*}
& \e{3\crc{2}\crc{3}32\crc{4}5} J_{(1,2,2,3)}(\crc{1},\crc{2}  \Jnot{4}  \crc{3},\crc{4}) \\
=~& \e{3\crc{2}\crc{3}32\crc{4}5} \Big( J_{(1,1,2,3)}(\crc{1}   \Jnot{\crc{2}}  \crc{2}   \Jnot{4}   \crc{3}, \crc{4}) +  J_{(1,2,2,3)}(\crc{1}, 2   \Jnot{4}   \crc{3}, \crc{4}) \, \Big).
\end{align*}
We continue the computation with the second term by applying a 2-expansion (as in \eqref{e case no square-1}) and then applying \eqref{e case no square0}:
\begin{align}
& \e{3\crc{2}\crc{3}32\crc{4}5}  J_{(1,2,2,3)}(\crc{1}, 2   \Jnot{4}   \crc{3}, \crc{4}) \notag \\
=~& \e{3\crc{2}\crc{3}32\crc{4}5} \Big( J_{(1,1,2,3)}(\crc{1}   \Jnot{2}   \crc{1}   \Jnot{4}   \crc{3}, \crc{4}) +
  J_{(1,2,2,3)}(\crc{1}, \crc{1}   \Jnot{4}   \crc{3}, \crc{4}) \, \Big) \notag \\
\equiv~& \e{3\crc{2}\crc{3}32\crc{4}5} J_{(1,1,2,3)}(\crc{1} \Jnot{24}   \crc{1} , \crc{3}, \crc{4}). \label{e Jcomp 1123}
\end{align}
Note that the second term of the second line is $\equiv 0$ by \eqref{e elem sym J0}.
The input data to the theorem for the term \eqref{e Jcomp 1123} is
$\e{v}= \e{3\crc{2}\crc{3}32\crc{4}5}$,  $\e{w} = \e{24}$, $R={\tiny\tableau{2 & 2\\ \crc{2} & 3 & 4\\ 3 & \crc{3} & \crc{4} & 5}}$, $j=1$, $j'=4$.
This is handled by the case  $w_1 =r_{j+1}= a+1$, $n_j = \crc{a}$, hence the computation proceeds by applying \eqref{e case square} and then a  $1$-expansion:
\begin{align}
& \e{3\crc{2}\crc{3}32\crc{4}5} J_{(1,1,2,3)}(\crc{1} \Jnot{24}   \crc{1} , \crc{3}, \crc{4}) \notag\\
\equiv~& \e{3\crc{2}\crc{3}32\crc{4}52} J_{(1,1,2,3)}(\crc{1} \Jnot{4}   \crc{1} , \crc{3}, \crc{4})\notag\\
=~& \e{3\crc{2}\crc{3}32\crc{4}52} \Big(\e{\crc{1}} J_{(0,1,2,3)}(\crc{1} \Jnot{4}   \crc{1} , \crc{3}, \crc{4})
+ J_{(1,1,2,3)}(1 \Jnot{4}   \crc{1} , \crc{3}, \crc{4})\Big)\notag\\
\equiv~& \e{3\crc{2}\crc{3}32\crc{4}52} \Big(\e{\crc{1}4} J_{(0,1,2,3)}(\crc{1}, \crc{1} , \crc{3}, \crc{4})
+ J_{(1,1,2,3)}(1 \Jnot{4}   \crc{1} , \crc{3}, \crc{4})\Big) \label{e example J comp}
\end{align}
The last congruence is by \eqref{e case no square0}.
The first term of \eqref{e example J comp} is $\equiv 0$ by \eqref{e elem sym J0}.
We continue with the inductive computation of the second term of \eqref{e example J comp} $\big( \e{v} = \e{3\crc{2}\crc{3}32\crc{4}52}$, $\e{w} = \e{4}$, $R={\tiny\tableau{2 & 2 \\ \crc{2} & 3 & 4 \\ 3 & \crc{3} & \crc{4} & 5}}$, $j=1$, $j'=4 \big)$:
\begin{align*}
~& \e{3\crc{2}\crc{3}32\crc{4}52} J_{(1,1,2,3)}(1 \Jnot{4}   \crc{1} , \crc{3}, \crc{4}) \\
=~&\e{3\crc{2}\crc{3}32\crc{4}52} \Big(\e{1} J_{(0,1,2,3)}(\crc{0} \Jnot{4}   \crc{1} , \crc{3}, \crc{4}) + J_{(1,1,2,3)}(\crc{0} \Jnot{4}   \crc{1} , \crc{3}, \crc{4}) \Big)\\
\equiv~&\e{3\crc{2}\crc{3}32\crc{4}52} \Big(\e{14} J_{(0,1,2,3)}(\crc{0}, \crc{1} , \crc{3}, \crc{4}) + J_{(1,1,2,3)}(\crc{0} \Jnot{4}   \crc{1} , \crc{3}, \crc{4}) \Big).
\end{align*}
The input data to the theorem for this first term is $\e{v}= \e{3\crc{2}\crc{3}32\crc{4}5214}$,  $\e{w}$ empty, $R={\tiny\tableau{1 \\ 2 & 2\\ \crc{2} & 3 & 4\\ 3 & \crc{3} & \crc{4} & 5}}$, $j= 4$, $j'=4$, $\alpha = (0, 1, 2, 3)$, $\mathbf{n} = (\crc{0}, \crc{1}, \crc{3}, \crc{4} )$.
\end{example}

\section{A strengthening of the main theorem and comparison to \cite{BLamLLT}}
\label{s generalizations}
We conjecture a strengthening of the main theorem and compare it to a similar conjecture in \cite{BLamLLT}.

\subsection{A conjectured strengthening}\label{ss a further conjectured strengthening}

We conjecture that Theorem \ref{t J intro} can be strengthened to hold in  $\U/\Ikronknuth$.
\begin{conjecture}\label{cj kronknuth}
For any partition  $\nu$,
\begin{align*}
\mathfrak{J}_{\nu}(\mathbf{u}) =
\sum_{ T \in \CT_{\nu}^< }\sqread(T)
 \qquad \text{in \hspace{-.7mm} $\U/\Ikronknuth$}.
\end{align*}
\end{conjecture}
In fact, we conjecture that Theorem \ref{t flag Schur k3 technical} holds exactly as stated  with  $\U/\Ikronknuth$ in place of  $\U/\Ikron$,
and we believe that the same inductive structure of the proof works in this setting, except some of the steps are much more difficult to justify.
Specifically, 
Lemmas~\ref{l decreasing} and \ref{l lambda1 eq lambda2 commute} are easily shown to hold with  $\U/\Ikronknuth$ in place of  $\U/\Ikron$,
whereas proving the correct analog of Corollary~\ref{c aaa+1 commute lam} for  $\U/\Ikronknuth$ seems to be the
main difficulty to overcome to adapt the proof.

The conjectured strengthening of Theorem \ref{t flag Schur k3 technical} to  $\U/\Ikronknuth$ was checked by computer in the following cases:
\begin{itemize}
\item $\nu = \alpha$ is a partition with at most 4 rows and at most 4 columns and size at most 10, and  $\crc{0} \le n_1 \le \cdots \le n_l \le \crc{3}$,
\item $\nu = \alpha$ is a partition with at most 6 rows and at most 6 columns and size at most 12, and  $\crc{0} \le n_1 \le \cdots \le n_l \le \crc{2}$.
\end{itemize}

\subsection{Comparison with results of \cite{BLamLLT}}
\label{ss comparison with results}
The statement and proof of Theorem \ref{t flag Schur k3 technical} have much in common with that of \cite[Theorem 4.8]{BLamLLT},
but there  are also important differences.
It is reasonable to expect that one result can be obtained from the other by a standardization argument,
but we believe this cannot be done.
It is quite possible that there is a more general monomial positivity result for noncommutative Schur functions
that yields both results as a special case.

Instead of comparing Theorem \ref{t flag Schur k3 technical} and \cite[Theorem 4.8]{BLamLLT}, we believe it
more natural to compare Conjecture \ref{cj kronknuth} with a conjecture in \cite{BLamLLT} similar to \cite[Theorem 4.8]{BLamLLT}.
We state this conjecture below, after introducing some notation.

Recall that $\U_\varnothing$ denotes the subalgebra of  $\U$ generated by  $u_a$ for  $a \in \A_\varnothing$.
The noncommutative Schur function  $\mathfrak{J}^\varnothing_\nu \in \U_\varnothing$ is defined in Definition~\ref{d normal noncom J}.


Let $\Irkstp{\le k}$ be the two-sided ideal of $\U_\varnothing$ corresponding to the relations
\begin{alignat}{2}
&\e{acb = cab} \qquad &&\text{for }\text{$c-a > k$ and $a < b < c$, $a, b, c \in \A_\varnothing$,} \label{e rel2} \\
&\e{bac = bca} \qquad &&\text{for }\text{$c-a > k$ and $a < b < c$, $a, b, c \in \A_\varnothing$,} \\
&\e{(ac-ca)b = b(ac-ca)} \qquad &&\text{for }\text{$c-a \le k$ and $a < b < c$, $a, b, c \in \A_\varnothing$,} \label{e rel4} \\
&{\e{v}} \qquad &&\text{for ordinary words $\e{v}$ with a repeated letter.} \label{e repeated letter}
\end{alignat}

\begin{conjecture}[{\cite[\textsection 5.2]{BLamLLT}}] \label{cj sqread}
For any partition  $\nu$,
\begin{align}
\mathfrak{J}^\varnothing_\nu  = \sum_{T \in \SYT'_\nu} \sqreadLLT(T) \quad \text{in }\, \U_\varnothing/\Irkstp{\le 3}. \label{e cj sqread}
\end{align}
\end{conjecture}

Here, $\SYT'_\nu$ denotes the set  of semistandard Young tableaux of shape  $\nu$ with entries in $\A_\varnothing$,
having no repeated letter.
To define $\sqreadLLT(T)$, first
draw arrows as shown between entries of $T$ for each of its $2\times 2$ subtableaux of the following forms:
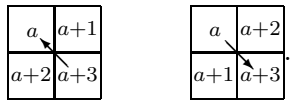
\setlength{\cellsize}{3.4ex}
\[{\tiny
\tableau{{\color{black}\put(.8,-1.6){\vector(-1,1){.6}}}\put(-.1,-0.96){$a$}&a\text{+}1\\a\text{+}2&a\text{+}3}\qquad \qquad
\tableau{{\color{black}\put(.25,-1.1){\vector(1,-1){.6}}}\put(-.1,-1){$a$}&a\text{+}2\\a\text{+}1&a\text{+}3}.}
\]
\setlength{\cellsize}{2.2ex}
The word $\sqreadLLT(T)$ is now defined as follows: let $D^1,D^2,\dots,D^t$ be the diagonals of $T$, starting from the southwest. Let $\e{v^i}$ be the result of reading, in the $\nwarr$ direction, the entries of $D^i$ that are $\nwarr$ arrow tails followed by the remaining entries of $D^i$, read in the $\searr$ direction. Set $\sqreadLLT(T)=\e{v^1}\e{v^2}\cdots \e{v^t}$.
For example,
\[
\sqreadLLT\bigg(\,
{ \footnotesize \tableau{
1&2&5&7\\3&4&6&8\\9\\}}
\, \bigg) = \e{934126587}.
\]

It is natural to wonder whether Conjecture \ref{cj kronknuth} can be obtained from Conjecture \ref{cj sqread}
by a standardization argument or vice versa.  We now argue informally that this is not possible in either direction.
We will freely use the notation from \textsection\ref{ss standardization}.
Also, by \emph{the standardization map} we will mean the $\ZZ$-linear map  $\U \to \U_\varnothing$ defined by  $\e{w} \mapsto \e{w}^{\stand^<}$.

Since the relations of  $\U/\Ikronknuth$ preserve colored content,
there is a  $\ZZ$-module decomposition
$\U / \Ikronknuth \cong \bigoplus_{\beta} (\U/ \Ikronknuth)_{\beta}$,
where  $(\U/ \Ikronknuth)_{\beta}$ denotes the  $\ZZ$-span of the colored words of colored content  $\beta$ in the algebra  $\U/ \Ikronknuth$.
Hence we can write
$\mathfrak{J}_{\nu}(\mathbf{u})$ uniquely as a sum
$\mathfrak{J}_{\nu}(\mathbf{u}) =
\sum_\beta (\mathfrak{J}_{\nu}(\mathbf{u}))_\beta,$ for
$(\mathfrak{J}_{\nu}(\mathbf{u}))_\beta \in (\U/ \Ikronknuth)_{\beta}$.
Conjecture \ref{cj sqread} then implies that for any colored content  $\beta$,
\begin{align}\label{ecj colored cont}
(\mathfrak{J}_{\nu}(\mathbf{u}))_{\beta} =
\sum_{\substack{ T \in \CT_{\nu}^< \\ T \text{ has colored content } \beta}} \sqread(T)
 \qquad \text{in \, $\U/\Ikronknuth$}.
\end{align}


We can consider $\mathfrak{J}^\varnothing_{\nu}$ as an element of  $\U_\varnothing$, written as the signed sum of words obtained by directly
expanding
its definition as a sum of products of  $e^\varnothing_k$'s.  Similarly,  $(\mathfrak{J}_{\nu}(\mathbf{u}))_{\beta}$ can be considered as an element of  $\U$ obtained by restricting
 the expression for $\mathfrak{J}_\nu(\mathbf{u}) \in \U$ to the colored words of colored content  $\beta$.
Let $\mathfrak{J}^\varnothing_{\nu}[\beta]$ denote the restriction of the signed sum $\mathfrak{J}^\varnothing_{\nu}$ to the permutations
$\S(\beta)$. 
It is true, then, that the image of  $(\mathfrak{J}_\nu(\mathbf{u}))_\beta$ under the standardization map
is $\mathfrak{J}^\varnothing_{\nu}[\beta]$.
However, as the examples below show,
the relations of  $\U/\Ikronknuth$ and those of  $\U_\varnothing/\Irkstp{\le 3}$ are not compatible in a way required
to relate the images of $(\mathfrak{J}_\nu(\mathbf{u}))_\beta$ and  $\mathfrak{J}^\varnothing_{\nu}[\beta]$ in these algebras.

\begin{example}
For $\nu = (2,2)$ and $\beta = ((1,1),(2,0))$, there holds
\[ (\mathfrak{J}_{\nu}(\mathbf{u}))_{\beta} = \e{\crc{1}21\crc{1}}  \qquad \text{in \, $\U/\Ikronknuth$}. \]
On the other hand,
\[
\mathfrak{J}^\varnothing_\nu[\beta]  = \e{3412 + 4132 - 4312} \quad \text{in }\, \U_\varnothing/\Irkstp{\le 3}
\]
and this quantity is not equal to a positive sum of monomials in  $\U_\varnothing/\Irkstp{\le 3}$.
Thus applying the standardization map to both sides of the identity \eqref{ecj colored cont}
cannot be used to deduce the identity
\eqref{e cj sqread}, at least not without writing  $\mathfrak{J}^\varnothing_\nu$ as a sum over $\mathfrak{J}^\varnothing_\nu[\beta]$ for some carefully selected colored contents  $\beta$.
And even if these carefully selected  $\mathfrak{J}^\varnothing_\nu[\beta]$ are monomial positive, this may not be enough
to prove \eqref{e cj sqread} from \eqref{ecj colored cont}
since elements of  $\Ikronknuth$ often standardize to elements which do not belong to $\Irkstp{\le 3}$; for example, $\e{\crc{2}1\crc{2}\crc{1} - \crc{2}\crc{2}1\crc{1}} \in \Ikronknuth$
but  $\e{4132 - 4312} \notin \Irkstp{\le 3}$.
\end{example}

\begin{example}
We now argue that inverse standardization cannot be used to deduce \eqref{ecj colored cont} from \eqref{e cj sqread}.
Let $\nu = (2,2,1)$ and $\beta = ((1,1),(3,0))$. There holds
\begin{align}
(\mathfrak{J}_{\nu}(\mathbf{u}))_{\beta} &= \e{\crc{1}\crc{1}21\crc{1}}  \quad \text{in } \, \U/\Ikronknuth, \label{eex 1}\\
\mathfrak{J}^\varnothing_\nu[\beta]  &= \e{43512} \quad \text{in }\, \U_\varnothing/\Irkstp{\le 3}.  \label{eex 2}
\end{align}
Even though  $(\e{\crc{1}\crc{1}21\crc{1}})^{\stand^<} = \e{43512}$, we cannot deduce \eqref{eex 1} (directly) from \eqref{eex 2}
because the element
\[\e{43215 - 43251} \in \Irkstp{\le 3},\]
but its inverse image under the standardization map, restricted to colored content  $\beta$, is 
\[\e{\crc{1}\crc{1}\crc{1}12 - \crc{1}\crc{1}\crc{1}21} \notin \Ikronknuth.\]
\end{example}

\section{Commuting super elementary symmetric functions}
\label{s commuting super elementary symmetric functions}
We adapt the theory of
noncommutative Schur functions to the super setting and thereby prove Theorems \ref{t intro basics} and \ref{t basics 2}.

\subsection{Elementary functions commute in $\U/\Ikronknuth$ and  $\U/\Iplacord{\lessdot}$}

\begin{proposition}\label{p es commute kron}
The noncommutative super elementary symmetric functions commute in $\U/\Ikronknuth$, i.e.,
 $e_k(\mathbf{u})e_l(\mathbf{u}) = e_l(\mathbf{u})e_k(\mathbf{u})$ in  $\U/\Ikronknuth$ for all  $k, l$.
\end{proposition}
\begin{proof}
Let $s, t$ be formal variables that commute with each other and with all
the~$u_z$, $z \in \A$.
Throughout the proof, we write $f \equiv g$ to mean that $f$ and $g$ are congruent modulo $\Ikronknuth[[s,t]]$, for any $f,g \in \U[[s,t]]$.
We also let $[f,g]$ denote the commutator $fg - gf$ for any  $f,g \in \U[[s,t]]$.

Define
\[
f_z =
\begin{cases}
1+u_zs & \text{ if  $z$ is unbarred}, \\
(1-u_zs)^{-1}  & \text{ if  $z$ is barred}.
\end{cases}
\qquad
g_z =
\begin{cases}
1+u_zt & \text{ if  $z$ is unbarred}, \\
(1-u_zt)^{-1} & \text{ if  $z$ is barred}.
\end{cases}
\]

We require the following key lemma.
\begin{lemma}
\label{l ab 3vars trick c}
For any $x,y,z \in \A$ such that $x < y < z$, $\e{y} = u_y$ commutes with $g_z^{-1}f_xg_zf_x^{-1}$ in the algebra $(\U/\Ikronknuth)[[s,t]]$.
\end{lemma}
\begin{proof}[Proof of Lemma \ref{l ab 3vars trick c}]

There are four cases depending on whether or not $x$ and $z$ are barred or unbarred. We first handle the case that $x$ and $z$ are unbarred.
First note that $\e{y}$ commutes with $g_z^{-1}f_xg_zf_x^{-1}$ in $(\U/\Ikronknuth)[[s,t]]$ if and only if $\e{y}$ commutes with $g_z^{-1}[f_x,g_z]f_x^{-1}$ in $(\U/\Ikronknuth)[[s,t]]$.
Further, $[f_x, g_z] = [\e{x}, \e{z}]st$.
By relations \eqref{e plac nat rel knuth3} and \eqref{e plac nat rel knuth3b}, $\e{z}[\e{x}, \e{z}] \equiv 0$ and $[\e{x}, \e{z}]\e{x} \equiv 0$.
Hence
\[g_z^{-1}[f_x,g_z]f_x^{-1} \equiv [\e{x},\e{z}]st.\]
Thus it suffices  to show that $\e{y}[\e{x},\e{z}] \equiv \e{[x,z]y}$; this holds by \eqref{e kron rel rotate12} if $x = z \myDownarrow \myDownarrow $ and by \eqref{e kronknuth rel knuth1} and \eqref{e kronknuth rel knuth2} if $x < z \myDownarrow \myDownarrow$.

If $x$ and $z$ are barred, then we instead show that $\e{y}$ commutes with $f_xg_z^{-1}f_x^{-1}g_z$ in $(\U/\Ikronknuth)[[s,t]]$. This is equivalent to showing that $\e{y}$ commutes with $f_x[g_z^{-1}, f_x^{-1}]g_z$ in $(\U/\Ikronknuth)[[s,t]]$.
There holds $[g_z^{-1}, f_x^{-1}] = [\e{z}, \e{x}]st$.
By relations \eqref{e plac nat rel knuth4} and \eqref{e plac nat rel knuth4b}, $\e{x}[\e{z}, \e{x}] \equiv 0$ and $[\e{z}, \e{x}]\e{z} \equiv 0$.
Hence
\[f_x[g_z^{-1}, f_x^{-1}]g_z \equiv [\e{z},\e{x}]st.\]
Thus it suffices  to show that $\e{y}[\e{z},\e{x}] \equiv \e{[z,x]y}$; this holds by \eqref{e kron rel rotate12} if $x = z \myDownarrow \myDownarrow $ and by \eqref{e kronknuth rel knuth1} and \eqref{e kronknuth rel knuth2} if $x < z \myDownarrow \myDownarrow$.

We now handle the case $x$ is barred and $z$ is unbarred. We compute
\[[f_x,g_z]f_x^{-1} = \sum_{d=1}^\infty [\e{x}^ds^d,\e{z}t](1-\e{x}s) = [\e{x},\e{z}]st+\sum_{d=2}^\infty ([\e{x}^d,\e{z}] - [\e{x}^{d-1},\e{z}]\e{x})s^dt \equiv [\e{x},\e{z}]st,\]
where the congruence follows from \eqref{e plac nat rel knuth4}.
Combining this with the fact $\e{z}[\e{x},\e{z}] \equiv 0$ yields
\begin{align}
g_z^{-1}[f_x,g_z]f_x^{-1} \equiv g_z^{-1}[\e{x},\e{z}]st \equiv [\e{x},\e{z}]st. \label{e zxzx}
\end{align}
Since $\e{y}[\e{x},\e{z}] \equiv \e{[x,z]y}$ by the relations \eqref{e kronknuth rel knuth1} and \eqref{e kronknuth rel knuth2}, this case of the lemma is proved.
The case $x$ is unbarred and $z$ is barred is similar.
\end{proof}

We continue with the main thread of the proof of Proposition \ref{p es commute kron}.
It is clear that $f_xg_x = g_xf_x$ for all $x \in \A$. We next claim that for any  $x,y \in \A$,  $x < y$,
\begin{align}
f_yf_xg_yg_x \equiv g_yg_xf_yf_x. \label{e ayaxbybx}
\end{align}
Note that this is equivalent to
\begin{align}
g_y^{-1}f_xg_yf_x^{-1} \equiv f_y^{-1}g_xf_yg_x^{-1}. \label{e ayaxbybx 2}
\end{align}
Further, the congruence \eqref{e ayaxbybx 2} is equivalent to $g_y^{-1}[f_x,g_y]f_x^{-1} \equiv f_y^{-1}[g_x,f_y]g_x^{-1}$. This follows from computations similar to but easier than those used to prove Lemma \ref{l ab 3vars trick c}.
For example, in the case $x$ is barred and $y$ is unbarred, \eqref{e zxzx} yields
\[g_y^{-1}[f_x,g_y]f_x^{-1} \equiv [\e{x},\e{y}]st \equiv f_y^{-1}[g_x,f_y]g_x^{-1}.\]

Now let $z_1 < z_2 < \cdots < z_{2N}$ denote the elements of $\A$ in natural order, and set $F_{j,i} = f_{z_j}f_{z_{j-1}}\cdots f_{z_i}$ and $G_{j,i} = g_{z_j}g_{z_{j-1}}\cdots g_{z_i}$.
The proposition is equivalent to the statement $F_{2N,1}G_{2N,1} \equiv G_{2N,1}F_{2N,1}$.
We will prove that $F_{j,i}G_{j,i} \equiv G_{j,i}F_{j,i}$ by induction on $j-i$.
By the previous paragraph, this holds for the base cases $j-i=0$ and $j-i=1$.
Now, for $j-i>1$, we obtain
\begin{align*}
 & F_{j,i}G_{j,i} \\
=~& F_{j,i+1}g_{z_j}g_{z_j}^{-1}f_{z_i}g_{z_j}f_{z_i}^{-1}f_{z_i}G_{j-1,i}\\
=~& F_{j,i+1}G_{j,i+1}G_{j-1,i+1}^{-1}\big( g_{z_j}^{-1}f_{z_i}g_{z_j}f_{z_i}^{-1}\big)F_{j-1,i+1}^{-1}F_{j-1,i}G_{j-1,i}\\
\equiv~& G_{j,i+1}F_{j,i+1}G_{j-1,i+1}^{-1}\big( g_{z_j}^{-1}f_{z_i}g_{z_j}f_{z_i}^{-1}\big)F_{j-1,i+1}^{-1}G_{j-1,i}F_{j-1,i} &&\text{by induction}\\
\equiv~& G_{j,i+1}F_{j,i+1}F_{j-1,i+1}^{-1}\big( g_{z_j}^{-1}f_{z_i}g_{z_j}f_{z_i}^{-1}\big)G_{j-1,i+1}^{-1}G_{j-1,i}F_{j-1,i} &&\text{by Lemma \ref{l ab 3vars trick c} and induction}\\
=~& G_{j,i+1}f_{z_j}\big( g_{z_j}^{-1}f_{z_i}g_{z_j}f_{z_i}^{-1}\big)g_{z_i}F_{j-1,i} &&\text{}\\
\equiv~& G_{j,i+1}f_{z_j}\big( f_{z_j}^{-1}g_{z_i}f_{z_j}g_{z_i}^{-1}\big)g_{z_i}F_{j-1,i} &&\text{by \eqref{e ayaxbybx 2}}\\
=~& G_{j,i}F_{j,i},
\end{align*}
as desired.
\end{proof}

\begin{proposition}\label{p es commute plac}
In the algebra $\U/\Iplacord{\lessdot}$, there holds
$e^\lessdot_k(\mathbf{u})e^\lessdot_l(\mathbf{u}) = e^\lessdot_l(\mathbf{u})e^\lessdot_k(\mathbf{u})$ for all  $k, l$.
\end{proposition}
\begin{proof}
This follows from a similar result for the ordinary plactic algebra \cite{LS} and a standardization argument using Proposition \ref{p plac stand}.
\end{proof}

\subsection{Noncommutative super Schur function basics}

Here we adapt the basic setup of \cite{FG, BF} to the super setting.
Throughout this section, fix a shuffle order $\lessdot$ on $\A$.

Let $x_1, x_2, \ldots$ be formal variables that commute with each other and with all
the~$u_z$, $z \in \A$. Define
\[
c_z(x_j) =
\begin{cases}
(1-u_zx_j)^{-1} & \text{ if  $z$ is unbarred}, \\
1+u_zx_j  & \text{ if  $z$ is barred}.
\end{cases}
\]
Define the ``noncommutative Cauchy product''
\begin{equation}
\label{eq:cauchy-product}
\Omega^\lessdot(\mathbf{x}, \mathbf{u}) =
\prod_{j=1}^\infty \bigl(
c_{z_1}(x_j) c_{z_2}(x_j) \cdots c_{z_{2N}}(x_j)
\bigr),
\end{equation}
where $z_1 \lessdot z_2 \lessdot \cdots \lessdot z_{2N}$ are the elements of $\A$.


The noncommutative Cauchy product $\Omega^\lessdot(\mathbf{x},\mathbf{u})$ is naturally expressed in terms of the \emph{noncommutative super homogeneous symmetric functions}, which are given by
\[
h^\lessdot_k(\mathbf{u})=\sum_{\substack{z_1 \ledotrow z_2 \ledotrow \cdots \ledotrow z_k \\ z_1,\dots,z_k \in \A}} u_{z_1} u_{z_2} \cdots u_{z_k} \ \in \U
\]
for any positive integer $k$; set $h^\lessdot_0(\mathbf{u})=1$ and $h^\lessdot_k(\mathbf{u}) = 0$ for $k<0$.
Here, for $ y,z \in \A$, the notation $y \ledotrow z$ means that either $y \lessdot z$ or $y$ and $z$ are equal unbarred letters.

Recall that the noncommutative super Schur functions for the order $\lessdot$ are denoted $\mathfrak{J}^\lessdot_\nu(\mathbf{u})$
(Definition~\ref{d noncommutative Schur order}). Also recall that $\langle\cdot ,\cdot  \rangle$ denotes the symmetric bilinear form on $\U$ for which the monomials (colored words) form an orthonormal basis.
We have the following generalization of the setup of \cite{FG, BF} to the super setting.

\begin{theorem} \label{t basics full}
Let  $I$ be   a two-sided ideal of \hspace{-.7mm} $\U$.
\begin{list}{\emph{(\roman{ctr})}}{\usecounter{ctr} \setlength{\itemsep}{3pt} \setlength{\topsep}{3pt}}
\item The noncommutative Cauchy product can be written in the following two ways (in \hspace{-.7mm} $\U$)
\begin{align}\label{et basics full}
\Omega^\lessdot(\mathbf{x}, \mathbf{u}) = \sum_{\text{colored words $\e{w}$}} Q_{\Des_\lessdot(\e{w})}(\mathbf{x})\e{w}
  =  \sum_{\text{weak compositions $\alpha$}} \mathbf{x}^\alpha h^\lessdot_{\alpha}(\mathbf{u}).
\end{align}
Hence for any $\gamma \in \U$, $F^\lessdot_\gamma(\mathbf{x}) = \big\langle \Omega^\lessdot(\mathbf{x}, \mathbf{u}) , \gamma \big\rangle.$
Here, $h^\lessdot_\alpha(\mathbf{u}) :=h^\lessdot_{\alpha_1}(\mathbf{u})h^\lessdot_{\alpha_2}(\mathbf{u}) \cdots $.
\item There holds $e^\lessdot_k(\mathbf{u})e^\lessdot_l(\mathbf{u})=e^\lessdot_l(\mathbf{u})e^\lessdot_k(\mathbf{u})$ in \hspace{-.7mm} $\U/I$ for all $k$, $l$ if and only if \newline
    \mbox{$h^\lessdot_k(\mathbf{u})h^\lessdot_l(\mathbf{u})=h^\lessdot_l(\mathbf{u})h^\lessdot_k(\mathbf{u})$}
    in \hspace{-.7mm} $\U/I$ for all $k$, $l$.
\item If $h^\lessdot_k(\mathbf{u})h^\lessdot_l(\mathbf{u})=h^\lessdot_l(\mathbf{u})h^\lessdot_k(\mathbf{u})$ in \hspace{-.7mm} $\U/I$ for all $k$, $l$, then
\begin{align} \label{e mh sJ}
 \Omega^\lessdot(\mathbf{x}, \mathbf{u}) = \sum_{\nu} m_{\nu}(\mathbf{x}) h^\lessdot_{\nu}(\mathbf{u}) = \sum_{\nu} s_\nu(\mathbf x)\mathfrak{J}^\lessdot_{\nu}(\mathbf{u})\quad \text{in \hspace{-.7mm} $\U/I$},
\end{align}
where the sums are over all partitions $\nu$.
Hence for any $\gamma\in I^\perp$,
\[F^\lessdot_\gamma(\mathbf{x})
= \big\langle \Omega^\lessdot(\mathbf{x}, \mathbf{u}), \gamma \big\rangle = \sum_{\nu} s_\nu(\mathbf x) \langle \mathfrak{J}^\lessdot_{\nu}(\mathbf{u}), \gamma \rangle. \]
\end{list}
\end{theorem}

\begin{proof}
Formula \eqref{et basics full} is seen by fixing a colored word  $\e{w}$ and collecting all monomials in the $x_i$ that appear with it in  the noncommutative Cauchy product,
or by fixing a monomial  $\mathbf{x}^\alpha$ and collecting all colored words that appear with it in the noncommutative Cauchy product.
The second part of (i) is  immediate from \eqref{et basics full} and Definition~\ref{d F gamma}.

We next prove (ii).
Let $f_z, g_z$ be as in the proof of Proposition \ref{p es commute kron}, and let \mbox{$z_1 \lessdot z_2 \lessdot \cdots \lessdot z_{2N}$} be the elements of $\A$.
Then the $e^\lessdot_k(\mathbf{u})$ commuting is equivalent to  $f_{z_{2N}}f_{z_{2N-1}}\cdots f_{z_1}$ commuting with  $g_{z_{2N}}g_{z_{2N-1}}\cdots g_{z_1}$
and the $h^\lessdot_k(\mathbf{u})$ commuting is equivalent to  $f_{z_1}^{-1} f_{z_2}^{-1} \cdots f_{z_{2N}}^{-1}$ commuting with
$g_{z_1}^{-1} g_{z_2}^{-1} \cdots g_{z_{2N}}^{-1}$ (in the algebra $(\U/I)[[s,t]]$).
Statement (ii) follows.


For (iii), the first equality of \eqref{e mh sJ} is immediate from (i) and the hypothesis that the $h^\lessdot_k(\mathbf{u})$ commute.
For the second equality of \eqref{e mh sJ}, simply note that when the $h^\lessdot_k(\mathbf{u})$ commute,
the subalgebra they generate is the surjective image of the ring of symmetric functions in commuting variables and hence all the usual identities hold.
\end{proof}

\section*{Acknowledgments}

We thank Sergey Fomin for helpful conversations and Elaine So and Emily Walters for help typing and typesetting figures.

\bibliographystyle{plain}
\bibliography{mycitations}

\end{document}